\documentclass{article}

\usepackage{amsfonts}
\usepackage{amsmath}
\usepackage{mathtools}
\usepackage{multirow}
\usepackage{epstopdf}
\usepackage{algorithm}
\usepackage{algorithmic}
\usepackage{bm}
\usepackage{microtype}
\usepackage{amssymb}
\usepackage{enumerate}

\newtheorem{theorem}{Theorem}

\newtheorem{remark}{Remark}
\newtheorem{lemma}{Lemma}
\newtheorem{proof}{Proof}
\newtheorem{corollary}{Corollary}

\newtheorem{assumption}{Assumption}

\newcommand{\x}{\mathbf{x}}
\newcommand{\y}{\mathbf{y}}
\newcommand{\z}{\mathbf{z}}

\renewcommand{\v}{\mathbf{v}}

\newcommand{\m}{\mathbf{m}}
\newcommand{\bdelta}{\mathbf{\delta}}
\newcommand{\blambda}{\mathbf{\lambda}}
\newcommand{\bell}{\mathbf{\ell}}

\newcommand{\I}{\mathbf{I}}
\newcommand{\X}{\mathbf{X}}
\newcommand{\Y}{\mathbf{Y}}
\newcommand{\U}{\mathbf{U}}
\newcommand{\V}{\mathbf{V}}
\newcommand{\A}{\mathbf{A}}
\newcommand{\bLambda}{\mathbf{\Lambda}}
\renewcommand{\H}{\mathbf{H}}
\renewcommand{\S}{\mathcal{S}}
\newcommand{\bO}{{\cal O}}
\newcommand{\R}{\mathbb{R}}
\newcommand{\K}{\mathcal{K}}

\newcommand{\B}{\mathbb{B}}

\newcommand{\<}{\left\langle}
\renewcommand{\>}{\right\rangle}

\DeclareMathOperator*{\argmin}{argmin}

\usepackage[accepted]{icml2020}

\icmltitlerunning{Restarted Nonconvex Accelerated Gradient Descent}

\begin{document}

\onecolumn
\icmltitle{Restarted Nonconvex Accelerated Gradient Descent: \\
No More Polylogarithmic Factor in the $\bO(\epsilon^{-7/4})$ Complexity}

\begin{icmlauthorlist}
\icmlauthor{Huan Li}{to}
\icmlauthor{Zhouchen Lin}{goo}
\end{icmlauthorlist}

\icmlaffiliation{to}{Institute of Robotics and Automatic Information Systems, College of Artificial Intelligence, Nankai University, Tianjin, China (lihuanss@nankai.edu.cn).\\}
\icmlaffiliation{goo}{National Key Lab of General AI, School of Intelligence Science and Technology, Peking University, Beijing, China
    (zlin@pku.edu.cn).}





\vskip 0.3in



\printAffiliationsAndNotice{}  

\begin{abstract}
This paper studies accelerated gradient methods for nonconvex optimization with Lipschitz continuous gradient and Hessian. We propose two simple accelerated gradient methods, restarted accelerated gradient descent (AGD) and restarted heavy ball (HB) method, and establish that our methods achieve an $\epsilon$-approximate first-order stationary point within $\bO(\epsilon^{-7/4})$ number of gradient evaluations by elementary proofs. Theoretically, our complexity does not hide any polylogarithmic factors, and thus it improves over the best known one by the $\bO(\log\frac{1}{\epsilon})$ factor. Our algorithms are simple in the sense that they only consist of Nesterov's classical AGD or Polyak's HB iterations, as well as a restart mechanism. They do not invoke negative curvature exploitation or minimization of regularized surrogate functions as the subroutines. In contrast with existing analysis, our elementary proofs use less advanced techniques and do not invoke the analysis of strongly convex AGD or HB.
\end{abstract}

\section{Introduction}
Nonconvex optimization has become the foundation of training machine learning models and emerging machine learning tasks can be modeled as nonconvex problems. Typical examples include matrix completion \citep{Hardt-2014-focs}, one bit matrix completion \citep{Davenport-2014}, robust PCA \citep{Prateek-2014-pca}, phase retrieval \citep{Candes-2015}, and deep learning \citep{Hinton-nature}. In this paper, we consider the following general nonconvex problem:
\begin{eqnarray}
\begin{aligned}\label{problem}
\min_{\x\in\R^d} f(\x),
\end{aligned}
\end{eqnarray}
where $f(\x)$ has Lipschitz continuous gradient and Hessian and it is bounded from below. Our goal is to find an $\epsilon$-approximate first-order stationary point, defined as
\begin{equation}
\|\nabla f(\x)\|\leq\epsilon.\notag
\end{equation}

Gradient descent, a fundamental algorithm in machine learning, is commonly used due to its simplicity and practical efficiency. Theoretically, gradient descent is the optimal method among the first-order algorithms for nonconvex optimization under the assumption that the gradient is Lipschitz \citep{Carmon-mp-2020}, which means that we cannot find a first-order method with theoretically faster convergence rate under these conditions. When we assume additional structure, such as the Hessian Lipschitz geometry, improvement is possible. On the other hand, for convex optimization, gradient descent is known to be suboptimal and several accelerated gradient methods with theoretically faster convergence rate were proposed. Typical examples include Polyak's heavy ball (HB) method \citep{Polyak-1964} and Nesterov's accelerated gradient descent (AGD) \citep{Nesterov1983,Nesterov1988,Nesterov-smooth}. Motivated by the theoretical optimality and practical efficiency of convex AGD and HB, AGD and HB have been extended to nonconvex optimization \citep{Carmon-2016,carmon-2017-guilty,Zhu-16-apg,jinchi-18-apg}. But there are still some issues, such as the suboptimal convergence rate and complex algorithms and proofs. In this paper, we study the restarted AGD and HB method, variants of the original AGD and HB by employing a restart mechanism. Our aim is to establish a slightly faster convergence rate than the state-of-the-art accelerated methods by elementary analysis for the two simple methods.

\subsection{Literature Review}

In this section, we briefly review the convergence rates of gradient descent, accelerated gradient descent, and the heavy ball method for convex optimization, as well as the state-of-the-art accelerated methods for nonconvex optimization.

\subsubsection{Accelerated Gradient Methods for Convex Optimization}

For convex problems, gradient descent is known to converge to an $\epsilon$-optimal solution within $\bO(\frac{L}{\epsilon})$ and $\bO(\frac{L}{\mu}\log\frac{1}{\epsilon})$ iterations for $L$-smooth convex problems and $\mu$-strongly convex problems, respectively \citep{Nesterov-2004}. Polyak's heavy ball method \citep{Polyak-1964} was the first accelerated first-order method, which finds an $\epsilon$-optimal solution in $\bO(\sqrt{\frac{L}{\mu}}\log\frac{1}{\epsilon})$ steps when the objective function is twice continuously differentiable, $L$-smooth, $\mu$-strongly convex, and the initializer is close enough to the minimum. Recently, \citet{wang-icml22} extended HB to the case without the locality condition. However, the $\bO(\sqrt{\frac{L}{\mu}}\log\frac{1}{\epsilon})$ complexity only holds after $\bO(\frac{L}{\mu})$ iterations. When strong convexity is absent, currently, only the $\bO(\frac{L}{\epsilon})$ complexity is proved for smooth convex problems \citep{Ghadimi-2015-ecc}, which is the same as gradient descent. In a series of celebrated works \citep{Nesterov1983,Nesterov1988,Nesterov-smooth}, Nesterov proposed several accelerated gradient descent methods. The same $\bO(\sqrt{\frac{L}{\mu}}\log\frac{1}{\epsilon})$ complexity is established for strongly convex problems without the twice continuous differentiability and locality assumptions. Moreover, when the objective is $L$-smooth and convex, Nesterov's accelerated methods find an $\epsilon$-optimal solution in $\bO(\sqrt{\frac{L}{\epsilon}})$ iterations, which is faster than gradient descent and the heavy ball method in theory. Nesterov's accelerated methods are proven to be optimal among the first-order methods for convex optimization \citep{Nesterov-2004}. For more topics on accelerated methods for convex optimization, interested readers can refer to the survey paper \citep{li-pieee}, for example.

\subsubsection{Accelerated Gradient Methods to Achieve Nonconvex First-order Stationary Point}

For nonconvex problems, gradient descent finds an $\epsilon$-approximate first-order stationary point of problem (\ref{problem}) in $\bO(\epsilon^{-2})$ iterations \citep{Nesterov-2004}. Enormous amount of effort has been spent on speeding up gradient descent in the last decade. \citet{hb-1993,iPiano-2014,iPiano-2018,JingweiLiang-16-nips} studied the convergence of the HB method, while \citet{Saecd2013,li-15-apg,li-2017-icml} studied AGD. The practical efficiency is verified empirically and there is no theoretical speedup under the assumption of Lipschitz gradient. With the additional Lipschitz Hessian assumption, \citet{carmon-2017-guilty} proposed a ``convex until proven guilty" mechanism with nested-loop, which converges to an $\epsilon$-approximate first-order stationary point within $\bO(\epsilon^{-7/4}\log\frac{1}{\epsilon})$ gradient and function evaluations. Their method alternates between negative curvature exploitation and inexact minimization of a regularized surrogate function, where in the latter subroutine, \citet{carmon-2017-guilty} add a proximal term to reduce the nonconvex subproblem to a convex one and use the convex AGD to minimize it until the function is ``guilty" of being nonconvex. When the third-order derivative of the objective is Lipschitz, the $\bO(\epsilon^{-5/3}\log\frac{1}{\epsilon})$ complexity can be obtained \citep{carmon-2017-guilty}.

\subsubsection{Accelerated Gradient Methods to Achieve Nonconvex Second-order Stationary Point}

When studying nonconvex accelerated methods, most works concentrate on the second-order stationary point (see definition in (\ref{def_2point})). \citet{Carmon-2016} combined the Lanczos method and regularized accelerated gradient descent, where the former is used to compute the eigenvector corresponding to the smallest negative eigenvalue to search descent directions of negative curvature. \citet{Zhu-16-apg} implemented the cubic regularized Newton method \citep{nesterov-cubic} carefully and computed the descent direction using accelerated method for fast approximate matrix inversion, while \citet{Carmon-siam-2020,Carmon-2018-nips} employed the Krylov subspace method to solve the cubic regularized Newton subproblems. The above methods find an $\epsilon$-approximate second-order stationary point with probability at least $1-\delta$ in $\bO(\epsilon^{-7/4}\log\frac{d}{\epsilon\delta})$ gradient and Hessian-vector product evaluations\footnote{\citet{Carmon-2016} use $\nabla^2 f(\x)\v=\lim_{h\rightarrow 0}\frac{\nabla f(\x+h\v)-\nabla f(\x)}{h}$ to approximate the Hessian-vector product.}, where $d$ is the dimension of $\x$ in problem (\ref{problem}). To avoid the Hessian-vector products, \citet{yang-2018-nips} and \citet{Allen-Zhu-2018-nc} proposed the NEON and NEON2 first-order procedures to extract directions of negative curvature from the Hessian, respectively, which can be used to turn a first-order stationary point finding algorithm into a second-order stationary point finding one. 
Other typical algorithms include the Newton-conjugate gradient \citep{Royer20mp} and the second-order line-search method \citep{Clement-siam-2018}, which are beyond the class of accelerated methods.

The above methods are nested-loop algorithms, where the outer loop needs to call a serious of subroutines such as negative curvature exploitation, minimization of regularized surrogate functions using convex AGD \citep{Carmon-2016,carmon-2017-guilty}, or computation of cubic regularized Newton directions \citep{Zhu-16-apg,Carmon-siam-2020,Carmon-2018-nips}. \citet{jinchi-18-apg} proposed the first single-loop accelerated method, which also finds an $\epsilon$-approximate second-order stationary point in $\bO(\epsilon^{-7/4}\log\frac{d}{\epsilon\delta})$ gradient and function computations with probability at least $1-\delta$. The algorithm in \citep{jinchi-18-apg} runs the classical AGD until the function becomes ``too nonconvex" locally, then it calls negative curvature exploitation. To the best of our knowledge, it is the simplest method among the nonconvex accelerated algorithms with fast rate guarantees.

Although achieving second-order stationary point guarantees the method to escape strict saddle points, some researchers show that gradient descent and its accelerated variants that converge to first-order stationary point always converge to local minimum. \citet{Jason-16-gd} proved that gradient descent converges to a local minimizer almost surely with random initialization. \citet{hb-sun-2019} gave the similar result for the heavy ball method. \citet{Neill2019mp} examined the behavior of HB and AGD near strict saddle points and proved that both methods diverge from these points more rapidly than gradient descent for specific quadratic functions.

\subsubsection{Lower Bound for Second-order Smooth Nonconvex Problems}

\citet{Carmon-mp2-2021} studied the lower bounds for finding stationary point using first-order methods. For nonconvex functions with Lipschitz continuous gradient and Hessian, they established that deterministic first-order methods cannot find $\epsilon$-approximate first-order stationary points in less than $\bO(\epsilon^{-12/7})$ gradient evaluations. There exists a gap of $\bO(\epsilon^{-1/28}\log\frac{1}{\epsilon})$ between this lower bound and the best known upper bound \citep{carmon-2017-guilty}. It remains an open problem of how to close this gap. It is also unclear which of the upper bound and lower bound is tight \citep[Section 7]{Carmon-mp2-2021}.

\subsection{Contribution}

All the above accelerated algorithms \citep{carmon-2017-guilty,Carmon-2016,Zhu-16-apg,Carmon-siam-2020,jinchi-18-apg} share the state-of-the-art $\bO(\epsilon^{-7/4}\log\frac{1}{\epsilon})$ complexity, which has a $\bO(\log\frac{1}{\epsilon})$ factor. As far as we know, even when we apply the methods designed to find second-order stationary point to the easier problem of finding first-order stationary one, we still cannot remove the $\bO(\log\frac{1}{\epsilon})$ factor. On the other hand, almost all the existing accelerated methods need to call additional subroutines and thus they are complex with nested loops. Even the single-loop method proposed in \citep{jinchi-18-apg} requires negative curvature exploitation.

In this paper, we propose two simple accelerated methods, restarted AGD and restarted HB, which have the following three advantages:

\begin{enumerate}
\item Our algorithms find an $\epsilon$-approximate first-order stationary point within $\bO(\epsilon^{-7/4})$ number of gradient evaluations under the conditions that both the gradient and Hessian are Lipschitz continuous. We do not hide any polylogarithmic factors in our complexity, and thus it improves over the best known one by the $\bO(\log\frac{1}{\epsilon})$ factor.
\item Our algorithms are simple in the sense that they only consist of Nesterov's classical AGD or Polyak's HB iterations, as well as a restart mechanism. They do not invoke negative curvature exploitation or minimization of regularized surrogate functions or computation of cubic regularized Newton directions as the subroutines. 
\item Technically, our elementary proofs use less advanced techniques compared with existing works. Especially, it is irrelevant to the analysis of strongly convex AGD or HB, which is crucial to cancel the $\bO(\log\frac{1}{\epsilon})$ factor. 
\end{enumerate}


\subsection{Notations and Assumptions}

We use lowercase bold letters to represent vectors, uppercase bold letters for matrices, and non-bold (both lowercase and uppercase) letters for scalars. Denote $\x_j$ and $\nabla_j f(\x)$ as the $j$th element of $\x$ and $\nabla f(\x)$, respectively. For the vectors produced in the iterative algorithms, for example, $\x$, denote $\x^k$ to be the value at the $k$th iteration. 
We denote $\|\cdot\|$ to be the $\ell_2$ Euclidean norm for vectors, $\|\cdot\|_2$ as the spectral norm and $\|\cdot\|_F$ as the Frobenius norm for matrices. We make the following standard assumptions in this paper.
\begin{assumption}\label{assump}
\begin{enumerate}
\item $f(\x)$ is $L$-gradient Lipschitz: $\|\nabla f(\x)-\nabla f(\y)\|\leq L\|\x-\y\|,\forall \x,\y\in\R^d$,
\item $f(\x)$ is $\rho$-Hessian Lipschitz: $\|\nabla^2 f(\x)-\nabla^2 f(\y)\|_2\leq \rho\|\x-\y\|,\forall \x,\y\in\R^d$,
\end{enumerate}
\end{assumption}
which yield the following two well-known inequalities:
\begin{eqnarray}
&\left|f(\y)- f(\x)-\<\nabla f(\x),\y-\x\>\right|\leq\frac{L}{2}\|\y-\x\|^2,\label{gd_smooth1}\\
&\left|f(\y)- f(\x)-\<\nabla f(\x),\y-\x\>-(\y-\x)^T\nabla^2 f(\x)(\y-\x)\right|\leq\frac{\rho}{6}\|\y-\x\|^3.\label{h_smooth1}
\end{eqnarray}
We also assume that the objective function is lower bounded, that is, $\min_{\x}f(\x)>-\infty$.

\section{Restarted Accelerated Gradient Descent}\label{sec2}

Nesterov's classical AGD consists of the following iterations:
\begin{eqnarray}
\begin{aligned}\notag
\y^{k}=\x^{k}+(1-\theta)(\x^{k}-\x^{k-1}),\quad\x^{k+1}=\y^{k}-\eta\nabla f(\y^{k}),
\end{aligned}
\end{eqnarray}
where $\theta=\frac{2\sqrt{\mu}}{\sqrt{L}+\sqrt{\mu}}$ for strongly convex problems and it varies as $\frac{3}{k+2}$ at the $k$th iteration for convex problems. The term $(\x^{k}-\x^{k-1})$ is often regarded as momentum. When applying the above iteration to nonconvex problems, the major challenge in faster convergence analysis is that the objective function (even the Hamiltonian potential function used in \citep{jinchi-18-apg}) does not decrease monotonically, especially when we set $\eta=\bO(\frac{1}{L})$ and $\theta$ small (for example, of the order $\epsilon$). To address this issue, \citet{jinchi-18-apg} invoke negative curvature exploitation when the local objective function is very nonconvex. An open problem is asked in Section 5 of \citep{jinchi-18-apg} whether negative curvature exploitation is indispensable to guarantee the fast rate. In contrast with \citep{jinchi-18-apg}, we use the restart mechanism to ensure the decrease of the objective function, and thus avoid negative curvature exploitation.

\begin{algorithm}[t]
   \caption{Restarted AGD for Nonconvex Optimization (RAGD-NC)}
   \label{AGD1}
\begin{algorithmic}[1]
   \STATE Initialize $\x^{-1}=\x^{0}=\x_{int}$, $k=0$.
   \WHILE{$k<K$}
   \STATE $\y^{k}=\x^{k}+(1-\theta)(\x^{k}-\x^{k-1})$
   \STATE $\x^{k+1}=\y^{k}-\eta\nabla f(\y^{k})$
   \STATE $k=k+1$
   \IF{$k\sum_{t=0}^{k-1}\|\x^{t+1}-\x^{t}\|^2> B^2$}
   \STATE $\x^{-1}=\x^{0}=\x^{k}$, $k=0$
   \ENDIF
   \ENDWHILE
   \STATE $K_0=\argmin_{\lfloor \frac{K}{2}\rfloor\leq k\leq K-1}\|\x^{k+1}-\x^{k}\|$
   \STATE Output $\hat\y=\frac{1}{K_0+1}\sum_{k=0}^{K_0}\y^{k}$
\end{algorithmic}
\end{algorithm}

We present our method in Algorithm \ref{AGD1}. It runs Nesterov's classical AGD iterations until the ``if condition" triggers. Then we reset $\x^0$ and $\x^{-1}$ equal to $\x^k$ and continue to the next round of AGD. The method terminates and outputs a specific average when the ``if condition" does not trigger in $K$ iterations. To simplify the description, we define one round of AGD between two successive restarts to be one ``epoch". The restart trick, first proposed in \citep{Donoghue-2015-NesRestart}, is motivated by \citep{congfang-19-sgd}, where a ball-mechanism is proposed as the stopping criteria to analyze SGD. 

Our main result is described in Theorem \ref{theorem1}, which establishes the $\bO(\epsilon^{-7/4})$ complexity to achieve an $\epsilon$-approximate first-order stationary point. We defer the proofs until Section \ref{sec:proof_t1}.
\begin{theorem}\label{theorem1}
Suppose that Assumption \ref{assump} holds. Let $\eta=\frac{1}{4L}$, $B=\sqrt{\frac{\epsilon}{\rho}}$, $\theta=4(\epsilon\rho\eta^2)^{1/4}\in(0,1]$, and $K=\frac{1}{\theta}$. Then Algorithm \ref{AGD1} terminates in at most $\frac{\triangle_fL^{1/2}\rho^{1/4}}{\epsilon^{7/4}}$ gradient computations and the output satisfies $\|\nabla f(\hat\y)\|\leq82\epsilon$, where $\triangle_f=f(\x_{int})-\min_{\x}f(\x)$.
\end{theorem}

Among the existing methods, the ``convex until proven guilty" method proposed in \citep{carmon-2017-guilty} achieves an $\epsilon$-approximate first-order stationary point in $\bO (\frac{\triangle_fL^{1/2}\rho^{1/4}}{\epsilon^{7/4}}\log\frac{L\triangle_f}{\epsilon})$ gradient and function evaluations, which is slower than our method by the $\bO(\log\frac{1}{\epsilon})$ factor. The complexity established in other work focusing on second-order stationary point, such as \citep{Carmon-2016,Zhu-16-apg,Carmon-siam-2020,jinchi-18-apg}, also has the additional $\bO(\log\frac{1}{\epsilon})$ factor even when only pursuing first-order stationary point. Take \citep{jinchi-18-apg} as the example. Their Lemma 7 concentrates on the first-order stationary point. They built the proofs of their Lemmas 9 and 17 upon the analysis of strongly convex AGD, which generally requires $\bO(\sqrt{\frac{L}{\mu}}\log\frac{1}{\epsilon})$ iterations such that the gradient norm will be less than $\epsilon$. Thus, the $\bO(\log\frac{1}{\epsilon})$ factor appears.

\begin{remark}\label{remark1}
\begin{enumerate}
\item
The specific average on lines 10 and 11 of Algorithm \ref{AGD1} is the crucial technique to remove the $\bO(\log\frac{1}{\epsilon})$ factor. See the proof of Lemma \ref{lemma6}. This phenomenon that some averaged iterate converges faster than the final iterate theoretically has also been observed in other algorithms. For example, for Lipschitz and strongly convex functions, but not necessarily differentiable, \citet{Shamir-2013-icml} proved the $\bO(\frac{\log T}{T})$ error of the final iterate of SGD while the $\bO(\frac{1}{T})$ one for the suffix averaged iterate. Both rates are tight matching the corresponding lower bounds \citep{Harvey-19-sgd}. For linearly constrained convex problems, \citet{Davis-2017-admm} proved the $\bO(\frac{1}{\sqrt{T}})$ rate for the final iterate of ADMM while the $\bO(\frac{1}{T})$ one for the averaged iterate. The two rates are also tight \citep{Davis-2017-admm}.

We can extend this technical trick to the method proposed in \citep{jinchi-18-apg} and greatly simplify their proofs with the slightly faster $\bO(\epsilon^{-7/4})$ convergence rate. See the supplementary material of our conference version \citep{li-2022-icml}. On the other hand, we can also prove that the gradient at the last iterate in our method is small with norm being less than $\epsilon$ by employing the proof techniques in \citep{jinchi-18-apg}, at the expense of introducing the additional $\bO(\log\frac{1}{\epsilon})$ factor and  complicating the proofs.
\item
Restart plays the role of decreasing the objective function at each epoch of AGD. See Corollary \ref{lemma5}. Intuitively, when the iterates are far from the local starting point $\x^0$ or the momentum $\x^{k}-\x^{k-1}$ is large such that it may potentially increase the objective function, restart cancels the effect of momentum by setting it to 0. 
\item
As discussed in Section \ref{sec-comp}, since our proofs do not invoke the analysis of strongly convex AGD or HB, the acceleration mechanism for nonconvex optimization seems irrelevant to the analysis of convex AGD. Our proofs show that momentum and its parameter $\theta$ play an important role in the analysis of nonconvex acceleration mechanism.
\end{enumerate}
\end{remark}

\subsection{Adaptive Implementation and Infrequent Restart}

In Algorithm \ref{AGD1}, we set $B$ small in theory such that the method may restart frequently, making it almost reduce to the classical gradient descent, especially for high dimensional problems. To take advantage of the practical efficiency of AGD, we should reduce the frequency of restart. A straightforward idea is to set a large $B$ initially and reduce it gradually. We present an adaptive implementation of Algorithm \ref{AGD1} in Algorithm \ref{AGD1p}, which relaxes the restart condition of $k\sum_{t=0}^{k-1}\|\x^{t+1}-\x^{t}\|^2> B^2$ to $k\sum_{t=0}^{k-1}\|\x^{t+1}-\x^{t}\|^2> \max\{B^2,B_0^2\}$, where $B_0$ can be initialized much larger than $B$ and is decreased geometrically after each epoch. The decrease condition on line 8 of Algorithm \ref{AGD1p} comes form Corollary \ref{lemma5}. Intuitively, when $B_0\leq B$, we always have $f(\x^{k})-f(\x^0)\leq -\frac{7\epsilon^{3/2}}{8\sqrt{\rho}}$ from Corollary \ref{lemma5}. That is, line 11 never executes when $B_0$ decreases to be smaller than $B$ after $\bO(\log_{c_0}\frac{1}{\epsilon})$ epochs and Algorithm \ref{AGD1p} is equivalent to Algorithm \ref{AGD1} in this case. When the decrease condition on line 8 does not hold, which indicates that the algorithm may diverge, we discard the whole iterates in this epoch and go back to the last iterate of the previous epoch, which is stored in $\x_{cur}^{0}$. We terminate Algorithm \ref{AGD1p} when $B_0\leq B$ and $k$ equals to $K$. On the other hand, we output the one of $\x^K$ and $\hat\y$ with smaller gradient norm. In practice, the last iterate always converges faster than the averaged iterate. We describe the $\bO (\epsilon^{-7/4})$ complexity of Algorithm \ref{AGD1p} in Theorem \ref{theorem1p} and defer the proofs until Section \ref{sec:proof_t1p}.

\begin{algorithm}[t]
   \caption{Adaptively Restarted AGD for Nonconvex Optimization (Ada-RAGD-NC)}
   \label{AGD1p}
\begin{algorithmic}[1]
   \STATE Initialize $\x^{-1}=\x^{0}=\x_{cur}^0=\x_{int}$, $k=0$, $B_0$.
   \WHILE{$k<K$ or $B_0>B$}
   \STATE $\y^{k}=\x^{k}+(1-\theta)(\x^{k}-\x^{k-1})$
   \STATE $\x^{k+1}=\y^{k}-\eta\nabla f(\y^{k})$
   \STATE $k=k+1$
   \IF{$k\sum_{t=0}^{k-1}\|\x^{t+1}-\x^{t}\|^2> \max\{B^2,B_0^2\}$ or $k> K$}
   \STATE $B_0=B_0/c_0$
   \IF{$f(\x^{k})-f(\x^0)\leq -\gamma\frac{\epsilon^{3/2}}{\sqrt{\rho}}$}
   \STATE $\x^{-1}=\x^{0}=\x^{k}$, $\x_{cur}^0=\x^{k}$, $k=0$
   \ELSE
   \STATE $\x^{-1}=\x^{0}=\x_{cur}^{0}$, $k=0$, $B_0=B_0/c_1$
   \ENDIF
   \ENDIF
   \ENDWHILE
   \STATE $K_0=\argmin_{\lfloor \frac{K}{2}\rfloor\leq k\leq K-1}\|\x^{k+1}-\x^{k}\|$
   \STATE $\hat\y=\frac{1}{K_0+1}\sum_{k=0}^{K_0}\y^{k}$
   \STATE Output $\x_{out}=\argmin_{\x^{K},\hat\y}\{\|\nabla f(\x^{K})\|,\|\nabla f(\hat\y)\|\}$
\end{algorithmic}
\end{algorithm}

\begin{theorem}\label{theorem1p}
Suppose that Assumption \ref{assump} holds. Let $\eta=\frac{1}{4L}$, $B=\sqrt{\frac{\epsilon}{\rho}}$, $\theta=4(\epsilon\rho\eta^2)^{1/4}\in(0,1)$, $K=\lfloor\frac{1}{\theta}\rfloor$, $\gamma\leq\frac{7}{8}$, $c_0>1$, and $c_1>1$. Then Algorithm \ref{AGD1p} terminates in at most $\bO\left(\frac{\triangle_fL^{1/2}\rho^{1/4}}{\epsilon^{7/4}}+\frac{L^{1/2}}{\epsilon^{1/4}\rho^{1/4}}\log \frac{\rho B_0}{\epsilon}\right)$ gradient computations and $\bO\left(\frac{\triangle_f\sqrt{\rho}}{\epsilon^{3/2}}+\log \frac{\rho B_0}{\epsilon}\right)$ function evaluations, and the output satisfies $\|\nabla f(\x_{out})\|\leq\bO(\epsilon)$.
\end{theorem}

\begin{remark}\label{remark2}
Algorithm \ref{AGD1p} also applies to the case when the Lipschitz constants $L$ and $\rho$ are unknown. We can initialize a small guess of $\rho'$, tune an appropriate $\eta$, and replace line 11 of Algorithm \ref{AGD1p} by the following steps:
\begin{equation}
\x^{-1}=\x^{0}=\x_{cur}^{0}, \quad k=0, \quad B_0=\frac{B_0}{c_1},\quad \eta=\max\left(\frac{\eta}{c_2},\eta_{min}\right),\quad \rho'=\min(\rho' c_2^2,\rho_{max}'), \label{agdp-cont1}
\end{equation}
where $c_1\geq c_2>1$, and the output also satisfies $\|\nabla f(\x_{out})\|\leq\bO(\epsilon)$ within $\bO (\epsilon^{-7/4})$ gradient computations and $\bO (\epsilon^{-3/2})$ function evaluations. See Theorem \ref{theorem1pp} in Section \ref{sec:proof_t1p} for the details.
\end{remark}

\subsection{Extension to the Second-order Stationary Point}

Our restarted AGD can also find $\epsilon$-approximate second-order stationary point, namely a point $\x$ that satisfies
\begin{equation}
\|\nabla f(\x)\|\leq\epsilon\quad\mbox{and}\quad \lambda_{min}(\nabla^2 f(\x))\geq-\sqrt{\epsilon\rho},\label{def_2point}
\end{equation}
where $\lambda_{min}$ means the smallest eigenvalue. We follow \citep{jinchi-2017-icml,jinchi-18-apg} to add perturbations to the iterates. Specifically, we only need to replace line 7 of Algorithm \ref{AGD1} by the following step:
\begin{eqnarray}\label{cont12}
&\x^{-1}=\x^{0}=\x^{k}+\xi 1_{\|\nabla f(\y^{k-1})\|\leq\frac{B}{\eta}},\quad \xi\sim \mbox{Unif}(\B_0(r)),\quad k=0,
\end{eqnarray}
where $\mbox{Unif}(\B_0(r))$ means the uniform distribution in the ball $\B_0(r)$ with radius $r$ and center 0, and $1_{\|\nabla f(\y^{k-1})\|\leq\frac{B}{\eta}}=\left\{\begin{array}{cl}
    1, & \mbox{if } \|\nabla f(\y^{k-1})\|\leq\frac{B}{\eta},\\
    0, & \mbox{otherwise}.
  \end{array}\right.$

The convergence and complexity is presented in Theorem \ref{theorem3}. We see that the perturbed RAGD-NC needs at most $\bO(\epsilon^{-7/4}\log\frac{d}{\zeta\epsilon})$ gradient evaluations to find an $\epsilon$-approximate second-order stationary point with probability at least $1-\zeta$, where $d$ is the dimension of $\x$ in problem (\ref{problem}). Our algorithm has the same complexity with the one given in \citep{jinchi-18-apg}. Comparing with Theorem \ref{theorem1}, we see that this complexity is higher by the $\bO(\log\frac{d}{\zeta\epsilon})$ factor. Currently, it is unclear how to cancel it, and we conjecture that the polylogarithmic factor may not be removed when pursuing second-order stationary point \citep{Simchowitz-logd}.

\begin{theorem}\label{theorem3}
Suppose that Assumption \ref{assump} holds. Let $\chi=\bO(\log\frac{d}{\zeta\epsilon})\geq 1$, $\eta=\frac{1}{4L}$, $B=\frac{1}{288\chi^2}\sqrt{\frac{\epsilon}{\rho}}$, $\theta=\frac{1}{2}\left(\frac{\epsilon\rho}{L^2}\right)^{1/4}<1$, $K=\frac{2\chi}{\theta}$, $r=\min\{\frac{B}{2},\frac{\theta B}{20K},\sqrt{\frac{\theta B^2}{2K}}\}=\bO(\epsilon)$. Then the perturbed RAGD-NC (Algorithm \ref{AGD1} with (\ref{cont12})) terminates in at most $\bO\left(\frac{\triangle_fL^{1/2}\rho^{1/4}\chi^6}{\epsilon^{7/4}}\right)$ gradient computations and the output satisfies $\|\nabla f(\hat\y)\|\leq\epsilon$, where $\triangle_f=f(\x_{int})-\min_{\x}f(\x)$. It also satisfies $\lambda_{min}(\nabla^2 f(\hat\y))\geq -1.011\sqrt{\epsilon\rho}$ with probability at least $1-\zeta$.
\end{theorem}

The proof of this theorem is essentially identical to those in \citep{jinchi-18-apg}. We omit the proofs and they can be found in the supplementary material of our conference version \citep{li-2022-icml}.

\section{Restarted Heavy Ball Method}

Polyak's classical heavy ball method \citep{Polyak-1964} iterates with the following step
\begin{equation}\notag
\x^{k+1}=\x^{k}-\eta\nabla f(\x^{k})+(1-\theta)(\x^{k}-\x^{k-1}),
\end{equation}
where $\eta=\frac{4}{(\sqrt{L}+\sqrt{\mu})^2}$ and $1-\theta=\frac{(\sqrt{L}-\sqrt{\mu})^2}{(\sqrt{L}+\sqrt{\mu})^2}$ for strongly convex problems. In the deep learning literature, people often use the following equivalent iterations empirically with the running average \citep{hb-dl-2013},
\begin{eqnarray}
\begin{aligned}\notag
\m^k=\beta\m^{k-1}+\nabla f(\x^k),\quad\x^{k+1}=\x^k-\eta\m^k,
\end{aligned}
\end{eqnarray}
where $\m^{-1}=0$ and $\beta=1-\theta$ for the deterministic problems. When applying the heavy ball iteration to nonconvex optimization, people often set $\eta=\bO(\frac{\theta}{L})$ to ensure the convergence \citep{iPiano-2014,hb-sun-2019}, which prevents us from proving faster convergence in theory and slows down the algorithm in practice when $\theta$ is small. To address this issue, similar to RAGD-NC, we combine the restart mechanism with the heavy ball method such that $\eta=\bO(\frac{1}{L})$ while maintaining $\theta$ small. Our method is presented in Algorithm \ref{HB1}. It runs Polyak's classical HB iteration until the ``if condition" triggers. Then we restart from the auxiliary vector $\z^k$, a convex combination of $\x^k$ and $\x^{k-1}$, and do the next round of HB iterations. Algorithm \ref{HB1} shares almost the same framework as Algorithm \ref{AGD1}, and the only difference comes from the iterate $\z^k$, which is designed to fit the proof. See Remark \ref{proof-comp} for the detailed reason.

\begin{algorithm}[t]
   \caption{Restarted HB for Nonconvex Optimization (RHB-NC)}
   \label{HB1}
\begin{algorithmic}[1]
   \STATE Initialize $\x^{-1}=\x^{0}=\x_{int}$, $k=0$.
   \WHILE{$k<K$}
   \STATE $\x^{k+1}=\x^{k}-\eta\nabla f(\x^{k})+(1-\theta)(\x^{k}-\x^{k-1})$
   \STATE $k=k+1$
   \IF{$k\sum_{t=0}^{k-1}\|\x^{t+1}-\x^{t}\|^2> B^2$}
   \STATE $\z^k=\frac{\x^k+(1-2\theta)(1-\theta)\x^{k-1}}{1+(1-2\theta)(1-\theta)}$
   \STATE $\x^{-1}=\x^{0}=\z^{k}$, $k=0$
   \ENDIF
   \ENDWHILE
   \STATE $K_0=\argmin_{\lfloor \frac{K}{2}\rfloor\leq k\leq K-1}\|\x^{k+1}-\x^{k}\|$
   \STATE Output $\hat\x=\frac{1}{K_0+1}\sum_{k=0}^{K_0}\x^{k}$
\end{algorithmic}
\end{algorithm}

The main result is given in Theorem \ref{theorem1hb}, which also establishes the $\bO(\epsilon^{-7/4})$ complexity to find an $\epsilon$-approximate first-order stationary point, and we defer the proofs until Section \ref{sec:proof_t2}. Comparing with Theorem \ref{theorem1}, we see that the two algorithms need the same assumptions, share the same convergence rate, and have almost the same parameter settings, which indicate that no one is superior to the other in theory for nonconvex optimization. As a comparison, the heavy ball method requires more assumptions for strongly convex problems and has the slower convergence rate in theory for convex problems than AGD.

\begin{theorem}\label{theorem1hb}
Suppose that Assumption \ref{assump} holds. Let $\eta=\frac{1}{4L}$, $B=\sqrt{\frac{\epsilon}{4\rho}}$, $\theta=10\left(\epsilon\rho\eta^2\right)^{1/4}\in(0,\frac{1}{10}]$, and $K=\frac{1}{\theta}$. Then Algorithm \ref{HB1} terminates in at most $\frac{\triangle_fL^{1/2}\rho^{1/4}}{\epsilon^{7/4}}$ gradient computations and the output satisfies $\|\nabla f(\hat\x)\|\leq242\epsilon$, where $\triangle_f=f(\x_{int})-\min_{\x}f(\x)$.
\end{theorem}

\begin{algorithm}[t]
   \caption{Adaptively Restarted HB for Nonconvex Optimization (Ada-RHB-NC)}
   \label{HB1p}
\begin{algorithmic}[1]
   \STATE Initialize $\x^{-1}=\x^{0}=\x_{cur}^0=\x_{int}$, $k=0$, $B_0$.
   \WHILE{$k<K$ or $B_0>B$}
   \STATE $\x^{k+1}=\x^{k}-\eta\nabla f(\x^{k})+(1-\theta)(\x^{k}-\x^{k-1})$
   \STATE $k=k+1$
   \IF{$k\sum_{t=0}^{k-1}\|\x^{t+1}-\x^{t}\|^2> \max\{B^2,B_0^2\}$ or $k> K$}
   \STATE $\z^k=\frac{\x^k+(1-2\theta)(1-\theta)\x^{k-1}}{1+(1-2\theta)(1-\theta)}$
   \STATE $B_0=B_0/c_0$
   \IF{$f(\z^{k})-f(\x^0)\leq -\gamma\frac{\epsilon^{3/2}}{\sqrt{\rho}}$}
   \STATE $\x^{-1}=\x^{0}=\z^{k}$, $\x_{cur}^0=\z^{k}$, $k=0$
   \ELSE
   \STATE $\x^{-1}=\x^{0}=\x_{cur}^{0}$, $k=0$, $B_0=B_0/c_1$
   \ENDIF
   \ENDIF
   \ENDWHILE
   \STATE $K_0=\argmin_{\lfloor \frac{K}{2}\rfloor\leq k\leq K-1}\|\x^{k+1}-\x^{k}\|$
   \STATE $\hat\x=\frac{1}{K_0+1}\sum_{k=0}^{K_0}\x^{k}$
   \STATE Output $\x_{out}=\argmin_{\x^{K},\hat\x}\{\|\nabla f(\x^{K})\|,\|\nabla f(\hat\x)\|\}$
\end{algorithmic}
\end{algorithm}

In practice, Algorithm \ref{HB1} has the same disadvantages as Algorithm \ref{AGD1} when $B$ is small. Similar to Algorithm \ref{AGD1p}, we also propose an adaptive implementation of Algorithm \ref{HB1}, and present it in Algorithm \ref{HB1p}. Theorem \ref{theorem1hbp} gives the $\bO(\epsilon^{-7/4})$ complexity.

\begin{theorem}\label{theorem1hbp}
Suppose that Assumption \ref{assump} holds. Let $\eta=\frac{1}{4L}$, $B=\sqrt{\frac{\epsilon}{4\rho}}$, $\theta=10\left(\epsilon\rho\eta^2\right)^{1/4}\in(0,\frac{1}{10}]$, $K=\lfloor\frac{1}{\theta}\rfloor$, $\gamma\leq 1$, $c_0>1$, and $c_1>1$. Then Algorithm \ref{HB1p} terminates in at most $\bO\left(\frac{\triangle_fL^{1/2}\rho^{1/4}}{\epsilon^{7/4}}+\frac{L^{1/2}}{\epsilon^{1/4}\rho^{1/4}}\log \frac{\rho B_0}{\epsilon}\right)$ gradient computations and $\bO\left(\frac{\triangle_f\sqrt{\rho}}{\epsilon^{3/2}}+\log \frac{\rho B_0}{\epsilon}\right)$ function evaluations, and the output satisfies $\|\nabla f(\x_{out})\|\leq\bO(\epsilon)$.
\end{theorem}

\section{Proof of the Theorems}\label{section:proof}

We prove Theorems \ref{theorem1}, \ref{theorem1p}, and \ref{theorem1hb} in this section. The proof of Theorem \ref{theorem1hbp} is almost the same to that of Theorem \ref{theorem1p} and we omit the details.

\subsection{Proof of Theorem \ref{theorem1}}\label{sec:proof_t1}
We prove the convergence rate of Algorithm \ref{AGD1} in this section. Denote $\K$ to be the iteration number when the ``if condition" on line 6 of Algorithm \ref{AGD1} triggers, that is,
\begin{equation}
\K=\min_k\left\{k\Bigg|k\sum_{t=0}^{k-1}\|\x^{t+1}-\x^{t}\|^2>B^2\right\}.\label{define_K}
\end{equation}
For each epoch consisting of one round of AGD from iterations $k=0$ to $k=\K$, we have
\begin{subequations}
\begin{align}
&1\leq\K\leq K,\qquad\K\sum_{t=0}^{\K-1}\|\x^{t+1}-\x^{t}\|^2 >B^2,\qquad\mbox{and}\label{cond1}\\
&\|\x^{k}-\x^{0}\|^2=\left\|\sum_{t=0}^{k-1}\x^{t+1}-\x^{t}\right\|^2\leq k\sum_{t=0}^{k-1}\|\x^{t+1}-\x^{t}\|^2\leq B^2,\forall k< \K,\label{cond2}
\end{align}
\end{subequations}
where the last inequality comes from the definition of $\K$. From the update of $\y$ on line 3 of Algorithm \ref{AGD1}, we also have
\begin{equation}
\|\y^{k}-\x^{0}\|\leq\|\x^{k}-\x^{0}\|+\|\x^{k}-\x^{k-1}\|\leq 2B,\forall k< \K.\label{cond0}
\end{equation}
On the other hand, for the last epoch where the ``if condition" does not trigger and the while loop breaks when $k$ increases to $K$, we have
\begin{subequations}
\begin{align}
&\|\x^{k}-\x^{0}\|^2\leq k\sum_{t=0}^{k-1}\|\x^{t+1}-\x^{t}\|^2\leq B^2,\forall k\leq K,\label{cond3}\\
&\|\y^{k}-\x^{0}\|\leq 2B,\forall k\leq K.\label{cond4}
\end{align}
\end{subequations}
We will show that the function value decreases at least $\bO(\epsilon^{1.5})$ in each epoch except the last one in Sections \ref{subsec1} and \ref{subsec2}. Thus, Algorithm \ref{AGD1} terminates in at most $\bO(\epsilon^{-1.5})$ epochs. Since each epoch needs at most $\bO(\epsilon^{-0.25})$ iterations, Algorithm \ref{AGD1} requires at most $\bO(\epsilon^{-1.75})$ total gradient evaluations. In the last epoch, we will show in Section \ref{subsec3} that the gradient norm at the output iterate is less than $\bO(\epsilon)$.

\subsubsection{Large Gradient of $\|\nabla f(\y^{\K-1})\|$}\label{subsec1}
We first consider the case when $\|\nabla f(\y^{\K-1})\|$ is large.
\begin{lemma}\label{lemma1}
Suppose that Assumption \ref{assump} holds. Let $\eta\leq\frac{1}{4L}$ and $0\leq\theta\leq 1$. In each epoch of Algorithm \ref{AGD1} where the ``if condition" triggers, when $\|\nabla f(\y^{\K-1})\|> \frac{B}{\eta}$, we have
\begin{eqnarray}
\begin{aligned}\notag
f(\x^{\K})-f(\x^{0})\leq-\frac{B^2}{4\eta}.
\end{aligned}
\end{eqnarray}
\end{lemma}
\begin{proof}
As the gradient is $L$-Lipschitz, we have
\begin{eqnarray}
\begin{aligned}\label{cont11}
f(\x^{k+1})\leq& f(\y^{k})+\<\nabla f(\y^{k}),\x^{k+1}-\y^{k}\>+\frac{L}{2}\|\x^{k+1}-\y^{k}\|^2\\
=&f(\y^{k})-\eta\|\nabla f(\y^{k})\|^2+\frac{L\eta^2}{2}\|\nabla f(\y^{k})\|^2\\
\leq& f(\y^{k})-\frac{7\eta}{8}\|\nabla f(\y^{k})\|^2,
\end{aligned}
\end{eqnarray}
where we use the AGD iteration on line 4 of Algorithm \ref{AGD1} and $\eta\leq\frac{1}{4L}$. From the $L$-gradient Lipschitz, we also have
\begin{eqnarray}
\begin{aligned}\notag
f(\x^{k})\geq f(\y^{k})+\<\nabla f(\y^{k}),\x^{k}-\y^{k}\>-\frac{L}{2}\|\x^{k}-\y^{k}\|^2.
\end{aligned}
\end{eqnarray}
So we have
\begin{eqnarray}
\begin{aligned}\notag
&f(\x^{k+1})-f(\x^{k})\\
\leq& -\<\nabla f(\y^{k}),\x^{k}-\y^{k}\>+\frac{L}{2}\|\x^{k}-\y^{k}\|^2-\frac{7\eta}{8}\|\nabla f(\y^{k})\|^2\\
=&\frac{1}{\eta}\<\x^{k+1}-\y^{k},\x^{k}-\y^{k}\>+\frac{L}{2}\|\x^{k}-\y^{k}\|^2-\frac{7\eta}{8}\|\nabla f(\y^{k})\|^2\\
=& \frac{1}{2\eta}\left(\|\x^{k+1}-\y^{k}\|^2+\|\x^{k}-\y^{k}\|^2-\|\x^{k+1}-\x^{k}\|^2\right)+\frac{L}{2}\|\x^{k}-\y^{k}\|^2-\frac{7\eta}{8}\|\nabla f(\y^{k})\|^2\\
\overset{a}\leq& \frac{5}{8\eta}\|\x^{k}-\y^{k}\|^2-\frac{1}{2\eta}\|\x^{k+1}-\x^{k}\|^2-\frac{3\eta}{8}\|\nabla f(\y^{k})\|^2\\
\overset{b}\leq& \frac{5}{8\eta}\|\x^{k}-\x^{k-1}\|^2-\frac{1}{2\eta}\|\x^{k+1}-\x^{k}\|^2-\frac{3\eta}{8}\|\nabla f(\y^{k})\|^2,
\end{aligned}
\end{eqnarray}
where we use $L\leq\frac{1}{4\eta}$ in $\overset{a}\leq$ and $\|\x^{k}-\y^{k}\|=(1-\theta)\|\x^{k}-\x^{k-1}\|\leq\|\x^{k}-\x^{k-1}\|$ in $\overset{b}\leq$. Summing over $k=0,\cdots,\K-1$ and using $\x^{0}=\x^{-1}$, we have
\begin{eqnarray}
\begin{aligned}\notag
f(\x^{\K})-f(\x^{0})\leq& \frac{1}{8\eta}\sum_{k=0}^{\K-2}\|\x^{k+1}-\x^{k}\|^2-\frac{3\eta}{8}\sum_{k=0}^{\K-1}\|\nabla f(\y^{k})\|^2
\end{aligned}
\end{eqnarray}
\begin{eqnarray}
\hspace*{3cm}\begin{aligned}\notag
\overset{c}\leq& \frac{B^2}{8\eta}-\frac{3\eta}{8}\|\nabla f(\y^{\K-1})\|^2\overset{d}\leq \frac{B^2}{8\eta}-\frac{3 B^2}{8\eta}= -\frac{B^2}{4\eta},
\end{aligned}
\end{eqnarray}
where we use (\ref{cond2}) in $\overset{c}\leq$ and $\|\nabla f(\y^{\K-1})\|> \frac{B}{\eta}$ in $\overset{d}\leq$.
\end{proof}

\subsubsection{Small Gradient of $\|\nabla f(\y^{\K-1})\|$}\label{subsec2}

If $\|\nabla f(\y^{\K-1})\|\leq \frac{B}{\eta}$, then from the AGD iteration on line 4 and (\ref{cond0}) we have
\begin{eqnarray}
\begin{aligned}\notag
\|\x^{\K}-\x^{0}\|\leq \|\y^{\K-1}-\x^{0}\|+\eta\|\nabla f(\y^{\K-1})\|\leq 3B.
\end{aligned}
\end{eqnarray}
For each epoch, denote $\H=\nabla^2 f(\x^{0})$ to be the Hessian matrix at the starting iterate and $\H=\U\bLambda\U^T$ to be its eigenvalue decomposition with $\U,\bLambda\in\R^{d\times d}$. Let $\blambda_j$ be the $j$th eigenvalue. Define $\widetilde\x=\U^T\x$, $\widetilde\y=\U^T\y$, and $\widetilde\nabla f(\y)=\U^T\nabla f(\y)$. As the Hessian is $\rho$-Lipschitz, we have
\begin{eqnarray}
\begin{aligned}\label{sec-smooth2}
f(\x^{\K})-f(\x^{0})\leq&\<\nabla f(\x^{0}),\x^{\K}-\x^{0}\>+\frac{1}{2}(\x^{\K}-\x^{0})^T\H(\x^{\K}-\x^{0})+\frac{\rho}{6}\|\x^{\K}-\x^{0}\|^3\\
=&\<\widetilde\nabla f(\x^{0}),\widetilde\x^{\K}-\widetilde\x^{0}\>+\frac{1}{2}(\widetilde\x^{\K}-\widetilde\x^{0})^T\bLambda(\widetilde\x^{\K}-\widetilde\x^{0})+\frac{\rho}{6}\|\x^{\K}-\x^{0}\|^3\\
\leq&g(\widetilde\x^{\K})-g(\widetilde\x^{0})+4.5\rho B^3,
\end{aligned}
\end{eqnarray}
where we denote
\begin{eqnarray}
\begin{aligned}\label{def_g}
&g(\x)=\<\widetilde\nabla f(\x^{0}),\x-\widetilde\x^{0}\>+\frac{1}{2}(\x-\widetilde\x^{0})^T\bLambda(\x-\widetilde\x^{0})=\sum_{j=1}^dg_j(\x_j),\\
&g_j(x)=\<\widetilde\nabla_j f(\x^{0}),x-\widetilde\x_j^{0}\>+\frac{1}{2}\blambda_j(x-\widetilde\x_j^{0})^2.
\end{aligned}
\end{eqnarray}
Denoting
\begin{eqnarray}
\begin{aligned}\notag
\widetilde\bdelta_j^{k}=\widetilde\nabla_j f(\y^{k})-\nabla g_j(\widetilde\y_j^{k}),\qquad\widetilde\bdelta^{k}=\widetilde\nabla f(\y^{k})-\nabla g(\widetilde\y^{k}),
\end{aligned}
\end{eqnarray}
then the AGD iterations in Algorithm \ref{AGD1} can be rewritten as
\begin{subequations}
\begin{align}
&\widetilde\y_j^{k}=\widetilde\x_j^{k}+(1-\theta)(\widetilde\x_j^{k}-\widetilde\x_j^{k-1}),\label{s1}\\
&\widetilde\x_j^{k+1}=\widetilde\y_j^{k}-\eta\widetilde\nabla_j f(\y^{k})=\widetilde\y_j^{k}-\eta\nabla g_j(\widetilde\y_j^{k})-\eta\widetilde\bdelta_j^{k},\label{s2}
\end{align}
\end{subequations}
and $\|\widetilde\bdelta^{k}\|$ can be bounded as
\begin{eqnarray}
\begin{aligned}\label{sec-smooth}
\|\widetilde\bdelta^{k}\|=&\|\widetilde\nabla f(\y^{k})-\widetilde\nabla f(\x^{0})-\bLambda(\widetilde\y^{k}-\widetilde\x^{0})\|\\
=&\|\nabla f(\y^{k})-\nabla f(\x^{0})-\H(\y^{k}-\x^{0})\|\\
=&\left\|\left(\int_0^1\nabla^2 f(\x^{0}+t(\y^{k}-\x^{0}))-\H\right)(\y^{k}-\x^{0})dt\right\|\\
\leq&\frac{\rho}{2}\|\y^{k}-\x^{0}\|^2\leq 2\rho B^2
\end{aligned}
\end{eqnarray}
for any $k<\K$, where we use the $\rho$-Lipschitz Hessian assumption and (\ref{cond0}) in the last two inequalities, respectively.

Thanks to (\ref{sec-smooth2}), to prove the decrease from $f(\x^{0})$ to $f(\x^{\K})$, we only need to study the decrease of $g(\x)$. Iterations (\ref{s1}) and (\ref{s2}) can be regarded as applying AGD to the quadratic approximation $g(\x)$ coordinately with the approximation error $\widetilde\bdelta^{k}$, where the later can be controlled within $\bO(\rho B^2)$. The quadratic approximation $g(\x)$ equals to the sum of $d$ scalar functions $g_j(\x_j)$. We decompose $g(\x)$ into $\sum_{j\in\S_1} g_j(\x_j)$ and $\sum_{j\in\S_2} g_j(\x_j)$, where
\begin{equation}
\S_1=\left\{j:\blambda_j\geq-\frac{\theta}{\eta}\right\}\quad\mbox{and}\quad \S_2=\left\{j:\blambda_j<-\frac{\theta}{\eta}\right\}.\notag
\end{equation}
We see that $g_j(x)$ is approximate convex when $j\in\S_1$, and strongly concave when $j\in\S_2$. We will prove the approximate decrease of $g_j(\x_j)$ in the above two cases. We first consider $\sum_{j\in\S_1} g_j(\x_j)$ in the following lemma.
\begin{lemma}\label{lemma2}
Suppose that Assumption \ref{assump} holds. Let $\eta\leq\frac{1}{4L}$ and $0<\theta\leq 1$. In each epoch of Algorithm \ref{AGD1} where the ``if condition" triggers, when $\|\nabla f(\y^{\K-1})\|\leq \frac{B}{\eta}$, we have
\begin{eqnarray}
\begin{aligned}\label{cont2}
\sum_{j\in\S_1}g_j(\widetilde\x_j^{\K})-\sum_{j\in\S_1}g_j(\widetilde\x_j^{0})\leq -\sum_{j\in\S_1}\frac{3\theta}{8\eta}\sum_{k=0}^{\K-1}|\widetilde\x_j^{k+1}-\widetilde\x_j^{k}|^2+\frac{8\eta\rho^2B^4\K}{\theta}.
\end{aligned}
\end{eqnarray}
\end{lemma}
\begin{proof}
Since $g_j(x)$ is quadratic, we have
\begin{eqnarray}
\begin{aligned}\notag
g_j(\widetilde\x_j^{k+1})=& g_j(\widetilde\x_j^{k})+\<\nabla g_j(\widetilde\x_j^{k}),\widetilde\x_j^{k+1}-\widetilde\x_j^{k}\>+\frac{\blambda_j}{2}|\widetilde\x_j^{k+1}-\widetilde\x_j^{k}|^2\\
\overset{a}=& g_j(\widetilde\x_j^{k})-\frac{1}{\eta}\<\widetilde\x_j^{k+1}-\widetilde\y_j^{k}+\eta\widetilde\bdelta_j^{k},\widetilde\x_j^{k+1}-\widetilde\x_j^{k}\>\\
&+\<\nabla g_j(\widetilde\x_j^{k})-\nabla g_j(\widetilde\y_j^{k}),\widetilde\x_j^{k+1}-\widetilde\x_j^{k}\>+\frac{\blambda_j}{2}|\widetilde\x_j^{k+1}-\widetilde\x_j^{k}|^2\\
=& g_j(\widetilde\x_j^{k})-\frac{1}{\eta}\<\widetilde\x_j^{k+1}-\widetilde\y_j^{k},\widetilde\x_j^{k+1}-\widetilde\x_j^{k}\>-\<\widetilde\bdelta_j^{k},\widetilde\x_j^{k+1}-\widetilde\x_j^{k}\>\\
&+\blambda_j\<\widetilde\x_j^{k}-\widetilde\y_j^{k},\widetilde\x_j^{k+1}-\widetilde\x_j^{k}\>+\frac{\blambda_j}{2}|\widetilde\x_j^{k+1}-\widetilde\x_j^{k}|^2\\
=& g_j(\widetilde\x_j^{k})+\frac{1}{2\eta}\left(|\widetilde\x_j^{k}-\widetilde\y_j^{k}|^2-|\widetilde\x_j^{k+1}-\widetilde\y_j^{k}|^2-|\widetilde\x_j^{k+1}-\widetilde\x_j^{k}|^2\right)\\
&-\<\widetilde\bdelta_j^{k},\widetilde\x_j^{k+1}-\widetilde\x_j^{k}\>+\frac{\blambda_j}{2}\left(|\widetilde\x_j^{k+1}-\widetilde\y_j^{k}|^2-|\widetilde\x_j^{k}-\widetilde\y_j^{k}|^2\right)\\
\leq& g_j(\widetilde\x_j^{k})+\frac{1}{2\eta}\left(|\widetilde\x_j^{k}-\widetilde\y_j^{k}|^2-|\widetilde\x_j^{k+1}-\widetilde\y_j^{k}|^2-|\widetilde\x_j^{k+1}-\widetilde\x_j^{k}|^2\right)\\
&+\frac{1}{2\alpha}|\widetilde\bdelta_j^{k}|^2+\frac{\alpha}{2}|\widetilde\x_j^{k+1}-\widetilde\x_j^{k}|^2+\frac{\blambda_j}{2}\left(|\widetilde\x_j^{k+1}-\widetilde\y_j^{k}|^2-|\widetilde\x_j^{k}-\widetilde\y_j^{k}|^2\right),
\end{aligned}
\end{eqnarray}
for some positive constant $\alpha$ to be specified later, where we use (\ref{s2}) in $\overset{a}=$. Using $L\geq\blambda_j\geq-\frac{\theta}{\eta}$ when $j\in\S_1=\{j:\blambda_j\geq-\frac{\theta}{\eta}\}$ and $\left(-\frac{1}{2\eta}+\frac{\blambda_j}{2}\right)|\widetilde\x_j^{k+1}-\widetilde\y_j^{k}|^2\leq \left(-2L+\frac{L}{2}\right)|\widetilde\x_j^{k+1}-\widetilde\y_j^{k}|^2\leq 0$, we have for each $j\in\S_1$,
\begin{eqnarray}
\begin{aligned}\notag
g_j(\widetilde\x_j^{k+1})\leq& g_j(\widetilde\x_j^{k})+\frac{1}{2\eta}\left(|\widetilde\x_j^{k}-\widetilde\y_j^{k}|^2-|\widetilde\x_j^{k+1}-\widetilde\x_j^{k}|^2\right)+\frac{1}{2\alpha}|\widetilde\bdelta_j^{k}|^2+\frac{\alpha}{2}|\widetilde\x_j^{k+1}-\widetilde\x_j^{k}|^2+\frac{\theta}{2\eta}|\widetilde\x_j^{k}-\widetilde\y_j^{k}|^2\\
\overset{b}=& g_j(\widetilde\x_j^{k})+\frac{(1+\theta)(1-\theta)^2}{2\eta}|\widetilde\x_j^{k}-\widetilde\x_j^{k-1}|^2-\left(\frac{1}{2\eta}-\frac{\alpha}{2}\right)|\widetilde\x_j^{k+1}-\widetilde\x_j^{k}|^2+\frac{1}{2\alpha}|\widetilde\bdelta_j^{k}|^2,
\end{aligned}
\end{eqnarray}
where we use (\ref{s1}) in $\overset{b}=$. Defining the potential function
\begin{eqnarray}
\begin{aligned}\notag
\bell_j^{k+1}=g_j(\widetilde\x_j^{k+1})+\frac{(1+\theta)(1-\theta)^2}{2\eta}|\widetilde\x_j^{k+1}-\widetilde\x_j^{k}|^2,
\end{aligned}
\end{eqnarray}
we have
\begin{eqnarray}
\begin{aligned}\notag
\bell_j^{k+1}\leq& \bell_j^k-\left(\frac{1}{2\eta}-\frac{\alpha}{2}-\frac{(1+\theta)(1-\theta)^2}{2\eta}\right)|\widetilde\x_j^{k+1}-\widetilde\x_j^{k}|^2+\frac{1}{2\alpha}|\widetilde\bdelta_j^{k}|^2\\
\overset{c}\leq& \bell_j^k-\frac{3\theta}{8\eta}|\widetilde\x_j^{k+1}-\widetilde\x_j^{k}|^2+\frac{2\eta}{\theta}|\widetilde\bdelta_j^{k}|^2,
\end{aligned}
\end{eqnarray}
where we let $\alpha=\frac{\theta}{4\eta}$ in $\overset{c}\leq$ such that $\frac{1}{2\eta}-\frac{\theta}{8\eta}-\frac{(1+\theta)(1-\theta)^2}{2\eta}=\frac{3\theta}{8\eta}+\frac{\theta^2}{2\eta}-\frac{\theta^3}{2\eta}\geq\frac{3\theta}{8\eta} $. Summing over $k=0,1,\cdots,\K-1$ and $j\in\S_1$, using $\x^{0}-\x^{-1}=0$, we have
\begin{eqnarray}
\begin{aligned}\notag
\sum_{j\in\S_1}g_j(\widetilde\x_j^{\K})\leq\sum_{j\in\S_1}\bell_j^{\K}\leq& \sum_{j\in\S_1}g_j(\widetilde\x_j^{0})-\sum_{j\in\S_1}\frac{3\theta}{8\eta}\sum_{k=0}^{\K-1}|\widetilde\x_j^{k+1}-\widetilde\x_j^{k}|^2+\frac{2\eta}{\theta}\sum_{k=0}^{\K-1}\|\widetilde\bdelta^{k}\|^2\\
\overset{d}\leq& \sum_{j\in\S_1}g_j(\widetilde\x_j^{0})-\sum_{j\in\S_1}\frac{3\theta}{8\eta}\sum_{k=0}^{\K-1}|\widetilde\x_j^{k+1}-\widetilde\x_j^{k}|^2+\frac{8\eta\rho^2B^4\K}{\theta},
\end{aligned}
\end{eqnarray}
where we use (\ref{sec-smooth}) in $\overset{d}\leq$.
\end{proof}

Next, we consider $\sum_{j\in\S_2} g_j(\x_j)$.
\begin{lemma}\label{lemma3}
Suppose that Assumption \ref{assump} holds. Let $\eta\leq\frac{1}{4L}$ and $0<\theta\leq 1$. In each epoch of Algorithm \ref{AGD1} where the ``if condition" triggers, when $\|\nabla f(\y^{\K-1})\|\leq \frac{B}{\eta}$, we have
\begin{eqnarray}
\begin{aligned}\label{cont3}
\sum_{j\in\S_2}g_j(\widetilde\x_j^{\K})-\sum_{j\in\S_2}g_j(\widetilde\x_j^{0})\leq -\sum_{j\in\S_2}\frac{\theta}{2\eta}\sum_{k=0}^{\K-1}|\widetilde\x_j^{k+1}-\widetilde\x_j^{k}|^2+\frac{2\eta\rho^2B^4\K}{\theta}.
\end{aligned}
\end{eqnarray}
\end{lemma}
\begin{proof}
Denoting $\v_j=\widetilde\x_j^{0}-\frac{1}{\blambda_j}\widetilde\nabla_j f(\x^{0})$, $g_j(x)$ can be rewritten as
\begin{eqnarray}
\begin{aligned}\notag
g_j(x)=&\frac{\blambda_j}{2}\left(x-\widetilde\x_j^{0}+\frac{1}{\blambda_j}\widetilde\nabla_j f(\x^{0})\right)^2-\frac{1}{2\blambda_j}|\widetilde\nabla_j f(\x^{0})|^2\\
=&\frac{\blambda_j}{2}|x-\v_j|^2-\frac{1}{2\blambda_j}|\widetilde\nabla_j f(\x^{0})|^2.
\end{aligned}
\end{eqnarray}
For each $j\in\S_2=\{j:\blambda_j<-\frac{\theta}{\eta}\}$, we have
\begin{eqnarray}
\begin{aligned}\label{cont1}
g_j(\widetilde\x_j^{k+1})-g_j(\widetilde\x_j^{k})=&\frac{\blambda_j}{2}|\widetilde\x_j^{k+1}-\v_j|^2-\frac{\blambda_j}{2}|\widetilde\x_j^{k}-\v_j|^2\\
=&\frac{\blambda_j}{2}|\widetilde\x_j^{k+1}-\widetilde\x_j^{k}|^2+\blambda_j\<\widetilde\x_j^{k+1}-\widetilde\x_j^{k},\widetilde\x_j^{k}-\v_j\>\\
\leq&-\frac{\theta}{2\eta}|\widetilde\x_j^{k+1}-\widetilde\x_j^{k}|^2+\blambda_j\<\widetilde\x_j^{k+1}-\widetilde\x_j^{k},\widetilde\x_j^{k}-\v_j\>.
\end{aligned}
\end{eqnarray}
So we only need to bound the second term. From (\ref{s2}) and (\ref{s1}), we have
\begin{eqnarray}
\begin{aligned}\notag
\widetilde\x_j^{k+1}-\widetilde\x_j^{k}=&\widetilde\y_j^{k}-\widetilde\x_j^{k}-\eta\nabla g_j(\widetilde\y_j^{k})-\eta\widetilde\bdelta_j^{k}\\
=&(1-\theta)(\widetilde\x_j^{k}-\widetilde\x_j^{k-1})-\eta\nabla g_j(\widetilde\y_j^{k})-\eta\widetilde\bdelta_j^{k}\\
=&(1-\theta)(\widetilde\x_j^{k}-\widetilde\x_j^{k-1})-\eta\blambda_j(\widetilde\y_j^{k}-\v_j)-\eta\widetilde\bdelta_j^{k}\\
=&(1-\theta)(\widetilde\x_j^{k}-\widetilde\x_j^{k-1})-\eta\blambda_j(\widetilde\x_j^{k}-\v_j+(1-\theta)(\widetilde\x_j^{k}-\widetilde\x_j^{k-1}))-\eta\widetilde\bdelta_j^{k}.
\end{aligned}
\end{eqnarray}
So for each $j\in\S_2$, we have
\begin{eqnarray}
\begin{aligned}\notag
&\blambda_j\<\widetilde\x_j^{k+1}-\widetilde\x_j^{k},\widetilde\x_j^{k}-\v_j\>\\
=&(1-\theta)\blambda_j\<\widetilde\x_j^{k}-\widetilde\x_j^{k-1},\widetilde\x_j^{k}-\v_j\>-\eta\blambda_j^2|\widetilde\x_j^{k}-\v_j|^2-\eta\blambda_j^2(1-\theta)\<\widetilde\x_j^{k}-\widetilde\x_j^{k-1},\widetilde\x_j^{k}-\v_j\>-\eta\blambda_j\<\widetilde\bdelta_j^{k},\widetilde\x_j^{k}-\v_j\>\\
\leq&(1-\theta)\blambda_j\<\widetilde\x_j^{k}-\widetilde\x_j^{k-1},\widetilde\x_j^{k}-\v_j\>-\eta\blambda_j^2|\widetilde\x_j^{k}-\v_j|^2\\
&+\frac{\eta\blambda_j^2(1-\theta)}{2}\left(|\widetilde\x_j^{k}-\widetilde\x_j^{k-1}|^2+|\widetilde\x_j^{k}-\v_j|^2\right)+\frac{\eta}{2(1+\theta)}|\widetilde\bdelta_j^{k}|^2+\frac{\eta\blambda_j^2(1+\theta)}{2}|\widetilde\x_j^{k}-\v_j|^2\\
=&(1-\theta)\blambda_j\<\widetilde\x_j^{k}-\widetilde\x_j^{k-1},\widetilde\x_j^{k}-\v_j\>+\frac{\eta\blambda_j^2(1-\theta)}{2}|\widetilde\x_j^{k}-\widetilde\x_j^{k-1}|^2+\frac{\eta}{2(1+\theta)}|\widetilde\bdelta_j^{k}|^2\\
=&(1-\theta)\blambda_j\<\widetilde\x_j^{k}-\widetilde\x_j^{k-1},\widetilde\x_j^{k-1}-\v_j\>+(1-\theta)\blambda_j|\widetilde\x_j^{k}-\widetilde\x_j^{k-1}|^2+\frac{\eta\blambda_j^2(1-\theta)}{2}|\widetilde\x_j^{k}-\widetilde\x_j^{k-1}|^2+\frac{\eta}{2(1+\theta)}|\widetilde\bdelta_j^{k}|^2\\
\overset{a}\leq&(1-\theta)\blambda_j\<\widetilde\x_j^{k}-\widetilde\x_j^{k-1},\widetilde\x_j^{k-1}-\v_j\>+\frac{\eta}{2}|\widetilde\bdelta_j^{k}|^2,
\end{aligned}
\end{eqnarray}
where we use $\left(1+\frac{\eta\blambda_j}{2}\right)(1-\theta)\geq \left(1-\frac{\eta L}{2}\right)(1-\theta)\geq 0$ and $\blambda_j<0$ when $j\in\S_2$ in $\overset{a}\geq$. So we have
\begin{eqnarray}
\begin{aligned}\notag
\blambda_j\<\widetilde\x_j^{k+1}-\widetilde\x_j^{k},\widetilde\x_j^{k}-\v_j\>\leq&(1-\theta)^k\blambda_j\<\widetilde\x_j^{1}-\widetilde\x_j^{0},\widetilde\x_j^{0}-\v_j\>+\frac{\eta}{2}\sum_{t=1}^k(1-\theta)^{k-t}|\widetilde\bdelta_j^{t}|^2\\
\overset{b}=&-(1-\theta)^k\eta\blambda_j^2|\widetilde\x_j^{0}-\v_j|^2+\frac{\eta}{2}\sum_{t=1}^k(1-\theta)^{k-t}|\widetilde\bdelta_j^{t}|^2\\
\leq&\frac{\eta}{2}\sum_{t=1}^k(1-\theta)^{k-t}|\widetilde\bdelta_j^{t}|^2,
\end{aligned}
\end{eqnarray}
where we use
\begin{eqnarray}
\begin{aligned}\notag
\widetilde\x_j^{1}-\widetilde\x_j^{0}=&\widetilde\x_j^{1}-\widetilde\y_j^{0}=-\eta\widetilde\nabla_j f(\y^{0})=-\eta\widetilde\nabla_j f(\x^{0})\\
=&-\eta\nabla g_j(\widetilde\x_j^{0})=-\eta\blambda_j(\widetilde\x_j^{0}-\v_j)
\end{aligned}
\end{eqnarray}
in $\overset{b}=$. Plugging into (\ref{cont1}), we have
\begin{eqnarray}
\begin{aligned}\notag
&g_j(\widetilde\x_j^{k+1})-g_j(\widetilde\x_j^{k})\leq-\frac{\theta}{2\eta}|\widetilde\x_j^{k+1}-\widetilde\x_j^{k}|^2+\frac{\eta}{2}\sum_{t=1}^k(1-\theta)^{k-t}|\widetilde\bdelta_j^{t}|^2.
\end{aligned}
\end{eqnarray}
Summing over $k=0,1,\cdots,\K-1$ and $j\in\S_2$, we have
\begin{eqnarray}
\begin{aligned}\notag
\sum_{j\in\S_2}g_j(\widetilde\x_j^{\K})-\sum_{j\in\S_2}g_j(\widetilde\x_j^{0})\leq&-\sum_{j\in\S_2}\frac{\theta}{2\eta}\sum_{k=0}^{\K-1}|\widetilde\x_j^{k+1}-\widetilde\x_j^{k}|^2+\frac{\eta}{2}\sum_{k=0}^{\K-1}\sum_{t=1}^k(1-\theta)^{k-t}\|\widetilde\bdelta^{t}\|^2\\
\overset{c}\leq& -\sum_{j\in\S_2}\frac{\theta}{2\eta}\sum_{k=0}^{\K-1}|\widetilde\x_j^{k+1}-\widetilde\x_j^{k}|^2+2\eta\rho^2B^4\sum_{k=0}^{\K-1}\sum_{t=1}^k(1-\theta)^{k-t}\\
\leq& -\sum_{j\in\S_2}\frac{\theta}{2\eta}\sum_{k=0}^{\K-1}|\widetilde\x_j^{k+1}-\widetilde\x_j^{k}|^2+\frac{2\eta\rho^2B^4\K}{\theta},
\end{aligned}
\end{eqnarray}
where we use (\ref{sec-smooth}) in $\overset{c}\leq$.
\end{proof}
Putting Lemmas \ref{lemma2} and \ref{lemma3} together, we have the following lemma.
\begin{lemma}\label{lemma4}
Suppose that Assumption \ref{assump} holds. Let $\eta\leq\frac{1}{4L}$ and $0<\theta\leq 1$. In each epoch of Algorithm \ref{AGD1} where the ``if condition" triggers, when $\|\nabla f(\y^{\K-1})\|\leq \frac{B}{\eta}$, we have
\begin{eqnarray}
\begin{aligned}\label{cont9}
f(\x^{\K})-f(\x^{0})\leq -\frac{3\theta B^2}{8\eta K}+\frac{10\eta\rho^2B^4K}{\theta}+4.5\rho B^3.
\end{aligned}
\end{eqnarray}
\end{lemma}
\begin{proof}
Summing over (\ref{cont2}) and (\ref{cont3}), we have
\begin{eqnarray}
\begin{aligned}\notag
g(\widetilde\x^{\K})-g(\widetilde\x^{0})=&\sum_{j\in\S_1\cup\S_2}g_j(\widetilde\x_j^{\K})-g_j(\widetilde\x_j^{0})\\
\leq& -\frac{3\theta}{8\eta}\sum_{k=0}^{\K-1}\|\widetilde\x^{k+1}-\widetilde\x^{k}\|^2+\frac{10\eta\rho^2B^4\K}{\theta}\\
=& -\frac{3\theta}{8\eta}\sum_{k=0}^{\K-1}\|\x^{k+1}-\x^{k}\|^2+\frac{10\eta\rho^2B^4\K}{\theta}\\
\overset{a}\leq& -\frac{3\theta B^2}{8\eta\K}+\frac{10\eta\rho^2B^4\K}{\theta},
\end{aligned}
\end{eqnarray}
where we use (\ref{cond1}) in $\overset{a}\leq$. Plugging into (\ref{sec-smooth2}) and using $\K\leq K$, we have
\begin{eqnarray}
\begin{aligned}\notag
f(\x^{\K})-f(\x^{0})\leq&-\frac{3\theta B^2}{8\eta\K}+\frac{10\eta\rho^2B^4\K}{\theta}+4.5\rho B^3\\
\leq&-\frac{3\theta B^2}{8\eta K}+\frac{10\eta\rho^2B^4K}{\theta}+4.5\rho B^3.
\end{aligned}
\end{eqnarray}
\end{proof}
From Lemmas \ref{lemma1} and \ref{lemma4}, we can establish the decrease of $f(\x)$ in each epoch.
\begin{corollary}\label{lemma5}
Suppose that Assumption \ref{assump} holds. Use the parameter settings in Theorem \ref{theorem1}. In each epoch of Algorithm \ref{AGD1} where the ``if condition" triggers, we have
\begin{eqnarray}
\begin{aligned}\label{cont8}
f(\x^{\K})-f(\x^{0})\leq -\frac{7\epsilon^{3/2}}{8\sqrt{\rho}}.
\end{aligned}
\end{eqnarray}
\end{corollary}
\begin{proof}
Combing Lemmas \ref{lemma1} and \ref{lemma4} and using the parameter settings, we have
\begin{eqnarray}
\begin{aligned}\notag
f(\x^{\K})-f(\x^{0})\leq -\min\left\{\frac{3\theta B^2}{8\eta K}-\frac{10\eta\rho^2B^4K}{\theta}-4.5\rho B^3,\frac{B^2}{4\eta}\right\}=-\min\left\{\frac{7\epsilon^{3/2}}{8\sqrt{\rho}},\frac{\epsilon}{4\eta\rho}\right\}.
\end{aligned}
\end{eqnarray}
From $\theta=4(\epsilon\rho\eta^2)^{1/4}\leq 1$, we have $\frac{7\epsilon^{3/2}}{8\sqrt{\rho}}\leq\frac{\epsilon}{4\eta\rho}$.
\end{proof}

\subsubsection{Small Gradient in the Last Epoch}\label{subsec3}
We first give the following lemma for the last epoch.
\begin{lemma}\label{lemma6}
Suppose that Assumption \ref{assump} holds. Use the parameter settings in Theorem \ref{theorem1}. In the last epoch of Algorithm \ref{AGD1} where the ``if condition" does not trigger, we have $\|\nabla f(\hat\y)\|\leq82\epsilon$.
\end{lemma}
\begin{proof}
Denote $\widetilde\y=\U^T\hat\y=\frac{1}{K_0+1}\sum_{k=0}^{K_0}\U^T\y^{k}=\frac{1}{K_0+1}\sum_{k=0}^{K_0}\widetilde\y^{k}$. Since $g$ is quadratic, we have
\begin{align}\label{cont10}
\|\nabla g(\widetilde\y)\|=&\left\|\frac{1}{K_0+1}\sum_{k=0}^{K_0}\nabla g(\widetilde\y^{k})\right\|\nonumber\\
\overset{a}=&\frac{1}{\eta (K_0+1)}\left\|\sum_{k=0}^{K_0}\left(\widetilde\x^{k+1}-\widetilde\y^k+\eta\widetilde\bdelta^{k}\right)\right\|\nonumber\\
=&\frac{1}{\eta (K_0+1)}\left\|\sum_{k=0}^{K_0}\left(\widetilde\x^{k+1}-\widetilde\x^k-(1-\theta)(\widetilde\x^{k}-\widetilde\x^{k-1})+\eta\widetilde\bdelta^{k}\right)\right\|\nonumber\\
\overset{b}=&\frac{1}{\eta (K_0+1)}\left\|\widetilde\x^{K_0+1}-\widetilde\x^0-(1-\theta)(\widetilde\x^{K_0}-\widetilde\x^{0})+\eta\sum_{k=0}^{K_0}\widetilde\bdelta^{k}\right\|\nonumber\\
=&\frac{1}{\eta (K_0+1)}\left\|\widetilde\x^{K_0+1}-\widetilde\x^{K_0}+\theta(\widetilde\x^{K_0}-\widetilde\x^{0})+\eta\sum_{k=0}^{K_0}\widetilde\bdelta^{k}\right\|\nonumber\\
\leq&\frac{1}{\eta(K_0+1)}\left(\|\widetilde\x^{K_0+1}-\widetilde\x^{K_0}\|+\theta\|\widetilde\x^{K_0}-\widetilde\x^{0}\|+\eta\sum_{k=0}^{K_0}\|\widetilde\bdelta^{k}\|\right)\nonumber\\
\overset{c}\leq&\frac{2}{\eta K}\|\widetilde\x^{K_0+1}-\widetilde\x^{K_0}\|+\frac{2\theta B}{\eta K}+2\rho B^2,
\end{align}
where we use (\ref{s2}) in $\overset{a}=$, $\x^{-1}=\x^{0}$ in $\overset{b}=$, $K_0+1\geq\frac{K}{2}$, (\ref{cond3}), (\ref{sec-smooth}), and (\ref{cond4}) in $\overset{c}\leq$. From $K_0=\argmin_{\lfloor \frac{K}{2}\rfloor\leq k\leq K-1}\|\x^{k+1}-\x^{k}\|$, we have
\begin{eqnarray}
\begin{aligned}\label{cont7}
\|\x^{K_0+1}-\x^{K_0}\|^2\leq& \frac{1}{K-\lfloor K/2\rfloor}\sum_{k=\lfloor K/2\rfloor}^{K-1}\|\x^{k+1}-\x^{k}\|^2\\
\leq& \frac{1}{K-\lfloor K/2\rfloor}\sum_{k=0}^{K-1}\|\x^{k+1}-\x^{k}\|^2\\
\overset{d}\leq&\frac{1}{K-\lfloor K/2\rfloor}\frac{B^2}{K}\leq\frac{2B^2}{K^2},
\end{aligned}
\end{eqnarray}
where we use (\ref{cond3}) in $\overset{d}\leq$. On the other hand, we also have
\begin{eqnarray}
\begin{aligned}\notag
\|\nabla f(\hat\y)\|=\|\widetilde\nabla f(\hat\y)\|\leq& \|\nabla g(\widetilde\y)\|+\|\widetilde\nabla f(\hat\y)-\nabla g(\widetilde\y)\|\\
=&\|\nabla g(\widetilde\y)\|+\|\widetilde\nabla f(\hat\y)-\widetilde\nabla f(\x^{0})-\bLambda(\widetilde\y-\widetilde\x^{0})\|\\
=&\|\nabla g(\widetilde\y)\|+\|\nabla f(\hat\y)-\nabla f(\x^{0})-\H(\hat\y-\x^{0})\|\\
\leq&\|\nabla g(\widetilde\y)\|+\frac{\rho}{2}\|\hat\y-\x^{0}\|^2\overset{e}\leq \|\nabla g(\widetilde\y)\|+2\rho B^2,
\end{aligned}
\end{eqnarray}
where we use $\|\hat\y-\x^{0}\|\leq\frac{1}{K_0+1}\sum_{k=0}^{K_0}\|\y^{k}-\x^{0}\|\leq 2B$ from (\ref{cond4}) in $\overset{e}\leq$. So we have
\begin{eqnarray}
\begin{aligned}\notag
\|\nabla f(\hat\y)\|\leq\frac{2\sqrt{2}B}{\eta K^2}+\frac{2\theta B}{\eta K}+4\rho B^2\leq82\epsilon.
\end{aligned}
\end{eqnarray}
\end{proof}
Based Corollary \ref{lemma5} and Lemma \ref{lemma6}, we can prove Theorem \ref{theorem1}.
\begin{proof}
For each epoch where the ``if condition" triggers, we have (\ref{cont8}). Note that at the beginning of each epoch, we set $\x^0$ to be the last iterate $\x^{\K}$ in the previous epoch. Summing (\ref{cont8}) over all epochs, say $N$ total epochs, and using $\min_{\x}f(\x)\leq f(\x^{\K})$, we have
\begin{eqnarray}
\begin{aligned}\notag
\min_{\x}f(\x)-f(\x_{int})\leq -N\frac{7\epsilon^{3/2}}{8\sqrt{\rho}}.
\end{aligned}
\end{eqnarray}
So the algorithm will terminate (that is, the ``if condition" does not trigger and the while loop breaks) in at most $\frac{8\triangle_f\sqrt{\rho}}{7\epsilon^{3/2}}$ epochs. Since each epoch needs at most $K=\frac{1}{2}\left(\frac{L^2}{\epsilon\rho}\right)^{1/4}$ gradient evaluations, the total number of gradient evaluations must be less than $\frac{\triangle_fL^{1/2}\rho^{1/4}}{\epsilon^{7/4}}$.
On the other hand, in the last epoch, we have $\|\nabla f(\hat\y)\|\leq82\epsilon$ from Lemma \ref{lemma6}.
\end{proof}
\begin{remark}
The purpose of using the specific average as the output and $k\sum_{t=0}^{k-1}\|\x^{t+1}-\x^{t}\|^2> B^2$ in the ``if condition" in Algorithm \ref{AGD1}, rather than $\|\x^k-\x^0\|\geq B$, is to establish (\ref{cont7}).
\end{remark}

\subsection{Discussion on the Acceleration Mechanism}\label{sec-comp}
When we replace the AGD iterations in Algorithm \ref{AGD1} by the gradient descent steps $\x^{k+1}=\x^k-\eta\nabla f(\x^k)$ with step-size $\eta=\frac{1}{4L}$, similar to (\ref{cont11}), the descent property in each epoch becomes
\begin{eqnarray}
\begin{aligned}\notag
f(\x^{\K})-f(\x^0)\leq -\frac{7}{8\eta}\sum_{k=0}^{\K-1}\|\x^{k+1}-\x^k\|^2\leq -\frac{7B^2}{8\eta\K},
\end{aligned}
\end{eqnarray}
and the gradient norm at the averaged output $\hat\x=\frac{1}{K}\sum_{k=0}^{K-1}\x^k$ can be bounded as
\begin{eqnarray}
\begin{aligned}\notag
\|\nabla g(\hat\x)\|\leq \frac{1}{\eta K}\|\x^{K}-\x^0\|+2\rho B^2\leq \frac{B}{\eta K}+2\rho B^2.
\end{aligned}
\end{eqnarray}
By setting $B=\sqrt{\frac{\epsilon}{\rho}}$ and $K=\frac{L}{\sqrt{\epsilon\rho}}$, we have the $\bO(\epsilon^{-2})$ total complexity.

Comparing the above two inequalities with (\ref{cont9}) and (\ref{cont10}), respectively, we see that the momentum parameter $\theta$ is crucial to speedup the convergence of AGD because it allows smaller $K$ than that of GD, that is, $\frac{1}{\epsilon^{1/4}}$ v.s. $\frac{1}{\epsilon^{1/2}}$ for AGD and GD, respectively. Accordingly, smaller $K$ results in less total gradient evaluations since both methods need $\bO(\epsilon^{-1.5})$ epochs. The above comparisons show the importance of momentum and its parameter $\theta$ in the nonconvex acceleration mechanism. On the other hand, since our proofs do not invoke the analysis of strongly convex AGD, we conjecture that the nonconvex acceleration mechanism seems irrelevant to the analysis of convex AGD.

\subsection{Proof of Theorem \ref{theorem1p}}\label{sec:proof_t1p}
In this section, we prove a stronger theorem, where we replace lines 11 and 8 of Algorithm \ref{AGD1p} by (\ref{agdp-cont1}) and $f(\x^{\K})-f(\x^{0})\leq -\gamma\frac{\epsilon^{3/2}}{\sqrt{\rho'}}$, respectively. Denote $\eta_{int}$, $\rho_{int}'$, and $B_{0,int}$ to be the initializations of $\eta$, $\rho'$, and $B_0$, respectively.
\begin{theorem}\label{theorem1pp}
Suppose that Assumption \ref{assump} holds. Let $B=\sqrt{\frac{\epsilon}{\rho'}}$, $\theta=4(\epsilon\rho'\eta^2)^{1/4}\in(0,1)$, and $K=\lfloor\frac{1}{\theta}\rfloor$ in each epoch, where $\eta$ and $\rho'$ may change dynamically during epochs. Let $\gamma\leq\frac{7}{8}$, $c_0>1$, $c_1\geq c_2>1$, $\eta_{min}\leq\frac{1}{4L}$, and $\rho_{max}'\geq\rho$, where $\frac{\rho_{max}'}{\rho_{int}'}=\left(\frac{\eta_{int}}{\eta_{min}}\right)^2$. Then Algorithm \ref{AGD1p} with (\ref{agdp-cont1}) terminates in at most $\bO (\epsilon^{-7/4})$ gradient computations and $\bO (\epsilon^{-3/2})$ function evaluations, and the output satisfies $\|\nabla f(\x_{out})\|\leq\bO(\epsilon)$.
\end{theorem}

Using the same proofs of Corollary \ref{lemma5} and Lemma \ref{lemma6}, we have the following two straight-forward corollaries.
\begin{corollary}\label{lemmap1}
Suppose that Assumption \ref{assump} holds. Let $B=\sqrt{\frac{\epsilon}{\rho'}}$, $\theta=4\left(\epsilon\rho'\eta^2\right)^{1/4}\in(0,1]$, and $K=\lfloor\frac{1}{\theta}\rfloor$ in one epoch, where $\eta\leq \frac{1}{4L}$ and $\rho'\geq\rho$. Assume that $\K\sum_{t=0}^{\K-1}\|\x^{t+1}-\x^{t}\|^2> B^2$ for some $\K\leq K$ and $k\sum_{t=0}^{k-1}\|\x^{t+1}-\x^{t}\|^2\leq B^2$ for all $k<\K$, then for the iterations
\begin{eqnarray}
\begin{aligned}\label{lemmapite}
\y^{k}=\x^{k}+(1-\theta)(\x^{k}-\x^{k-1}),\quad \x^{k+1}=\y^{k}-\eta\nabla f(\y^{k})
\end{aligned}
\end{eqnarray}
starting from $\x^0=\x^{-1}$, we have
\begin{eqnarray}
\begin{aligned}\label{lemmapdec}
f(\x^{\K})-f(\x^{0})\leq -\frac{7\epsilon^{3/2}}{8\sqrt{\rho'}}.
\end{aligned}
\end{eqnarray}
\end{corollary}
\begin{corollary}\label{lemmap2}
Suppose that Assumption \ref{assump} holds. Let $B=\sqrt{\frac{\epsilon}{\rho'}}$, $\theta=4\left(\epsilon\rho'\eta^2\right)^{1/4}\in(0,1)$, and $K=\lfloor\frac{1}{\theta}\rfloor$ in one epoch. Assume that $k\sum_{t=0}^{k-1}\|\x^{t+1}-\x^{t}\|^2\leq B^2$ for all $k\leq K$, then for the iterations (\ref{lemmapite}), we have
\begin{eqnarray}
\begin{aligned}\notag
\|\nabla f(\hat\y)\|\leq\frac{2\sqrt{2}B}{\eta K^2}+\frac{2\theta B}{\eta K}+4\rho B^2 \leq \frac{78\epsilon}{(1-\theta)^2}+\frac{4\rho\epsilon}{\rho'},
\end{aligned}
\end{eqnarray}
where $\hat\y$ is defined on lines 15 and 16 of Algorithm \ref{AGD1p}.
\end{corollary}
Now, we can prove Theorem \ref{theorem1pp}.
\begin{proof}
Recall that we define one round of AGD to be one epoch and the parameters $\eta$, $\rho'$, $B_0$, and $B$ do not change during each epoch. From the update of $\eta$ and $\rho'$, we know $\theta$ and $K$ never change. That is, $\theta=4\left(\epsilon\rho_{int}'\eta_{int}^2\right)^{1/4}=4\left(\epsilon\rho_{max}'\eta_{min}^2\right)^{1/4}$ and $K=\lfloor\frac{1}{\theta}\rfloor$ all the time.

We first consider the last epoch if the algorithm terminates. From line 2 of Algorithm \ref{AGD1p}, we know the while loop breaks when $k\geq K$ and $B_0\leq B$. In the last epoch where the ``if condition" on line 6 does not trigger, we have $k\sum_{t=0}^{k-1}\|\x^{t+1}-\x^{t}\|^2\leq \max \{B^2,B_0^2\}$ and $k\leq K$. So the last epoch consists of $K$ iterations and $k\sum_{t=0}^{k-1}\|\x^{t+1}-\x^{t}\|^2\leq B^2$ for all $k\leq K$. From Corollary \ref{lemmap2}, we have $\|\nabla f(\hat\y)\|\leq\frac{78\epsilon}{(1-\theta)^2}+\frac{4\rho\epsilon}{\rho_{int}'}=\bO(\epsilon)$.

Next, we prove the algorithm will terminate in at most $\bO(\epsilon^{-3/2})$ epochs. In each epoch where the ``if condition" on line 6 triggers, we execute either line 9 or line 11 (in fact, step (\ref{agdp-cont1})), depending on the condition on line 8. Denote one epoch to be valid when the ``if condition" on line 8 holds. Otherwise, denote this epoch to be invalid, where invalid means that we discard the whole iterates in this epoch and reset $\x^0$ and $\x^{-1}$ to be the last iterate in the previous valid epoch, which is stored in $\x_{cur}^0$.

Consider the total number of invalid epochs. We can prove that invalid epoch never appears and line 11 never executes when $B_0\leq B$, $\eta\leq \frac{1}{4L}$, and $\rho'\geq \rho$. In fact, for each epoch except the last one, when $B_0\leq B$, we always have $k< K$. Otherwise, the while loop on line 2 breaks. Thus, when the ``if condition " on line 6 triggers, we must have $\K\sum_{t=0}^{\K-1}\|\x^{t+1}-\x^{t}\|^2> B^2$ for some $\K\leq K$ and $k\sum_{t=0}^{k-1}\|\x^{t+1}-\x^{t}\|^2\leq B^2$ for all $k<\K$. From Corollary \ref{lemmap1}, we have $f(\x^{\K})-f(\x^{0})\leq -\frac{7\epsilon^{3/2}}{8\sqrt{\rho'}}\leq -\gamma\frac{\epsilon^{3/2}}{\sqrt{\rho'}}$, which triggers the condition on line 8. Thus, line 11 never executes. So we only need to count the number of epochs such that $B_0\leq B$, $\eta\leq \frac{1}{4L}$, and $\rho'\geq \rho$. Letting
\begin{eqnarray}
\begin{aligned}\label{cont5}
\frac{B_{0,int}}{(c_0c_1)^{N_{iv}}}\leq\sqrt{\frac{\epsilon}{\rho_{int}'c_2^{2N_{iv}}}}, \quad\frac{\eta_{int}}{c_2^{N_{iv}}}\leq\frac{1}{4L}, \quad\rho_{int}'c_2^{2N_{iv}}\geq\rho,
\end{aligned}
\end{eqnarray}
we have $N_{iv}=\bO\left(\log_{c_0c_1/c_2}\frac{B_{0,int}\rho_{int}'}{\epsilon}+\log_{c_2}(L\eta_{int})+\log_{c_2}\frac{\rho}{\rho_{int}'}\right)$. So we only need $\bO(\log \frac{C}{\epsilon})$ invalid epochs for some constant $C$ to get $B_0\leq B$, $\eta\leq\frac{1}{4L}$, and $\rho'\geq\rho$.

Consider the valid epochs. Since each valid epoch decreases the objective value at least $\gamma\frac{\epsilon^{3/2}}{\sqrt{\rho'}}\geq\gamma\frac{\epsilon^{3/2}}{\sqrt{\rho_{max}'}}$, we have at most $\frac{\triangle_f\sqrt{\rho_{max}'}}{\gamma\epsilon^{3/2}}$ valid epochs.

Putting the two cases together, we need at most $\bO\left(\frac{\triangle_f\sqrt{\rho_{max}'}}{\epsilon^{3/2}}+\log \frac{C}{\epsilon}\right)$ epochs, and accordingly, $\bO\left(\frac{\triangle_f\sqrt{\rho_{max}'}}{\epsilon^{3/2}}+\log \frac{C}{\epsilon}\right)$ function evaluations. On the other hand, each epoch, no matter valid or not, needs at most $K+1=\bO\left(\frac{1}{(\epsilon\rho'\eta^2)^{1/4}}\right)=\bO\left(\frac{1}{(\epsilon\rho_{max}'\eta_{min}^2)^{1/4}}\right)$ gradient computations (see line 6 of Algorithm \ref{AGD1p}). So the total number of gradient computations is $\bO\left(\left(\frac{\triangle_f\sqrt{\rho_{max}'}}{\epsilon^{3/2}}+\log \frac{C}{\epsilon}\right)\left(\frac{1}{(\epsilon\rho_{max}'\eta_{min}^2)^{1/4}}\right)\right)=\bO\left(\frac{\triangle_f(\rho_{max}')^{1/4}}{\epsilon^{7/4}\eta_{min}^{1/2}}+\frac{1}{(\epsilon\rho_{max}'\eta_{min}^2)^{1/4}}\log \frac{C}{\epsilon}\right)$.

\end{proof}
At last, we consider Theorem \ref{theorem1p}. In the original Algorithm \ref{AGD1p} where we do not dynamically change $\rho'$ and $\eta$, we only need to replace (\ref{cont5}) by $\frac{B_{0,int}}{c_0^{N_{iv}}}\leq\sqrt{\frac{\epsilon}{\rho}}$ such that $N_{iv}=\bO\left(\log_{c_0}\frac{\rho B_{0,int}}{\epsilon}\right)$, even if line 11 in Algorithm \ref{AGD1p} is never triggered. Since $\rho'$ and $\eta$ are fixed at $\rho$ and $\frac{1}{4L}$, respectively, we have the complexity of $\bO\left(\frac{\triangle_f\sqrt{\rho}}{\epsilon^{3/2}}+\log \frac{\rho B_{0,int}}{\epsilon}\right)$ function evaluations and $\bO\left(\frac{\triangle_fL^{1/2}\rho^{1/4}}{\epsilon^{7/4}}+\frac{L^{1/2}}{\epsilon^{1/4}\rho^{1/4}}\log \frac{\rho B_{0,int}}{\epsilon}\right)$ gradient computations.

\subsection{Proof of Theorem \ref{theorem1hb}}\label{sec:proof_t2}

We follow the proof sketch in Section \ref{sec:proof_t1} and use the notations therein. Specifically, (\ref{cond1}), (\ref{cond2}), and (\ref{cond3}) also hold for the heavy ball method.

\subsubsection{Large Gradient of $\|\nabla f(\x^{\K-1})\|$}\label{subsec1hb}
Similar to Lemma \ref{lemma1}, we first consider the case when $\|\nabla f(\y^{\K-1})\|$ is large and give the following lemma for Algorithm \ref{HB1}.
\begin{lemma}\label{lemma1hb}
Suppose that Assumption \ref{assump} holds. Let $\eta\leq\frac{1}{4L}$ and $0\leq\theta\leq \frac{1}{10}$. In each epoch of Algorithm \ref{HB1} where the ``if condition" triggers, when $\|\nabla f(\x^{\K-1})\|> \frac{4B}{\eta}$, we have
\begin{eqnarray}
\begin{aligned}\notag
f(\z^{\K})-f(\x^{0})\leq-\frac{3B^2}{16\eta}.
\end{aligned}
\end{eqnarray}
\end{lemma}
\begin{proof}
As the gradient is $L$-Lipschitz, we have
\begin{eqnarray}
\begin{aligned}\notag
f(\x^{k+1})\leq& f(\x^{k})+\<\nabla f(\x^{k}),\x^{k+1}-\x^{k}\>+\frac{L}{2}\|\x^{k+1}-\x^{k}\|^2\\
\overset{a}=& f(\x^{k})+\<\nabla f(\x^{k}),(1-\theta)(\x^{k}-\x^{k-1})-\eta\nabla f(\x^k)\>+\frac{L}{2}\|(1-\theta)(\x^{k}-\x^{k-1})-\eta\nabla f(\x^k)\|^2\\
\leq& f(\x^{k})-\eta\|\nabla f(\x^{k})\|^2+\frac{\eta}{2}\|\nabla f(\x^k)\|^2+\frac{1}{2\eta}\|\x^{k}-\x^{k-1}\|^2+L\|\x^{k}-\x^{k-1}\|^2+L\eta^2\|\nabla f(\x^k)\|^2\\
\overset{b}\leq& f(\x^{k})-\frac{\eta}{4}\|\nabla f(\x^{k})\|^2+\frac{3}{4\eta}\|\x^{k}-\x^{k-1}\|^2,
\end{aligned}
\end{eqnarray}
where we use the heavy ball iteration on line 3 of Algorithm \ref{HB1} in $\overset{a}=$ and $\eta\leq\frac{1}{4L}$ in $\overset{b}\leq$. Summing over $k=0,\cdots,\K-1$ and using $\x^{0}=\x^{-1}$, we have
\begin{eqnarray}
\begin{aligned}\label{hbcont1}
f(\x^{\K})-f(\x^{0})\leq \frac{3}{4\eta}\sum_{k=0}^{\K-2}\|\x^{k+1}-\x^{k}\|^2-\frac{\eta}{4}\sum_{k=0}^{\K-1}\|\nabla f(\x^{k})\|^2\overset{c}\leq \frac{3B^2}{4\eta}-\frac{\eta}{4}\|\nabla f(\x^{\K-1})\|^2,
\end{aligned}
\end{eqnarray}
where we use (\ref{cond2}) in $\overset{c}\leq$. On the other hand, we also have
\begin{eqnarray}
\begin{aligned}\label{hbcont2}
f(\x^{\K-1})-f(\x^{0})\leq \frac{3}{4\eta}\sum_{k=0}^{\K-3}\|\x^{k+1}-\x^{k}\|^2-\frac{\eta}{4}\sum_{k=0}^{\K-2}\|\nabla f(\x^{k})\|^2\leq \frac{3B^2}{4\eta}.
\end{aligned}
\end{eqnarray}
Define $h(\x)=f(\x)+\frac{L}{2}\|\x-\x^0\|^2$. We know $h(\x)$ is convex since $\nabla^2 f(\x)\succeq-L\I$. Thus we have $h(\z^{\K})\leq \alpha h(\x^{\K})+(1-\alpha)h(\x^{\K-1})$ with $\alpha=\frac{1}{1+(1-2\theta)(1-\theta)}\in [\frac{1}{2},\frac{1}{1.72}]$ and $\z^{\K}=\alpha \x^{\K}+(1-\alpha)\x^{\K-1}$, which further yields
\begin{eqnarray}
\begin{aligned}\label{hbcont3}
f(\z^{\K})\leq& \alpha f(\x^{\K})+(1-\alpha)f(\x^{\K-1})+\frac{L\alpha}{2}\|\x^{\K}-\x^0\|^2+\frac{L(1-\alpha)}{2}\|\x^{\K-1}-\x^0\|^2\\
&-\frac{L}{2}\|\alpha (\x^{\K}-\x^0)+(1-\alpha)(\x^{\K-1}-\x^0)\|^2\\
\overset{d}=&\alpha f(\x^{\K})+(1-\alpha)f(\x^{\K-1})+\frac{L\alpha(1-\alpha)}{2}\|\x^{\K}-\x^{\K-1}\|^2\\
\leq&\alpha f(\x^{\K})+(1-\alpha)f(\x^{\K-1})+\frac{1}{32\eta}\|\x^{\K}-\x^{\K-1}\|^2\\
\leq&\alpha f(\x^{\K})+(1-\alpha)f(\x^{\K-1})+\frac{1}{16\eta}\|\x^{\K-1}-\x^{\K-2}\|^2+\frac{\eta}{16}\|\nabla f(\x^{\K-1})\|^2,
\end{aligned}
\end{eqnarray}
where we use $|\alpha x+(1-\alpha)y|^2=\alpha x^2+(1-\alpha)y^2-\alpha(1-\alpha)|x-y|^2$ in $\overset{d}=$. Plugging (\ref{hbcont1}) and (\ref{hbcont2}) into (\ref{hbcont3}) and using $\|\x^{\K-1}-\x^{\K-2}\|^2\leq B^2$, we have
\begin{eqnarray}
\begin{aligned}\notag
f(\z^{\K})-f(\x^{0})\leq&  \frac{3B^2}{4\eta}+\frac{B^2}{16\eta}-\frac{\eta\alpha}{4}\|\nabla f(\x^{\K-1})\|^2+\frac{\eta}{16}\|\nabla f(\x^{\K-1})\|^2\\
\leq&  \frac{13B^2}{16\eta}-\frac{\eta}{16}\|\nabla f(\x^{\K-1})\|^2\overset{e}\leq -\frac{3B^2}{16\eta},
\end{aligned}
\end{eqnarray}
where we use $\|\nabla f(\x^{\K-1})\|> \frac{4B}{\eta}$ in $\overset{e}\leq$.
\end{proof}

\subsubsection{Small Gradient of $\|\nabla f(\x^{\K-1})\|$}\label{subsec2hb}

If $\|\nabla f(\x^{\K-1})\|\leq \frac{4B}{\eta}$, then from the heavy ball iteration on line 3 of Algorithm \ref{HB1} and (\ref{cond2}) we have
\begin{eqnarray}
\begin{aligned}\notag
\|\x^{\K}-\x^{0}\|\leq \|\x^{\K-1}-\x^{0}\|+\eta\|\nabla f(\x^{\K-1})\|+(1-\theta)\|\x^{\K-1}-\x^{\K-2}\|\leq 6B.
\end{aligned}
\end{eqnarray}
Similar to (\ref{sec-smooth2}), using the definition of $g(\x)$ in (\ref{def_g}), we have
\begin{eqnarray}
\begin{aligned}\label{sec-smooth3}
f(\x^{\K})-f(\x^{0})\leq g(\widetilde\x^{\K})-g(\widetilde\x^{0})+36\rho B^3.
\end{aligned}
\end{eqnarray}
Denoting
\begin{eqnarray}
\begin{aligned}\notag
\widetilde\bdelta_j^{k}=\widetilde\nabla_j f(\x^{k})-\nabla g_j(\widetilde\x_j^{k}),\qquad\widetilde\bdelta^{k}=\widetilde\nabla f(\x^{k})-\nabla g(\widetilde\x^{k}),
\end{aligned}
\end{eqnarray}
then the heavy ball iteration in Algorithm \ref{HB1} can be rewritten as
\begin{eqnarray}
\begin{aligned}\label{stephb}
\widetilde\x_j^{k+1}=&\widetilde\x_j^{k}-\eta\widetilde\nabla_j f(\x^{k})+(1-\theta)(\widetilde\x_j^{k}-\widetilde\x_j^{k-1})\\
=&\widetilde\x_j^{k}-\eta\nabla g_j(\widetilde\x_j^{k})-\eta\widetilde\bdelta_j^{k}+(1-\theta)(\widetilde\x_j^{k}-\widetilde\x_j^{k-1}).
\end{aligned}
\end{eqnarray}
Similar to (\ref{sec-smooth}), $\|\widetilde\bdelta^{k}\|$ can also be bounded as
\begin{eqnarray}
\begin{aligned}\label{sec-smooth-hb}
\|\widetilde\bdelta^{k}\|\leq\frac{\rho}{2}\|\x^{k}-\x^{0}\|^2\leq \frac{\rho B^2}{2}
\end{aligned}
\end{eqnarray}
for any $k<\K$.

\begin{lemma}\label{lemma4hb}
Suppose that Assumption \ref{assump} holds. Let $\eta\leq\frac{1}{4L}$ and $0\leq\theta\leq \frac{1}{10}$. In each epoch of Algorithm \ref{HB1} where the ``if condition" triggers, when $\|\nabla f(\x^{\K-1})\|\leq\frac{4B}{\eta}$, we have
\begin{eqnarray}
\begin{aligned}\notag
f(\z^{\K})-f(\x^{0})\leq -\frac{9\theta B^2}{40\eta\K}+\frac{\eta\rho^2B^4\K}{4\theta}+36\rho B^3.
\end{aligned}
\end{eqnarray}
\end{lemma}
\begin{proof}
Since $g_j(x)$ is quadratic, we have
\begin{eqnarray}
\begin{aligned}\label{cont1hb}
g_j(\widetilde\x_j^{k+1})-g_j(\widetilde\x_j^{k})=&\frac{\blambda_j}{2}|\widetilde\x_j^{k+1}-\widetilde\x_j^{0}|^2-\frac{\blambda_j}{2}|\widetilde\x_j^{k}-\widetilde\x_j^{0}|^2+\<\widetilde\nabla_j f(\x^{0}),\widetilde\x_j^{k+1}-\widetilde\x_j^{k}\>\\
=&\frac{\blambda_j}{2}|\widetilde\x_j^{k+1}-\widetilde\x_j^{k}|^2+\<\widetilde\x_j^{k+1}-\widetilde\x_j^{k},\blambda_j(\widetilde\x_j^{k}-\widetilde\x_j^{0})+\widetilde\nabla_j f(\x^{0})\>\\
=&\frac{\blambda_j}{2}|\widetilde\x_j^{k+1}-\widetilde\x_j^{k}|^2+\<\widetilde\x_j^{k+1}-\widetilde\x_j^{k},\nabla g_j(\widetilde\x_j^{k})\>.
\end{aligned}
\end{eqnarray}
From (\ref{stephb}), we have
\begin{eqnarray}
\begin{aligned}\notag
&\<\widetilde\x_j^{k+1}-\widetilde\x_j^{k},\nabla g_j(\widetilde\x_j^{k})\>\\
=&(1-\theta)\<\widetilde\x_j^{k}-\widetilde\x_j^{k-1},\nabla g_j(\widetilde\x_j^{k})\>-\eta|\nabla g_j(\widetilde\x_j^{k})|^2-\eta\<\widetilde\bdelta_j^{k},\nabla g_j(\widetilde\x_j^{k})\>\\
\leq&(1-\theta)\<\widetilde\x_j^{k}-\widetilde\x_j^{k-1},\nabla g_j(\widetilde\x_j^{k})\>-\eta|\nabla g_j(\widetilde\x_j^{k})|^2+\frac{\eta}{4\theta}|\widetilde\bdelta_j^{k}|^2+\theta\eta|\nabla g_j(\widetilde\x_j^{k})|^2\\
=&(1-\theta)\<\widetilde\x_j^{k}-\widetilde\x_j^{k-1},\nabla g_j(\widetilde\x_j^{k})\>-(1-\theta)\eta|\nabla g_j(\widetilde\x_j^{k})|^2+\frac{\eta}{4\theta}|\widetilde\bdelta_j^{k}|^2.
\end{aligned}
\end{eqnarray}
Plugging into (\ref{cont1hb}), we have
\begin{eqnarray}
\begin{aligned}\label{cont3hb}
g_j(\widetilde\x_j^{k+1})-g_j(\widetilde\x_j^{k})\leq&\frac{\blambda_j}{2}|\widetilde\x_j^{k+1}-\widetilde\x_j^{k}|^2+(1-\theta)\<\widetilde\x_j^{k}-\widetilde\x_j^{k-1},\nabla g_j(\widetilde\x_j^{k})\>\\
&-(1-\theta)\eta|\nabla g_j(\widetilde\x_j^{k})|^2+\frac{\eta}{4\theta}|\widetilde\bdelta_j^{k}|^2.
\end{aligned}
\end{eqnarray}
Rearranging and squaring both sides of (\ref{stephb}) and using $(a+b)^2\leq(1+\frac{\theta}{1-\theta})a^2+(1+\frac{1-\theta}{\theta})b^2=\frac{a^2}{1-\theta}+\frac{b^2}{\theta}$, we have
\begin{eqnarray}
\begin{aligned}\label{cont4hb}
|\widetilde\x_j^{k+1}-\widetilde\x_j^{k}|^2\leq&\frac{1}{1-\theta}\left|(1-\theta)(\widetilde\x_j^{k}-\widetilde\x_j^{k-1})-\eta\nabla g_j(\widetilde\x_j^{k})\right|^2+\frac{\eta^2|\widetilde\bdelta_j^{k}|^2}{\theta}\\
=&(1-\theta)|\widetilde\x_j^{k}-\widetilde\x_j^{k-1}|^2+\frac{\eta^2}{1-\theta}|\nabla g_j(\widetilde\x_j^{k})|^2-2\eta\<\widetilde\x_j^{k}-\widetilde\x_j^{k-1},\nabla g_j(\widetilde\x_j^{k})\>+\frac{\eta^2|\widetilde\bdelta_j^{k}|^2}{\theta}.
\end{aligned}
\end{eqnarray}
Multiplying both sides of (\ref{cont4hb}) by $\frac{(1-\theta)^2}{\eta}$, adding it to (\ref{cont3hb}), and rearranging the terms, we have
\begin{eqnarray}
\begin{aligned}\notag
&g_j(\widetilde\x_j^{k+1})-g_j(\widetilde\x_j^{k})\\
\leq&-\left(\frac{(1-\theta)^2}{\eta}-\frac{\blambda_j}{2}\right)|\widetilde\x_j^{k+1}-\widetilde\x_j^{k}|^2+\frac{(1-\theta)^3}{\eta}|\widetilde\x_j^{k}-\widetilde\x_j^{k-1}|^2\\
&-(1-2\theta)(1-\theta)\<\widetilde\x_j^{k}-\widetilde\x_j^{k-1},\nabla g_j(\widetilde\x_j^{k})\>+\left(\frac{\eta}{4\theta}+\frac{\eta(1-\theta)^2}{\theta}\right)|\widetilde\bdelta_j^{k}|^2\\
=&-\left(\frac{(1-\theta)^2}{\eta}-\frac{\blambda_j}{2}\right)|\widetilde\x_j^{k+1}-\widetilde\x_j^{k}|^2+\frac{(1-\theta)^3}{\eta}|\widetilde\x_j^{k}-\widetilde\x_j^{k-1}|^2\\
&-(1-2\theta)(1-\theta)\<\widetilde\x_j^{k}-\widetilde\x_j^{k-1},\blambda_j(\widetilde\x_j^{k}-\widetilde\x_j^{0})+\widetilde\nabla_j f(\x^{0})\>+\left(\frac{\eta}{4\theta}+\frac{\eta(1-\theta)^2}{\theta}\right)|\widetilde\bdelta_j^{k}|^2
\end{aligned}
\end{eqnarray}
\begin{eqnarray}
\hspace*{1.8cm}\begin{aligned}\notag
\leq&-\left(\frac{(1-\theta)^2}{\eta}-\frac{\blambda_j}{2}\right)|\widetilde\x_j^{k+1}-\widetilde\x_j^{k}|^2+\frac{(1-\theta)^3}{\eta}|\widetilde\x_j^{k}-\widetilde\x_j^{k-1}|^2+\frac{5\eta}{4\theta}|\widetilde\bdelta_j^{k}|^2\\
&-(1-2\theta)(1-\theta)\left(\frac{\blambda_j}{2}|\widetilde\x_j^{k}-\widetilde\x_j^{k-1}|^2+\frac{\blambda_j}{2}|\widetilde\x_j^{k}-\widetilde\x_j^{0}|^2-\frac{\blambda_j}{2}|\widetilde\x_j^{k-1}-\widetilde\x_j^{0}|^2+\<\widetilde\x_j^{k}-\widetilde\x_j^{k-1},\widetilde\nabla_j f(\x^{0})\>\right)\\
=&-\left(\frac{(1-\theta)^2}{\eta}-\frac{\blambda_j}{2}\right)|\widetilde\x_j^{k+1}-\widetilde\x_j^{k}|^2+\left(\frac{(1-\theta)^3}{\eta}-\frac{(1-2\theta)(1-\theta)\blambda_j}{2}\right)|\widetilde\x_j^{k}-\widetilde\x_j^{k-1}|^2\\
&+\frac{5\eta}{4\theta}|\widetilde\bdelta_j^{k}|^2-(1-2\theta)(1-\theta)\left(g_j(\widetilde\x_j^{k})-g_j(\widetilde\x_j^{k-1})\right).
\end{aligned}
\end{eqnarray}
Note that
\begin{eqnarray}
\begin{aligned}\notag
&\frac{(1-\theta)^2}{\eta}-\frac{\blambda_j}{2}-\frac{(1-\theta)^3}{\eta}+\frac{(1-2\theta)(1-\theta)\blambda_j}{2}\\
=&\frac{(1-\theta)^2\theta}{\eta}-\frac{\blambda_j(3\theta-2\theta^2)}{2}\overset{a}\geq\frac{(1-\theta)^2\theta}{\eta}-\frac{3\theta-2\theta^2}{8\eta}\overset{b}\geq\frac{9\theta}{20\eta},
\end{aligned}
\end{eqnarray}
where we use $\blambda_j\leq L=\frac{1}{4\eta}$ in $\overset{a}\geq$ and $\theta\leq \frac{1}{10}$ in $\overset{b}\geq$. So we have
\begin{eqnarray}
\begin{aligned}\notag
&g_j(\widetilde\x_j^{k+1})-g_j(\widetilde\x_j^{k})+(1-2\theta)(1-\theta)\left(g_j(\widetilde\x_j^{k})-g_j(\widetilde\x_j^{k-1})\right)\\
\leq&-\left(\frac{(1-\theta)^3}{\eta}-\frac{(1-2\theta)(1-\theta)\blambda_j}{2}+\frac{9\theta}{20\eta}\right)|\widetilde\x_j^{k+1}-\widetilde\x_j^{k}|^2\\
&+\left(\frac{(1-\theta)^3}{\eta}-\frac{(1-2\theta)(1-\theta)\blambda_j}{2}\right)|\widetilde\x_j^{k}-\widetilde\x_j^{k-1}|^2+\frac{5\eta}{4\theta}|\widetilde\bdelta_j^{k}|^2.
\end{aligned}
\end{eqnarray}
Summing over $k=0,1,\cdots,\K-1$ and using $\x^0=\x^{-1}$, we have
\begin{eqnarray}
\begin{aligned}\label{cont7hb}
&g_j(\widetilde\x_j^{\K})-g_j(\widetilde\x_j^{0})+(1-2\theta)(1-\theta)\left(g_j(\widetilde\x_j^{\K-1})-g_j(\widetilde\x_j^{0})\right)\\
\leq&-\left(\frac{(1-\theta)^3}{\eta}-\frac{(1-2\theta)(1-\theta)\blambda_j}{2}\right)|\widetilde\x_j^{\K}-\widetilde\x_j^{\K-1}|^2-\frac{9\theta}{20\eta}\sum_{k=0}^{\K-1}|\widetilde\x_j^{k+1}-\widetilde\x_j^{k}|^2+\frac{5\eta}{4\theta}\sum_{k=0}^{\K-1}|\widetilde\bdelta_j^{k}|^2.
\end{aligned}
\end{eqnarray}
Denoting $\alpha=\frac{1}{1+(1-2\theta)(1-\theta)}\in [\frac{1}{2},\frac{1}{1.72}]$ and multiplying both sides of (\ref{cont7hb}) by $\alpha$, we have
\begin{eqnarray}
\begin{aligned}\label{cont8hb}
&\alpha\left(g_j(\widetilde\x_j^{\K})-g_j(\widetilde\x_j^{0})\right)+(1-\alpha)\left(g_j(\widetilde\x_j^{\K-1})-g_j(\widetilde\x_j^{0})\right)\\
\leq&-\alpha\left(\frac{(1-\theta)^3}{\eta}-\frac{(1-2\theta)(1-\theta)\blambda_j}{2}\right)|\widetilde\x_j^{\K}-\widetilde\x_j^{\K-1}|^2-\frac{9\alpha\theta}{20\eta}\sum_{k=0}^{\K-1}|\widetilde\x_j^{k+1}-\widetilde\x_j^{k}|^2+\frac{5\alpha\eta}{4\theta}\sum_{k=0}^{\K-1}|\widetilde\bdelta_j^{k}|^2\\
\leq&-\frac{1}{2}\left(\frac{(1-\theta)^3}{\eta}-\frac{(1-2\theta)(1-\theta)\blambda_j}{2}\right)|\widetilde\x_j^{\K}-\widetilde\x_j^{\K-1}|^2-\frac{9\theta}{40\eta}\sum_{k=0}^{\K-1}|\widetilde\x_j^{k+1}-\widetilde\x_j^{k}|^2+\frac{\eta}{\theta}\sum_{k=0}^{\K-1}|\widetilde\bdelta_j^{k}|^2,
\end{aligned}
\end{eqnarray}
where we use (\ref{hbcont4}) in the last inequality. On the other hand, from $\z^{\K}=\frac{\x^{\K}+(1-2\theta)(1-\theta)\x^{\K-1}}{1+(1-2\theta)(1-\theta)}=\alpha\x^{\K}+(1-\alpha)\x^{\K-1}$, we have
\begin{eqnarray}
\begin{aligned}\notag
&g_j(\widetilde\z_j^{\K})-g_j(\widetilde\x_j^{0})\\
=&\frac{\blambda_j}{2}\left|\alpha(\widetilde\x_j^{\K}-\widetilde\x_j^{0})+(1-\alpha)(\widetilde\x_j^{\K-1}-\widetilde\x_j^{0})\right|^2+\alpha\<\widetilde\nabla_j f(\x^{0}),\widetilde\x_j^{\K}-\widetilde\x_j^0\>\\
&+(1-\alpha)\<\widetilde\nabla_j f(\x^{0}),\widetilde\x_j^{\K-1}-\widetilde\x_j^0\>\\
\overset{c}=&\frac{\blambda_j\alpha}{2}|\widetilde\x_j^{\K}-\widetilde\x_j^0|^2+\frac{\blambda_j(1-\alpha)}{2}|\widetilde\x_j^{\K-1}-\widetilde\x_j^0|^2-\frac{\blambda_j\alpha(1-\alpha)}{2}|\widetilde\x_j^{\K}-\widetilde\x_j^{\K-1}|^2\\
&+\alpha\<\widetilde\nabla_j f(\x^{0}),\widetilde\x_j^{\K}-\widetilde\x_j^0\>+(1-\alpha)\<\widetilde\nabla_j f(\x^{0}),\widetilde\x_j^{\K-1}-\widetilde\x_j^0\>\\
=&\alpha \left(g_j(\widetilde\x_j^{\K})-g_j(\widetilde\x_j^{0})\right)+(1-\alpha)\left(g_j(\widetilde\x_j^{\K-1})-g_j(\widetilde\x_j^{0})\right)-\frac{\blambda_j\alpha(1-\alpha)}{2}|\widetilde\x_j^{\K}-\widetilde\x_j^{\K-1}|^2\\
\overset{d}\leq&\alpha \left(g_j(\widetilde\x_j^{\K})-g_j(\widetilde\x_j^{0})\right)+(1-\alpha)\left(g_j(\widetilde\x_j^{\K-1})-g_j(\widetilde\x_j^{0})\right)+\frac{1}{32\eta}|\widetilde\x_j^{\K}-\widetilde\x_j^{\K-1}|^2\\
\overset{e}\leq&-\frac{1}{2}\left(\frac{(1-\theta)^3}{\eta}-\frac{(1-2\theta)(1-\theta)\blambda_j}{2}-\frac{1}{16\eta}\right)|\widetilde\x_j^{\K}-\widetilde\x_j^{\K-1}|^2-\frac{9\theta}{40\eta}\sum_{k=0}^{\K-1}|\widetilde\x_j^{k+1}-\widetilde\x_j^{k}|^2+\frac{\eta}{\theta}\sum_{k=0}^{\K-1}|\widetilde\bdelta_j^{k}|^2\\
\overset{f}\leq&-\frac{9\theta}{40\eta}\sum_{k=0}^{\K-1}|\widetilde\x_j^{k+1}-\widetilde\x_j^{k}|^2+\frac{\eta}{\theta}\sum_{k=0}^{\K-1}|\widetilde\bdelta_j^{k}|^2,
\end{aligned}
\end{eqnarray}
where we use $|\alpha x+(1-\alpha)y|^2=\alpha x^2+(1-\alpha)y^2-\alpha(1-\alpha)|x-y|^2$ in $\overset{c}\leq$, $\blambda_j\geq-L=-\frac{1}{4\eta}$ in $\overset{d}\leq$, (\ref{cont8hb}) in $\overset{e}\leq$, and
\begin{eqnarray}
\begin{aligned}\label{hbcont4}
\frac{(1-\theta)^3}{\eta}-\frac{(1-2\theta)(1-\theta)\blambda_j}{2}-\frac{1}{16\eta}\geq\frac{(1-\theta)^3}{\eta}-\frac{1}{8\eta}-\frac{1}{16\eta}\geq 0
\end{aligned}
\end{eqnarray}
with $\theta\in[0,\frac{1}{10}]$ in $\overset{f}\leq$. Summing over $j$, using (\ref{sec-smooth-hb}) and (\ref{cond1}), we have
\begin{eqnarray}
\begin{aligned}\notag
g(\widetilde\z^{\K})-g(\widetilde\x^{0})=\sum_jg_j(\widetilde\z_j^{\K})-g_j(\widetilde\x_j^{0})\leq&-\frac{9\theta}{40\eta}\sum_{k=0}^{\K-1}\|\widetilde\x^{k+1}-\widetilde\x^{k}\|^2+\frac{\eta\rho^2B^4\K}{4\theta}\\
\leq&-\frac{9\theta B^2}{40\eta\K}+\frac{\eta\rho^2B^4\K}{4\theta}.
\end{aligned}
\end{eqnarray}
Plugging into (\ref{sec-smooth3}), we have the conclusion.
\end{proof}

\begin{remark}\label{proof-comp}
Comparing with the proofs in Lemmas \ref{lemma2} and \ref{lemma3}, we do not divide the eigenvalues into two groups in the proof of Lemma \ref{lemma4hb}. However, the bad thing is that $g_j(\widetilde\x_j^{\K})-g_j(\widetilde\x_j^{0})$ and $g_j(\widetilde\x_j^{\K-1})-g_j(\widetilde\x_j^{0})$ appear simultaneously on the left hand side of (\ref{cont7hb}). This is the reason why we introduce the vector $\z^k$ in Algorithm \ref{HB1}.
\end{remark}
Combing Lemmas \ref{lemma1hb} and \ref{lemma4hb}, similar to Corollary \ref{lemma5}, we have
\begin{eqnarray}
\begin{aligned}\notag
f(\z^{\K})-f(\x^{0})\leq -\frac{\epsilon^{3/2}}{\sqrt{\rho}}.
\end{aligned}
\end{eqnarray}
Similar to Lemma \ref{lemma6}, we also have $\|\nabla f(\hat\x)\|\leq \frac{2\sqrt{2}B}{\eta K^2}+\frac{2\theta B}{\eta K}+\rho B^2\leq242\epsilon$ in the last epoch. Using the same proofs of Theorem \ref{theorem1} at the end of Section \ref{subsec3}, we can prove Theorem \ref{theorem1hb}.

\section{Experiments}

We test the practical performance on the matrix completion problem \citep{Negahban-2012,Hardt-2014-focs} and one bit matrix completion problem \citep{Davenport-2014}, and end this section by discussing the gap between theory and practice.

\subsection{Matrix completion}\label{exp_sec1}

In matrix completion \citep{Negahban-2012,Hardt-2014-focs}, we aim to recover the true low rank matrix from a set of randomly observed entries, which can be formulated as follows:
\begin{eqnarray}
\min_{\X\in\mathbb{R}^{m\times n}} \frac{1}{2N}\sum_{(i,j)\in\mathcal O}(\X_{i,j}-\X_{i,j}^*)^2,\quad s.t.\quad \mbox{rank}(\X)\leq r,\notag
\end{eqnarray}
where $\mathcal O$ is the set of randomly observed entries with size $N$ and $\X^*$ is the true low rank matrix to recover. We reformulate the above problem in the following matrix factorization form:
\begin{eqnarray}
\begin{aligned}\notag
\min_{\U\in\mathbb{R}^{m\times r},\V\in\mathbb{R}^{n\times r}} &\frac{1}{2N}\sum_{(i,j)\in\mathcal O}((\U\V^T)_{i,j}-\X_{i,j}^*)^2+\frac{1}{2N}\|\U^T\U-\V^T\V\|_F^2,
\end{aligned}
\end{eqnarray}
where $r$ is the rank of $\X^*$ and the regularization is used to balance $\U$ and $\V$.

We verify the performance on the Movielens-10M, Movielens-20M and Netflix data sets, where the corresponding observed matrices are of size $69878\times 10677$, $138493\times 26744$, and $480189\times 17770$, respectively. We set $r=10$. Denote $\X_{\mathcal O}$ to be the observed data and $\A\Sigma\mathbf{B}^T$ to be its SVD. We initialize $\U=\A_{:,1:r}\sqrt{\Sigma_{1:r,1:r}}$ and $\V=\mathbf{B}_{:,1:r}\sqrt{\Sigma_{1:r,1:r}}$ for all the compared methods. It is efficient to compute the maximal $r$ singular values and the corresponding singular vectors of sparse matrices, for example, by Lanczos.

We compare Ada-RAGD-NC (Algorithm \ref{AGD1p}) and Ada-RHB-NC (Algorithm \ref{HB1p}) with Jin's AGD \citep{jinchi-18-apg}, the ``convex until proven guilty" method \citep{carmon-2017-guilty}, heuristic restarted AGD \citep{Donoghue-2015-NesRestart}, nonlinear conjugate gradient (CG) \citep{Polak-CG1969}, and gradient descent (GD). We do not compare with RAGD-NC (Algorithm \ref{AGD1}) and RHB-NC (Algorithm \ref{HB1}) because the two methods restart at almost every iteration due to the small hyperparameter $B$. Their performance is almost the same as GD and their plots almost coincide with that of GD. Heuristic RAGD consists of the following iterations
\begin{eqnarray}
\begin{aligned}\notag
\x^{k+1}=\y^k-\eta\nabla f(\y^k),\quad \y^{k+1}=\x^{k+1}+\frac{m^{k+1}-1}{m^{k+1}+2}(\x^{k+1}-\x^k),
\end{aligned}
\end{eqnarray}
where $m^0=1$ and
\begin{eqnarray}
\begin{aligned}\notag
\quad m^{k+1}=\left\{\begin{array}{cl}
    m^k+1, & \mbox{if } f(\x^{k+1})\leq f(\x^k),\\
    1, & \mbox{otherwise}.
  \end{array}\right.
\end{aligned}
\end{eqnarray}
The nonlinear conjugate gradient has the following steps
\begin{eqnarray}
\begin{aligned}\notag
\delta^k=-\nabla f(\x^k)+\max\left\{\frac{\<\nabla f(\x^k),\nabla f(\x^k)-\nabla f(\x^{k-1})\>}{\|\nabla f(\x^{k-1})\|^2},0\right\}\delta^{k-1},\quad \x^{k+1}=\x^k+\eta^k\delta^k,
\end{aligned}
\end{eqnarray}
where $\delta^{-1}=0$ and we follow \citep{carmon-2017-guilty} to set $\eta^k$ by the following backtracking line search: set $\eta^k=2\eta^{k-1}$ and check whether $f(\x^k+\eta^k\delta^k)\leq f(\x^k)+\frac{\eta^k\<\delta^k,\nabla f(\x^k)\>}{2}$ holds. If it does not hold, set $\eta^k=\eta^k/2$ and repeat.

We tune the best stepsize $\eta$ for each compared method (except CG) on each dataset. For CG, we set the same stepsize as GD since CG adaptively tune $\eta$ during the iterations. For Ada-RAGD-NC and Ada-RHB-NC, we set $\epsilon=10^{-4}$, $B=\sqrt{\frac{\epsilon}{\rho}}$, $\theta=0.005(\epsilon\rho\eta^2)^{1/4}$, $K=\lfloor1/\theta\rfloor$, $B_0=100$, $\gamma=10^{-5}$, $c_0=1+0.001t$ at the $t$th epoch, and $c_1=10$. We use (\ref{agdp-cont1}) to adaptively tune $\eta$ and $\rho$ since $\rho$ is unknown, where we set $c_2=2$ and initialize $\rho=1$. When preparing the experiments, we observe that proper $B_0$,  $\theta$, and $\gamma$ are crucial in the fast convergence of Ada-RAGD-NC and Ada-RHB-NC. We suggest to set $\theta$ in $(0,0.01]$ and $\gamma$ to be small such that line 11 in Algorithms \ref{AGD1p} and \ref{HB1p} is less frequently triggered. $B_0$ can be set larger when the methods restart frequently. For CG, we stop the line search when it repeats more than 10 times. For Jin's AGD, we tune $\theta=0.04(\epsilon\rho\eta^2)^{1/4}$ and follow \citep{jinchi-18-apg} to set $\gamma=\frac{\theta^2}{\eta}$ and $s=\frac{\gamma}{4\rho}$ in their method. Since the Hessian Lipschitz constant $\rho$ is unknown, we set it as 1 for Jin's AGD for simplicity. For the ``convex until proven guilty" method, we follow the theory in \citep{carmon-2017-guilty} to set the parameters except that we terminate the inner loop after 100 iterations to improve its practical performance. GD and heuristic RAGD have no hyperparameter to tune except the stepsize. Since the optimal function value $f(\x^*)$ is unknown, we run each method for 2000 iterations and use the minimum objective value to approximate the optimal one. We only plot the figures using the first 1000 iterations.

\begin{figure}[p]
\begin{tabular}{@{\extracolsep{0.001em}}c@{\extracolsep{0.001em}}c@{\extracolsep{0.001em}}c}
\includegraphics[width=0.33\textwidth,keepaspectratio]{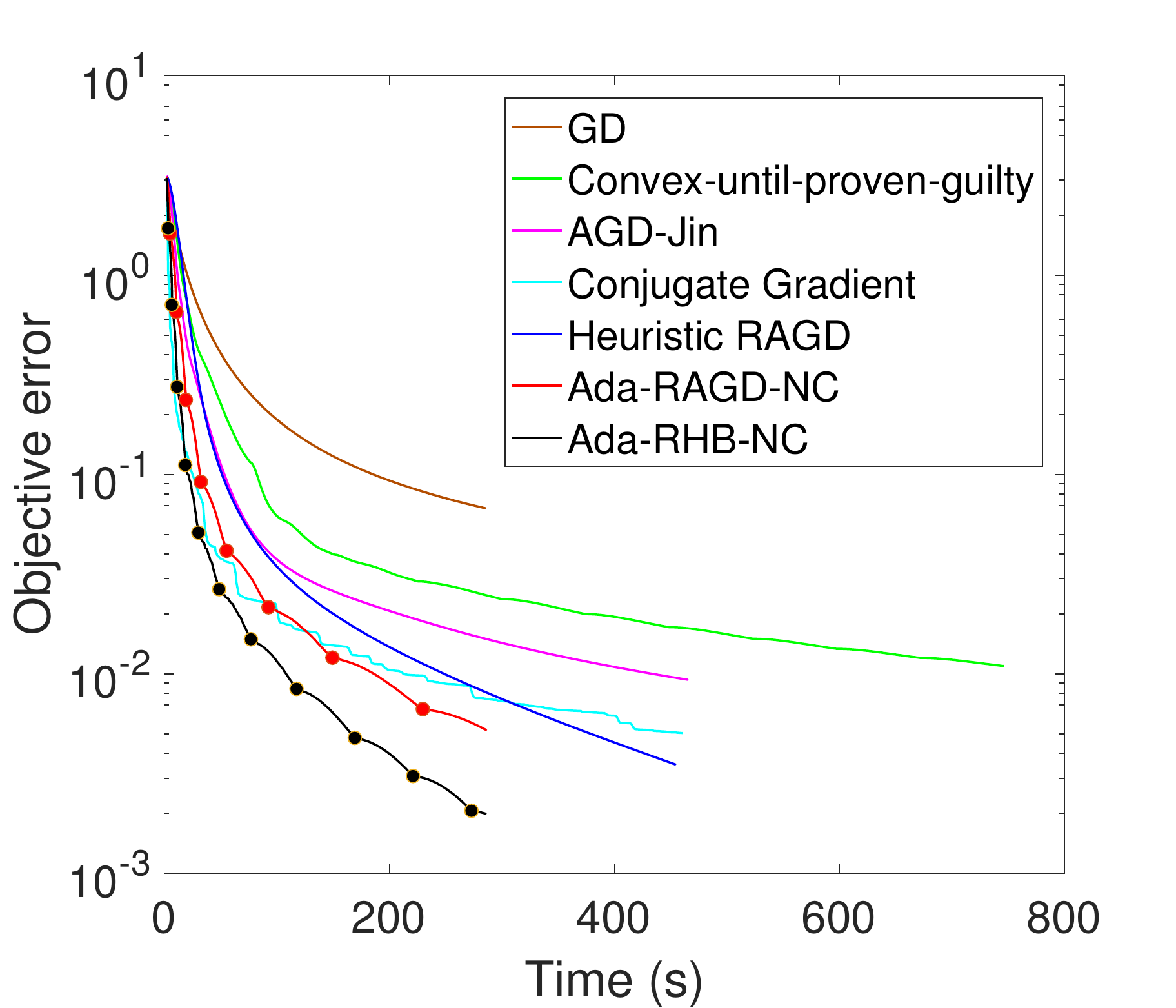}
&\includegraphics[width=0.33\textwidth,keepaspectratio]{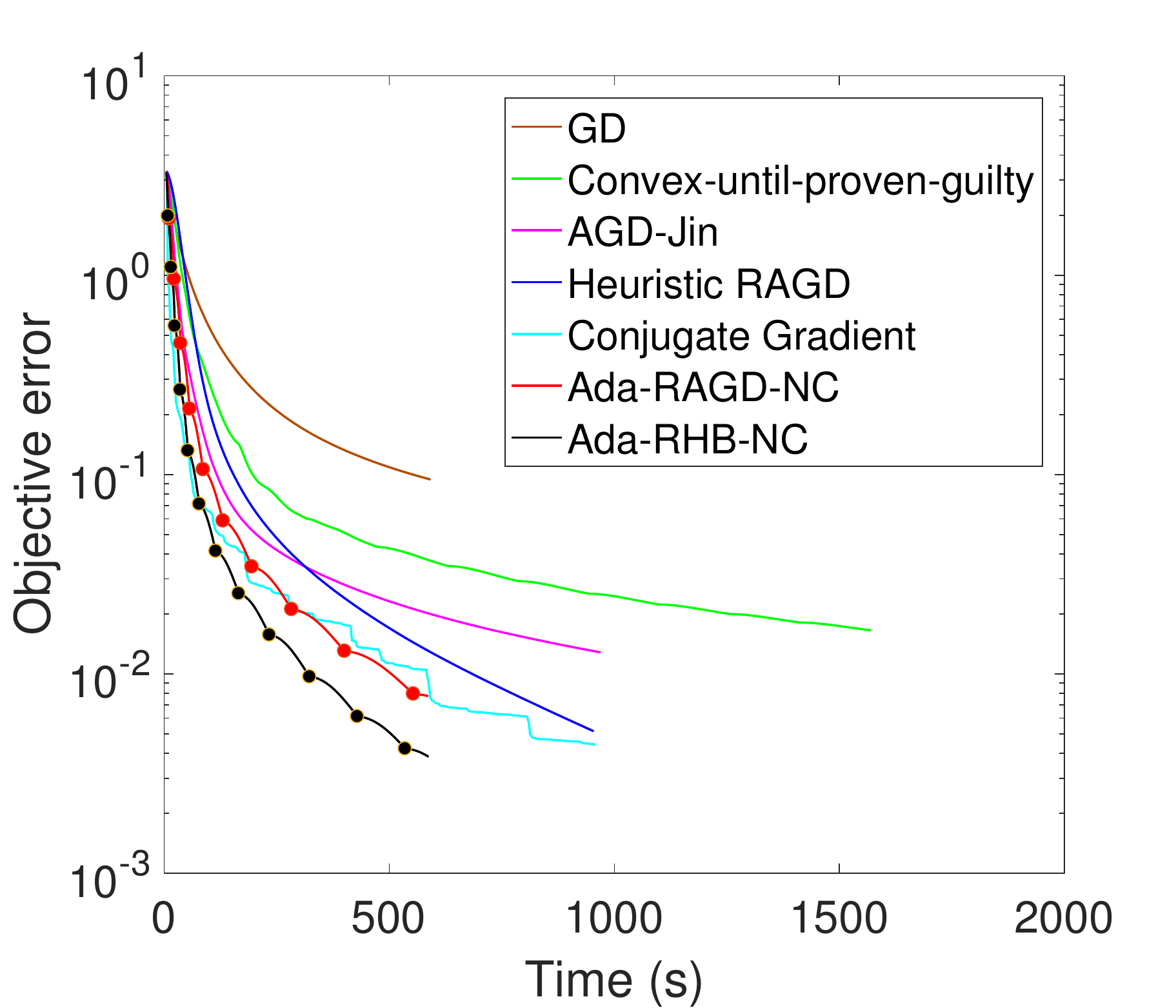}
&\includegraphics[width=0.33\textwidth,keepaspectratio]{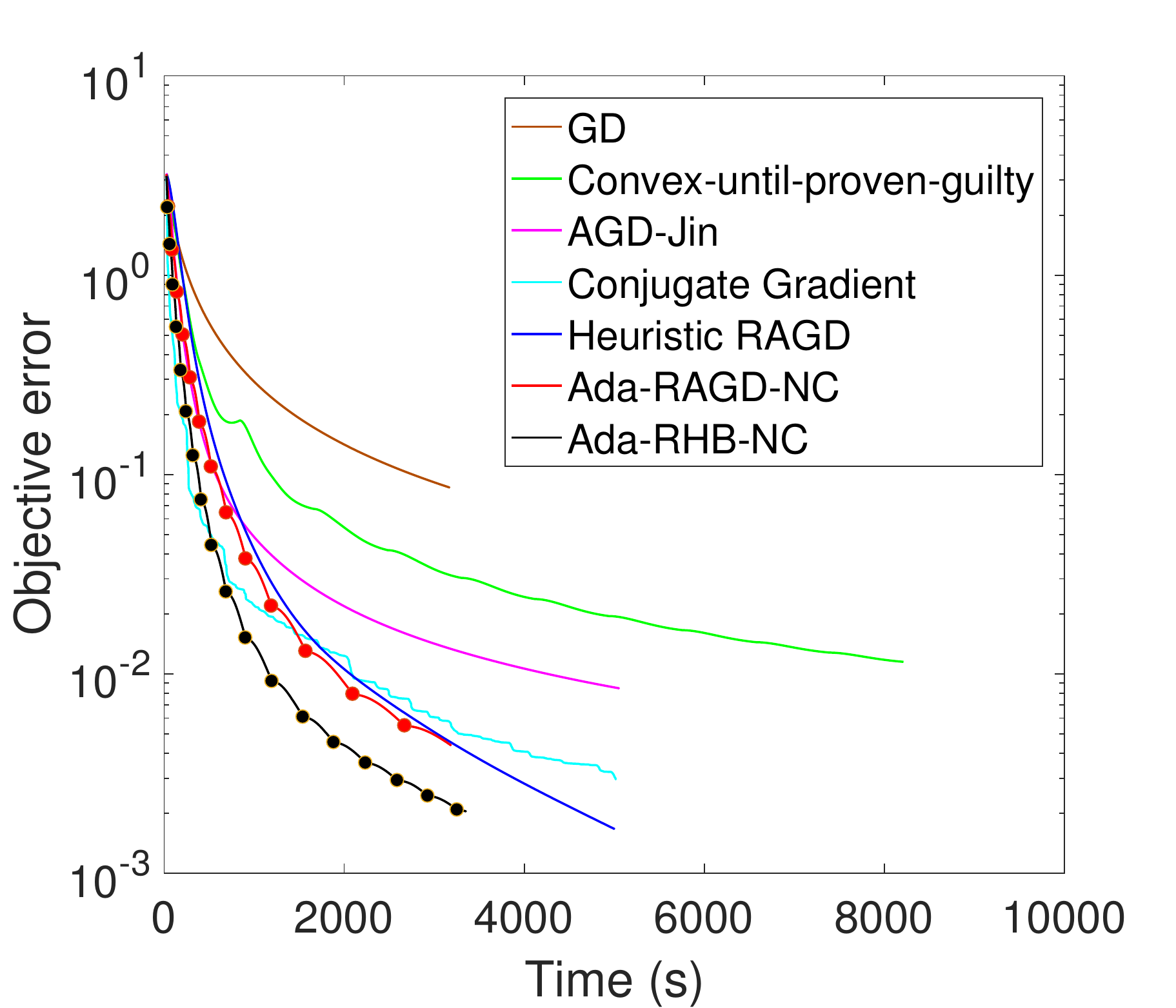}\\
\includegraphics[width=0.33\textwidth,keepaspectratio]{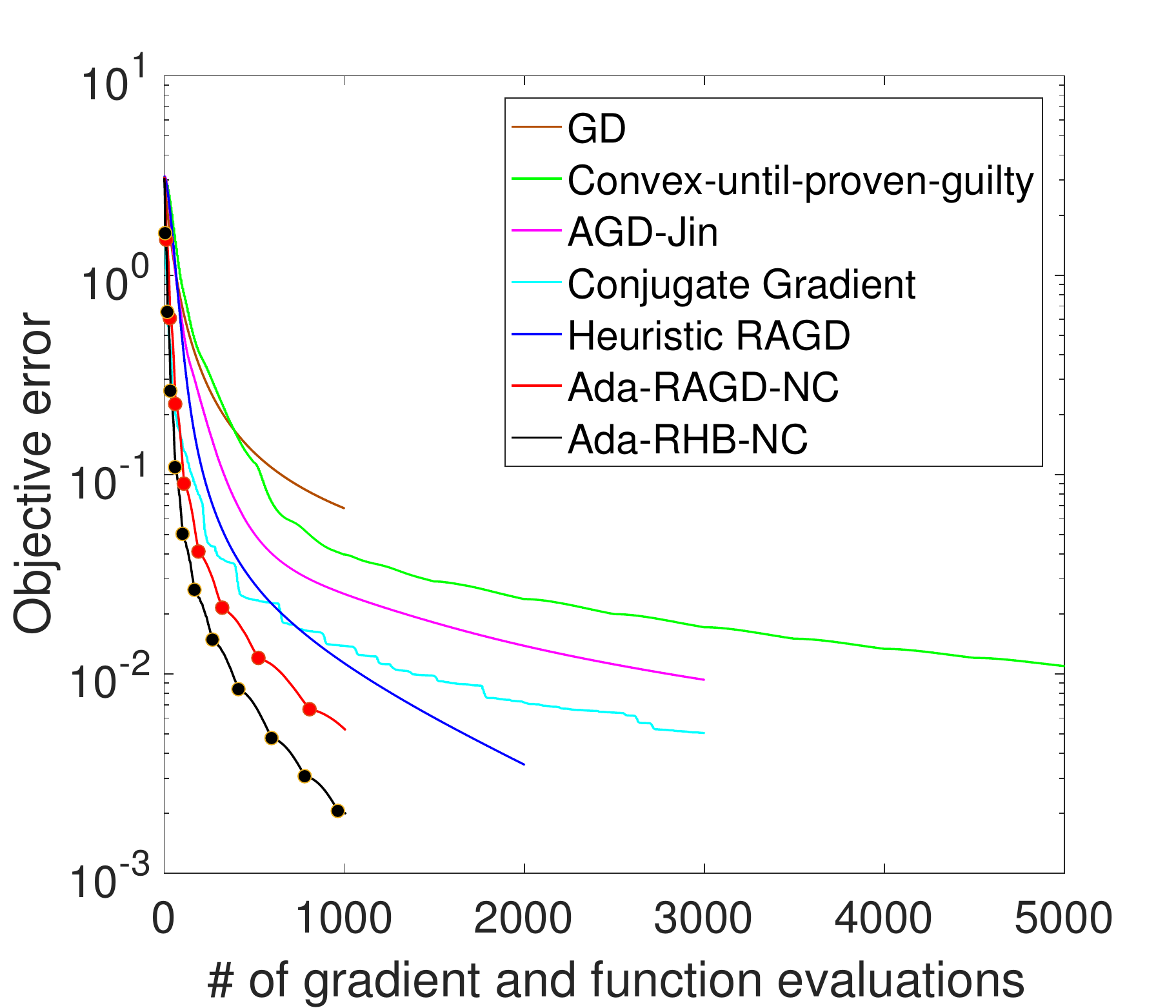}
&\includegraphics[width=0.33\textwidth,keepaspectratio]{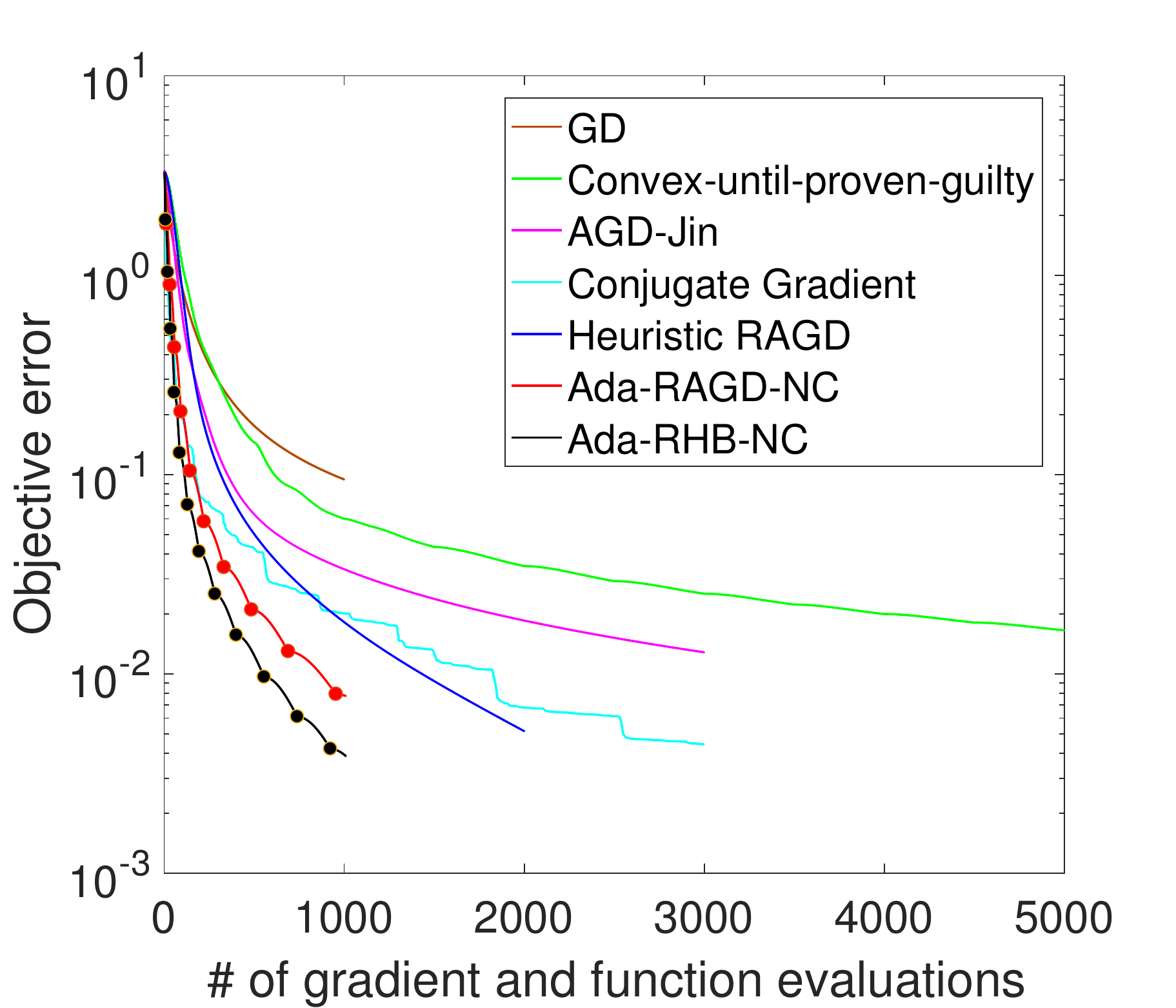}
&\includegraphics[width=0.33\textwidth,keepaspectratio]{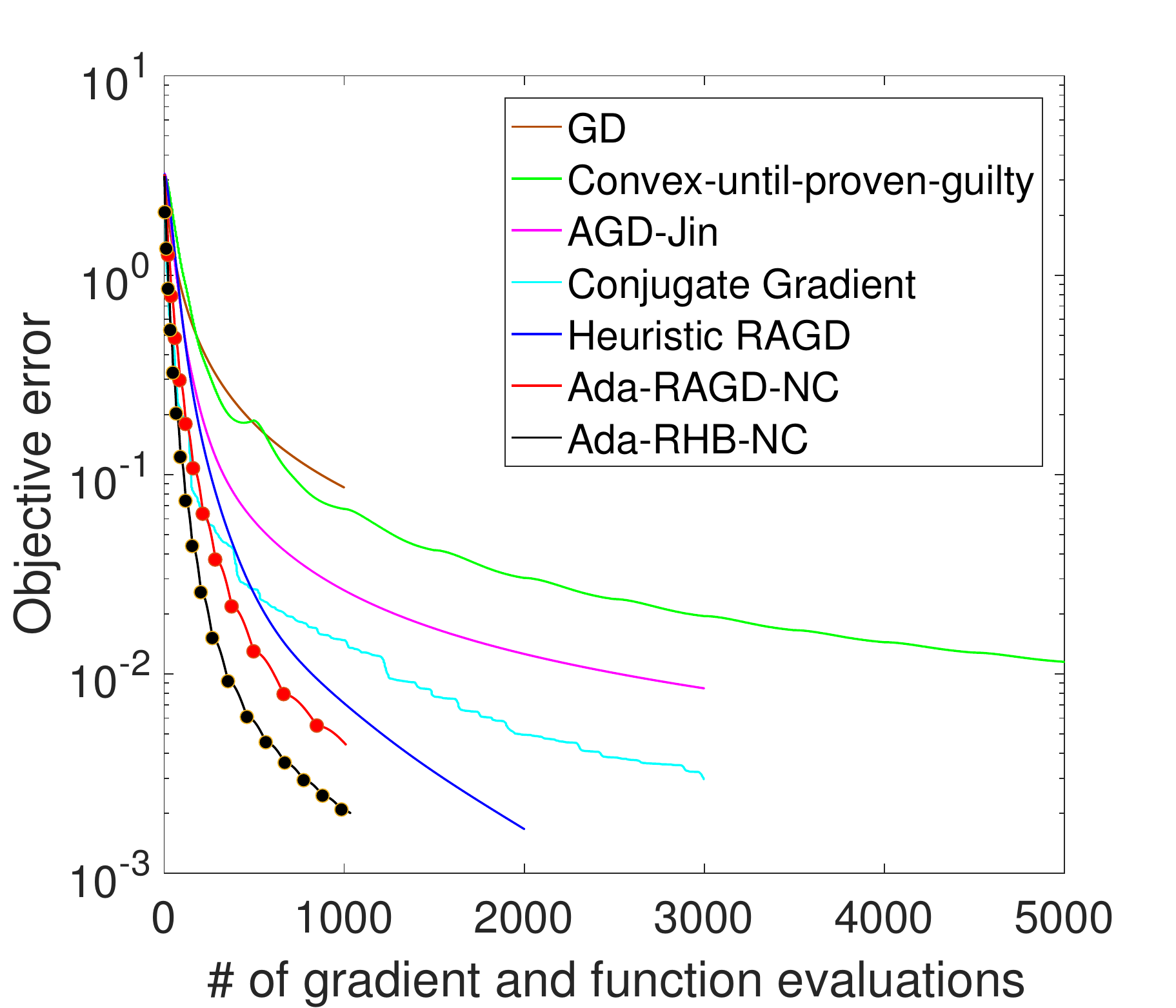}\\
\includegraphics[width=0.33\textwidth,keepaspectratio]{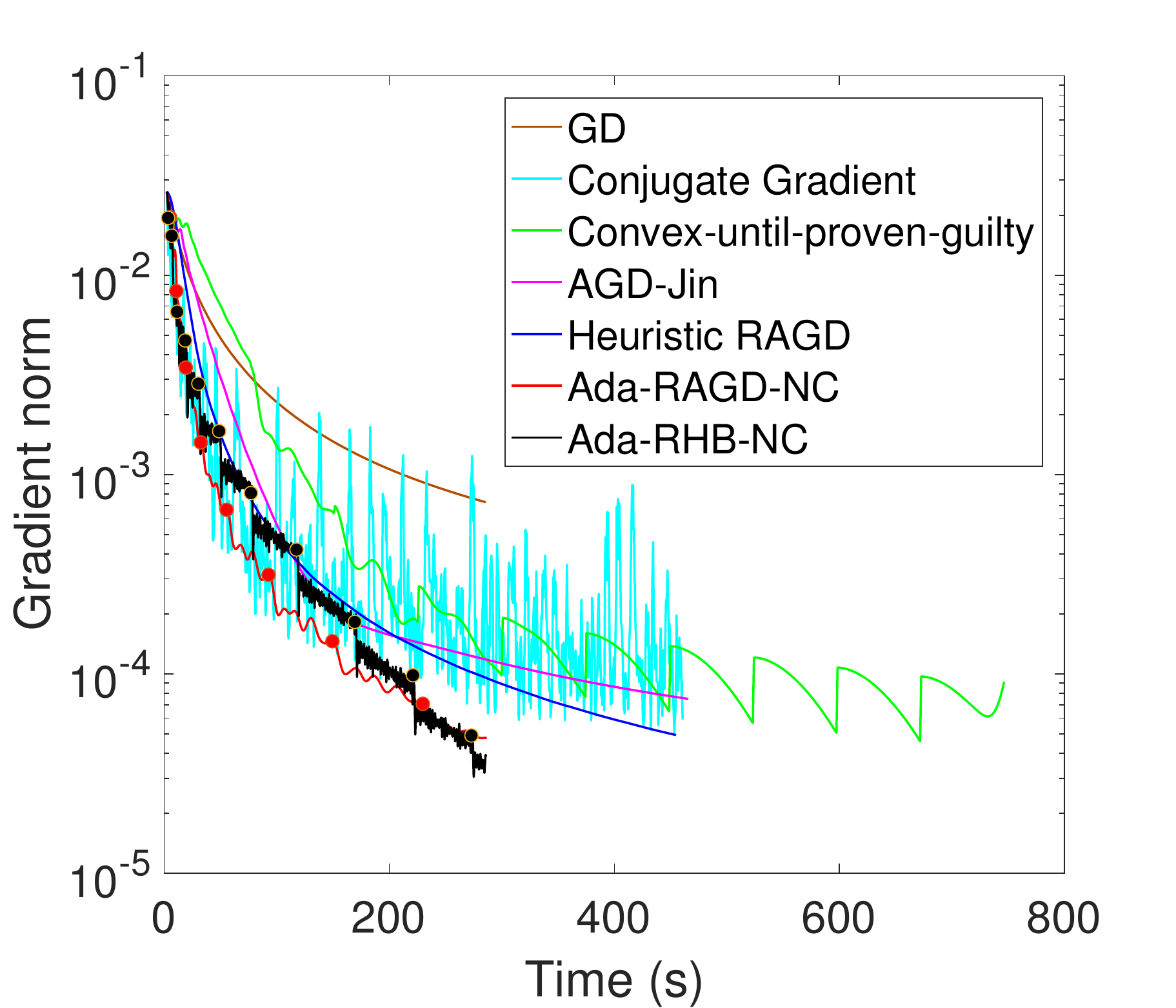}
&\includegraphics[width=0.33\textwidth,keepaspectratio]{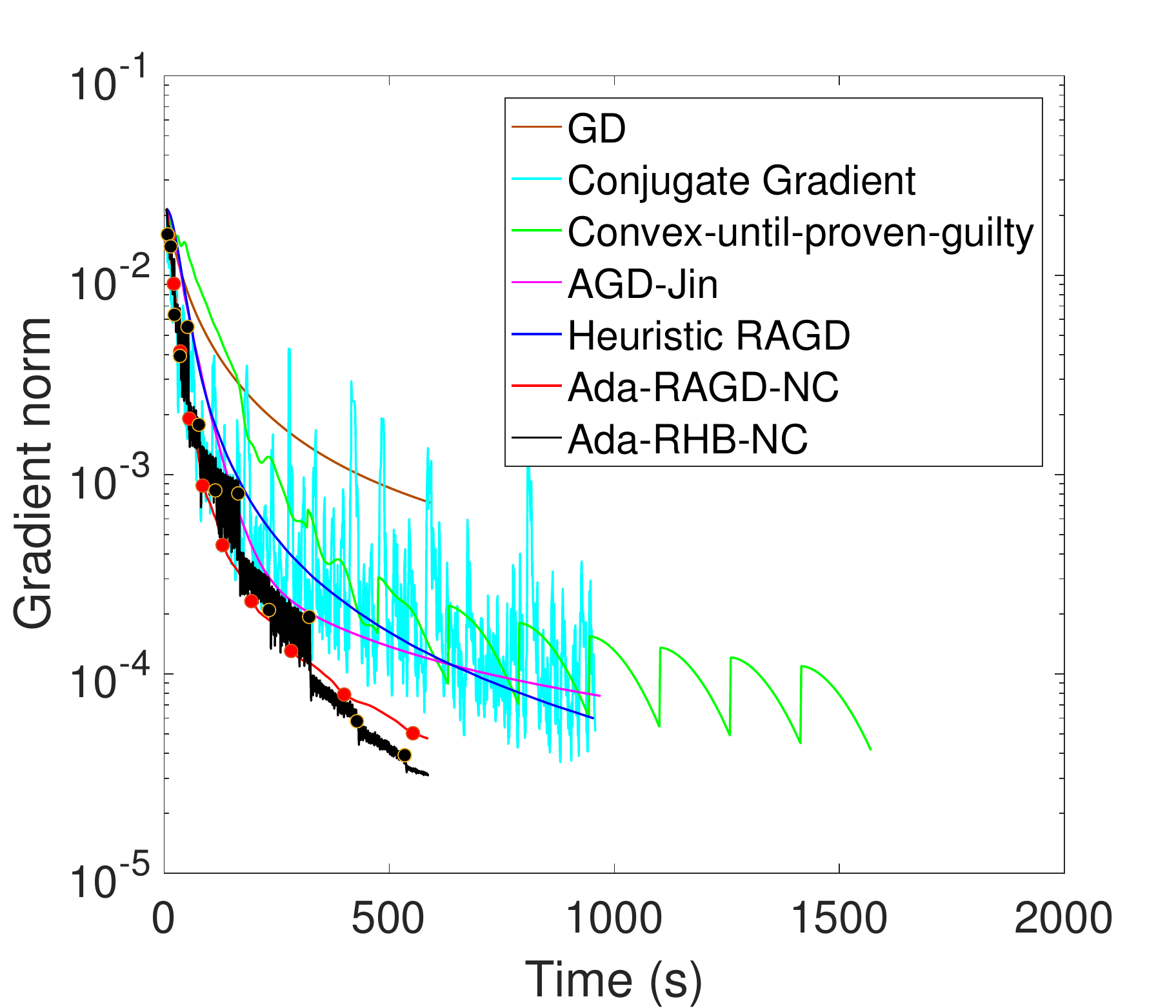}
&\includegraphics[width=0.33\textwidth,keepaspectratio]{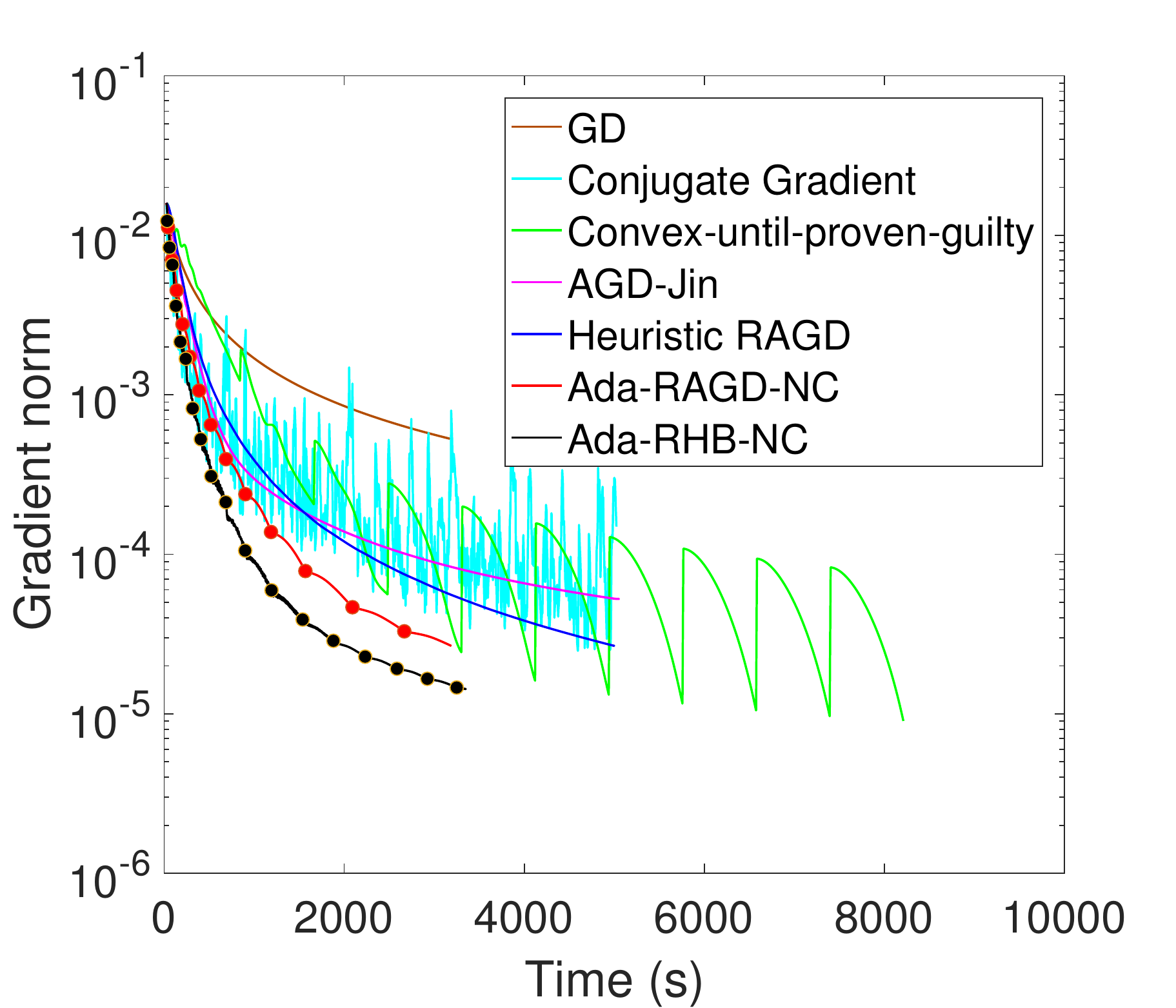}\\
\includegraphics[width=0.33\textwidth,keepaspectratio]{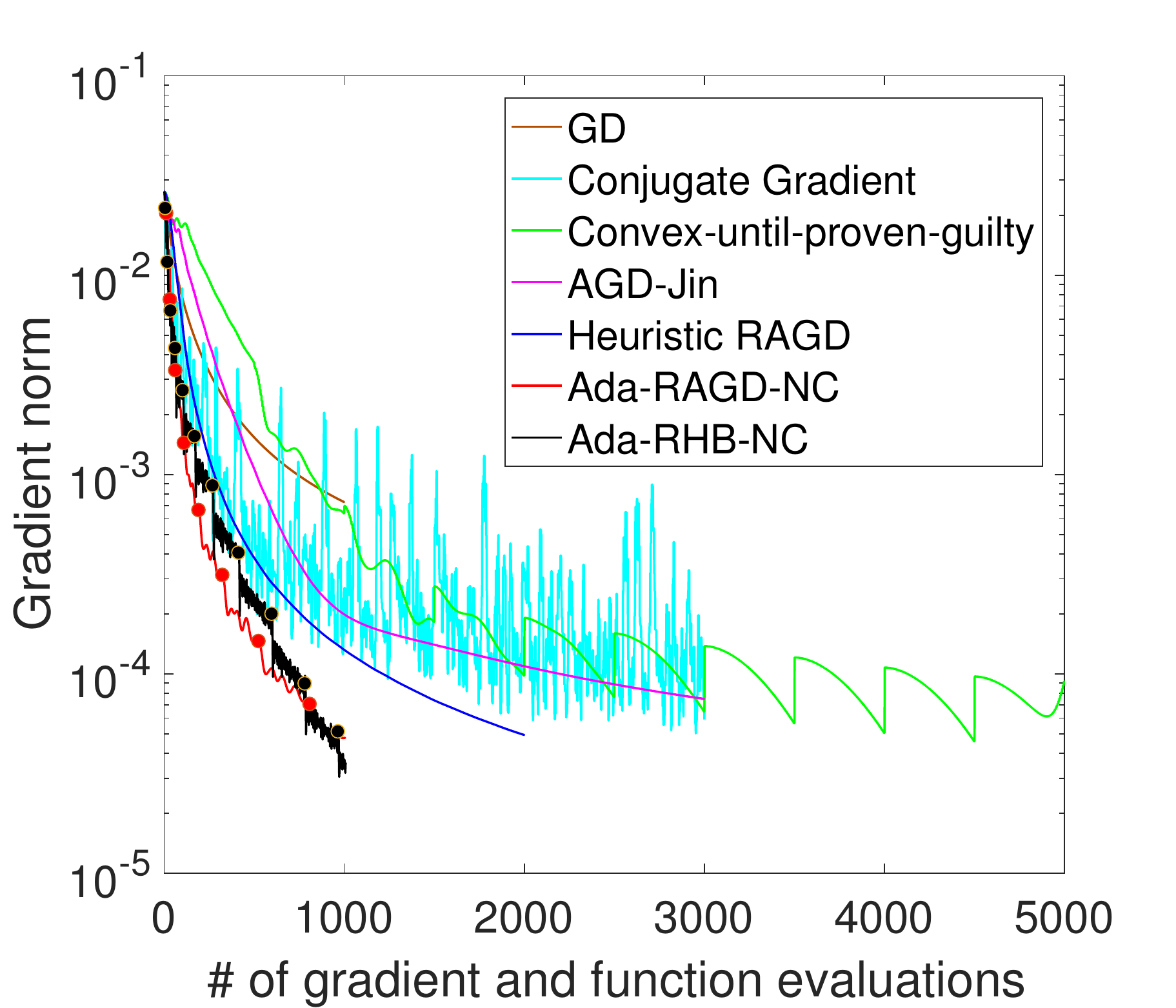}
&\includegraphics[width=0.33\textwidth,keepaspectratio]{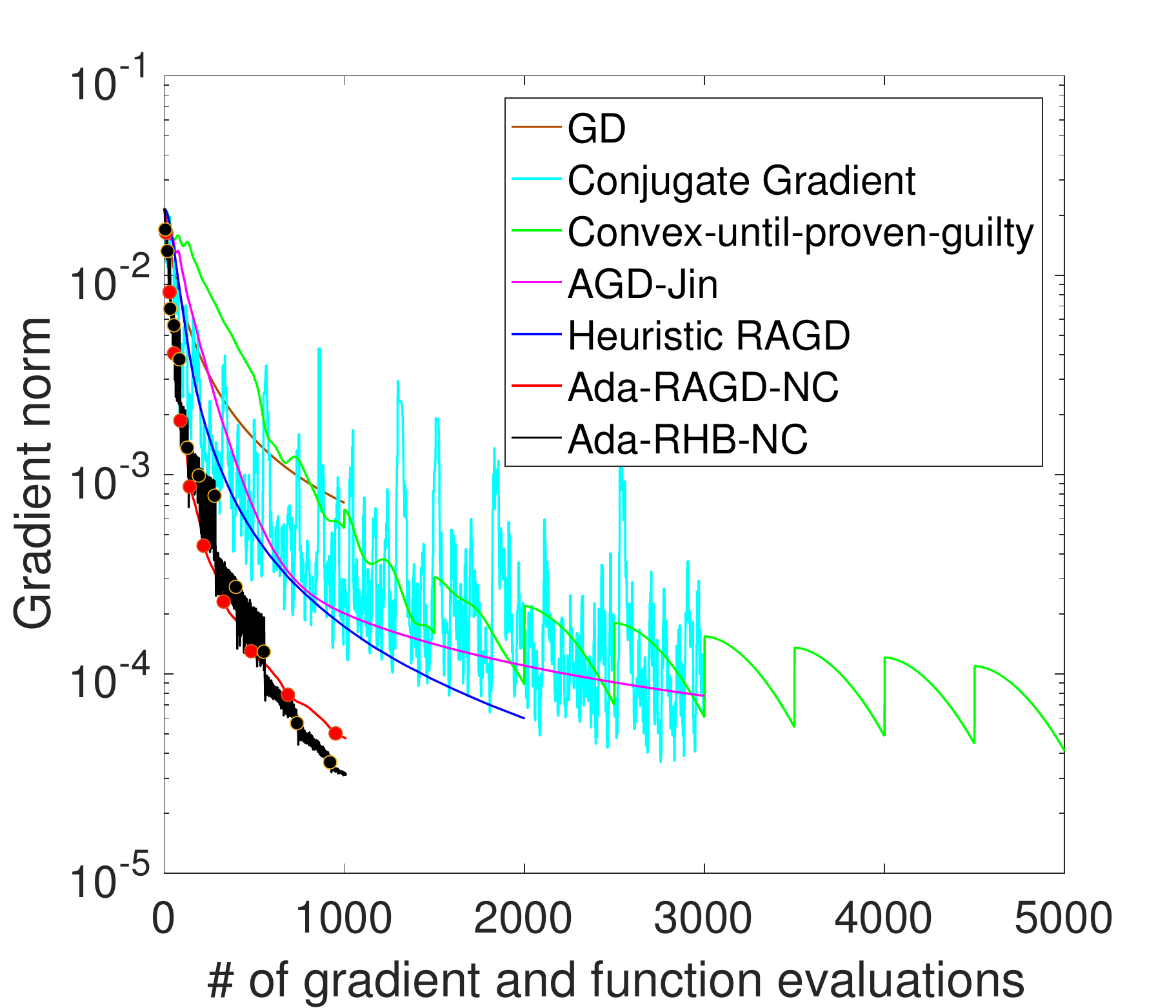}
&\includegraphics[width=0.33\textwidth,keepaspectratio]{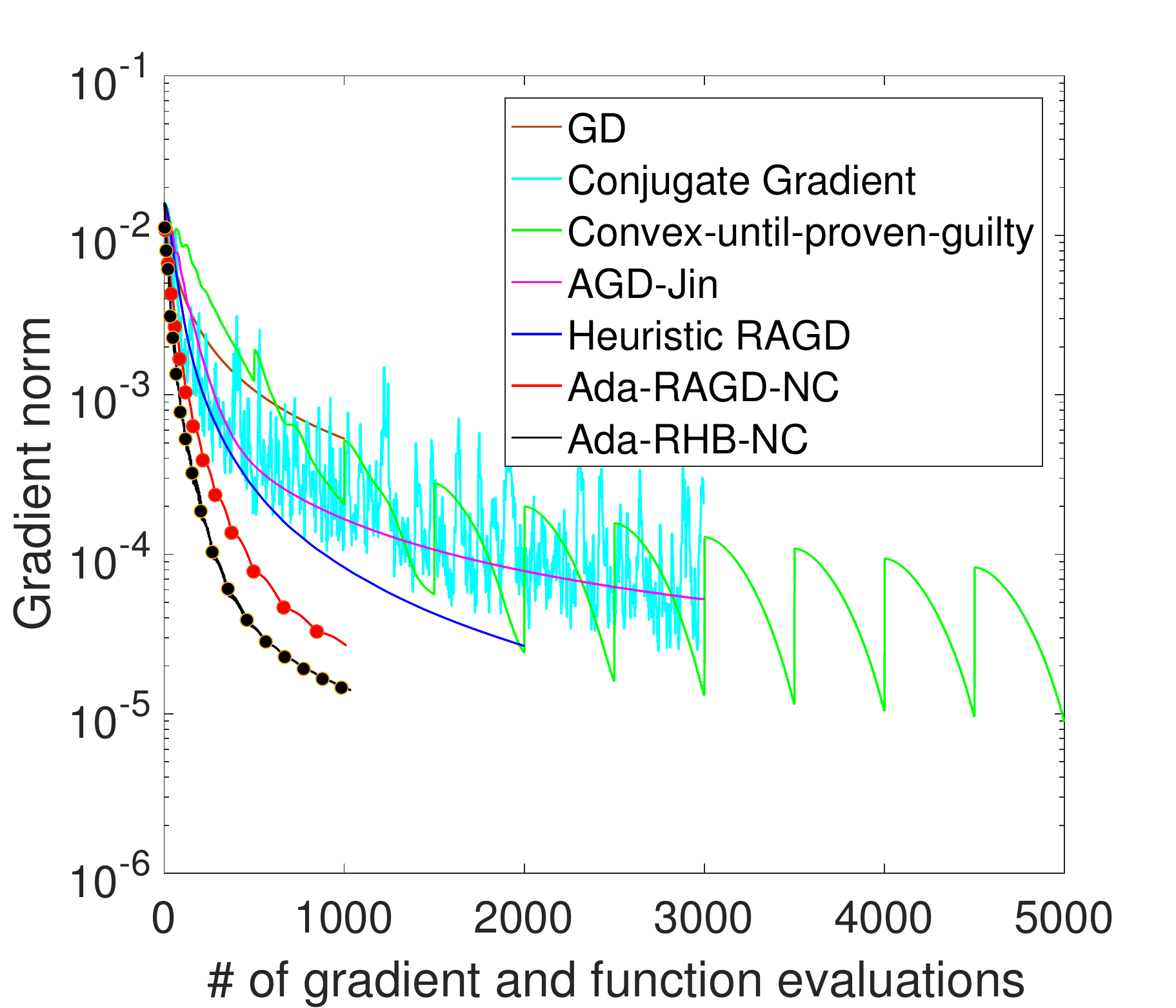}\\
(a) MovieLens-10M & (b) MovieLens-20M & (c) Netflix
\end{tabular}
\caption{Comparisons on the matrix completion problem. The first and second row: objective error. The third and forth row: gradient norm. The first and third row: use time as the horizontal axis. The second and forth row: use the number of function and gradient evaluations as the horizontal axis. Circles in Ada-RAGD-NC and Ada-RHB-NC indicate where restart occurs.}\label{fig1}
\end{figure}

Figure \ref{fig1} plots the objective error $f(\x^k)-f(\x^*)$ and gradient norm $\|\nabla f(\x^k)\|$. We use running time as the horizontal axis in the first and third row, and the number of function and gradient evaluations as the horizontal axis in the second and forth row. Note that the ``convex until proven guilty" method needs at least two gradient evaluations at each iteration while the other methods only need one. For the function evaluations, GD needs none, Ada-RAGD-NC and Ada-RHB-NC need one only when restart occurs, heuristic RAGD needs one at each iteration, Jin's AGD needs at least two, CG needs at least one, and the ``convex until proven guilty" method needs at least three at each iteration. Thus, GD and our Ada-RAGD-NC and Ada-RHB-NC need less total running time when we run all the methods for 1000 iterations. We see that all the accelerated methods perform better than GD, which verifies the efficiency of acceleration in nonconvex optimization. We also observe that our Ada-RAGD-NC and Ada-RHB-NC decrease the objective error and gradient norm to low level quickly. We observe that the gradient norms of CG oscillate during iterations. It may be because CG uses line search to tune the stepsize dynamically, which may be too large and aggressive. On the other hand, due to the specification of the matrix completion problem, we observe that Jin's AGD and the ``convex until proven guilty" method seldom run negative curvature exploitation.

\subsection{One bit matrix completion}

In one bit matrix completion \citep{Davenport-2014}, the signs of a random subset of entries are observed, rather than observing the actual entries. Given a probability density function, for example, the logistic function $f(x)=\frac{e^x}{1+e^x}$, we observe the sign of entry $\X_{i,j}$ as $\Y_{i,j}=1$ with probability $f(\X_{i,j})$, and observe the sign as $-1$ with probability $1-f(\X_{i,j})$. The training model is to minimize the following negative log-likelihood:
\begin{eqnarray}
\begin{aligned}\notag
\min_{\X\in\R^{m\times n}} &-\frac{1}{N}\sum_{(i,j)\in\mathcal O}\left\{\mathbf 1_{\Y_{i,j}=1}\mbox{log}(f(\X_{i,j}))+\mathbf 1_{\Y_{i,j}=-1}\mbox{log}(1-f(\X_{i,j}))\right\},\\
s.t.&\quad \mbox{rank}(\X)\leq r,\notag
\end{aligned}
\end{eqnarray}
where $\mathbf 1_{\Y_{i,j}=1}=\left\{\begin{array}{cl}
    1, & \mbox{if } \Y_{i,j}=1,\\
    0, & \mbox{otherwise}.
  \end{array}\right.$ We solve the following reformulated matrix factorization model:

\begin{eqnarray}
\begin{aligned}\notag
\min_{\U,\V}&-\frac{1}{N}\sum_{(i,j)\in\mathcal O}\left\{\mathbf 1_{\Y_{i,j}=1}\mbox{log}(f((\U\V^T)_{i,j}))+\mathbf 1_{\Y_{i,j}=-1}\mbox{log}(1-f((\U\V^T)_{i,j}))\right\}+\frac{1}{2N}\|\U^T\U-\V^T\V\|_F^2,
\end{aligned}
\end{eqnarray}
where $\U\in\R^{m\times r}$ and $\V\in\R^{n\times r}$. We compare Ada-RAGD-NC (Algorithm \ref{AGD1p}) and Ada-RHB-NC (Algorithm \ref{HB1p}) with the methods compared in Section \ref{exp_sec1}. The best stepsize is tuned for each method on each data set. We use the same initialization and set the same parameters as those in Section \ref{exp_sec1}, and also run each method for 1000 iterations. Figure \ref{fig2} plots the results. We see that acceleration also takes effect in nonconvex optimization and our Ada-RAGD-NC and Ada-RHB-NC also decrease the objective value and gradient norm to low level quickly.

\begin{figure}[p]
\begin{tabular}{@{\extracolsep{0.001em}}c@{\extracolsep{0.001em}}c@{\extracolsep{0.001em}}c}
\includegraphics[width=0.33\textwidth,keepaspectratio]{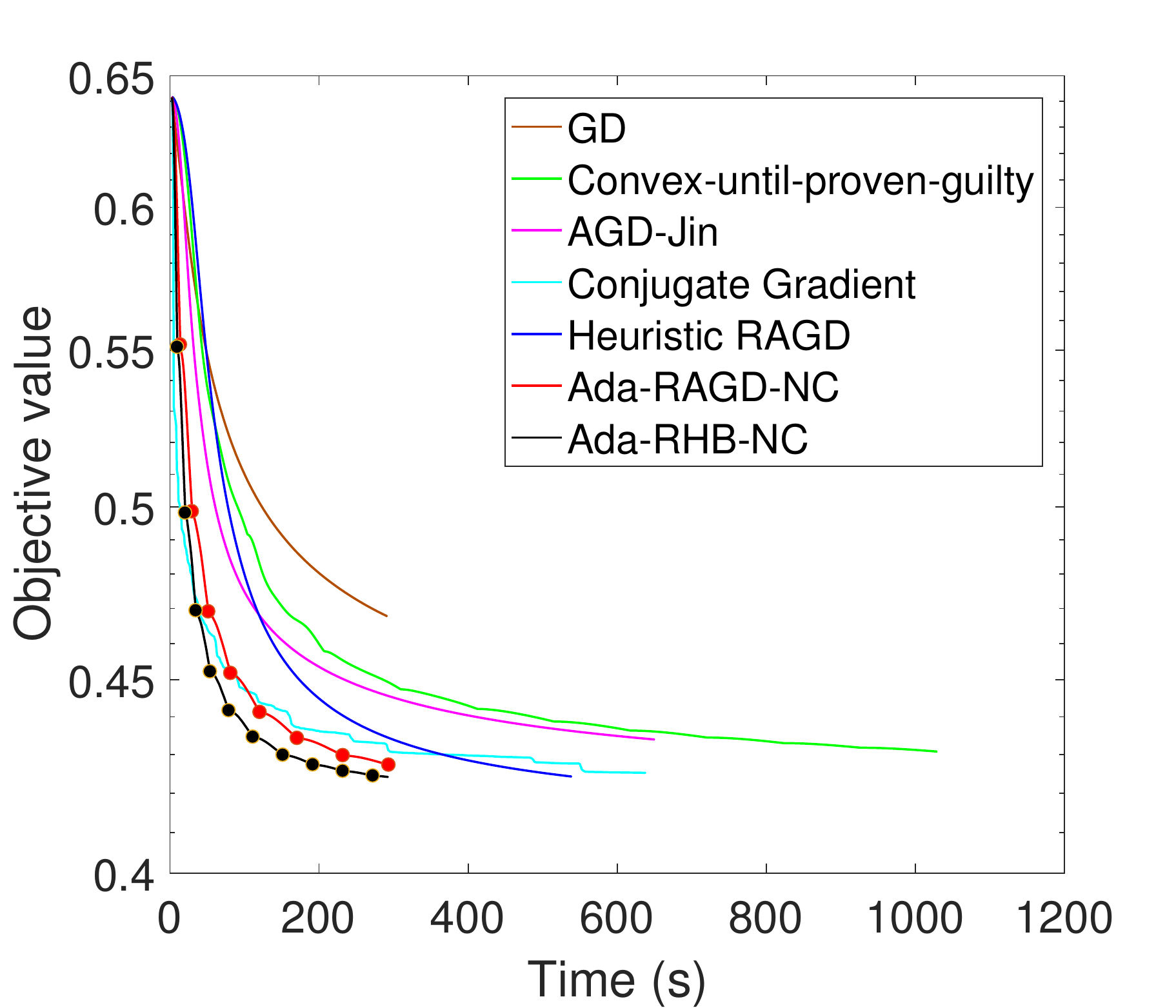}
&\includegraphics[width=0.33\textwidth,keepaspectratio]{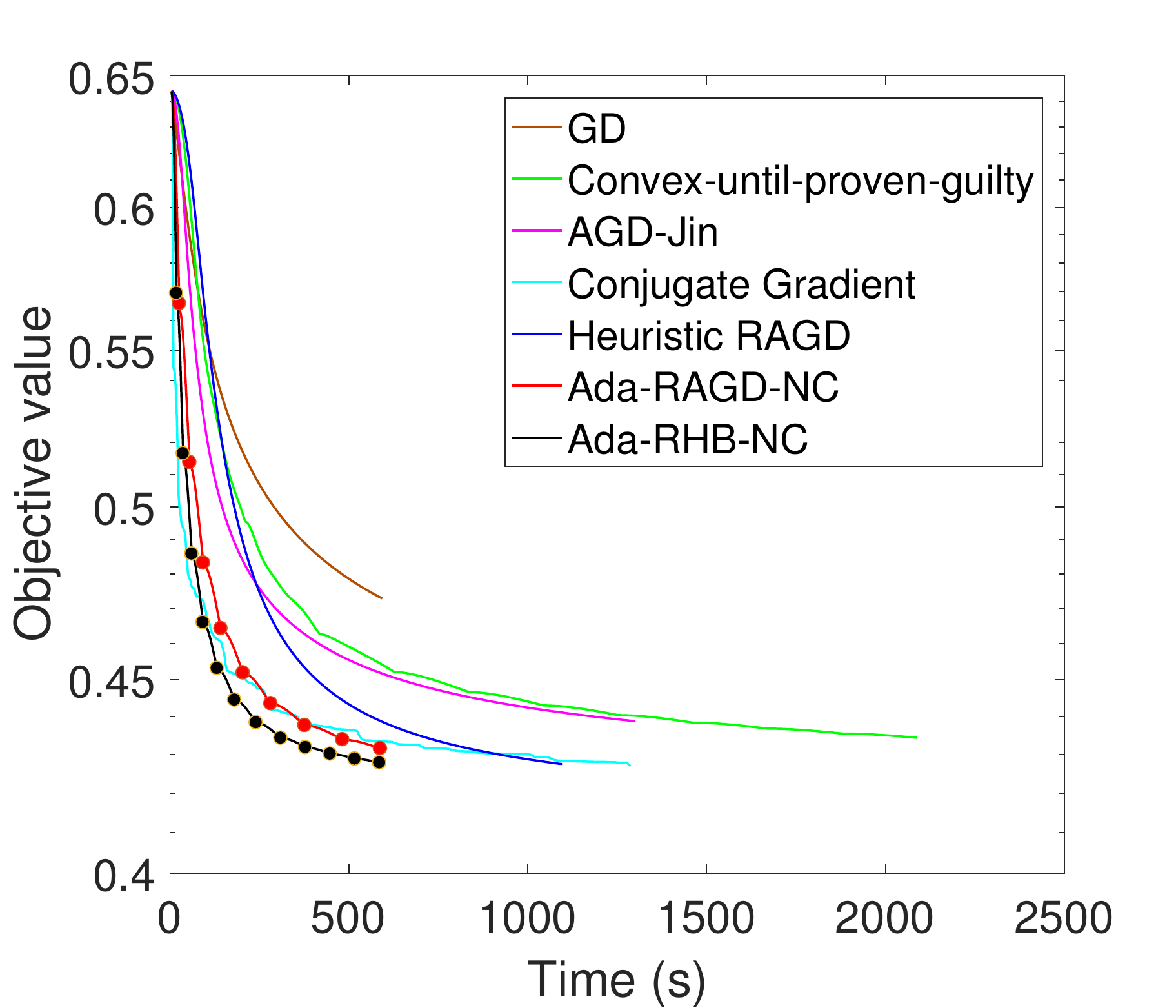}
&\includegraphics[width=0.33\textwidth,keepaspectratio]{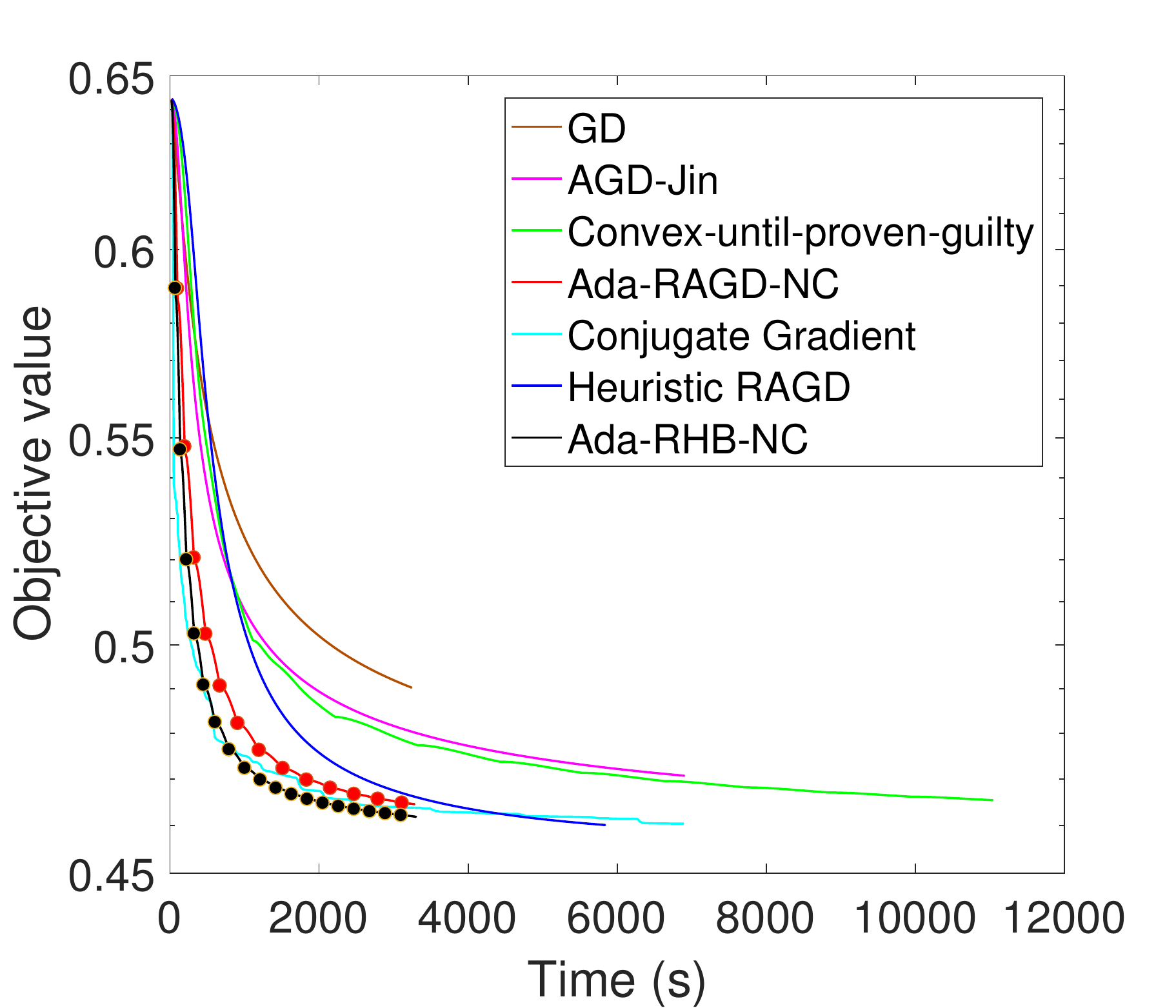}\\
\includegraphics[width=0.33\textwidth,keepaspectratio]{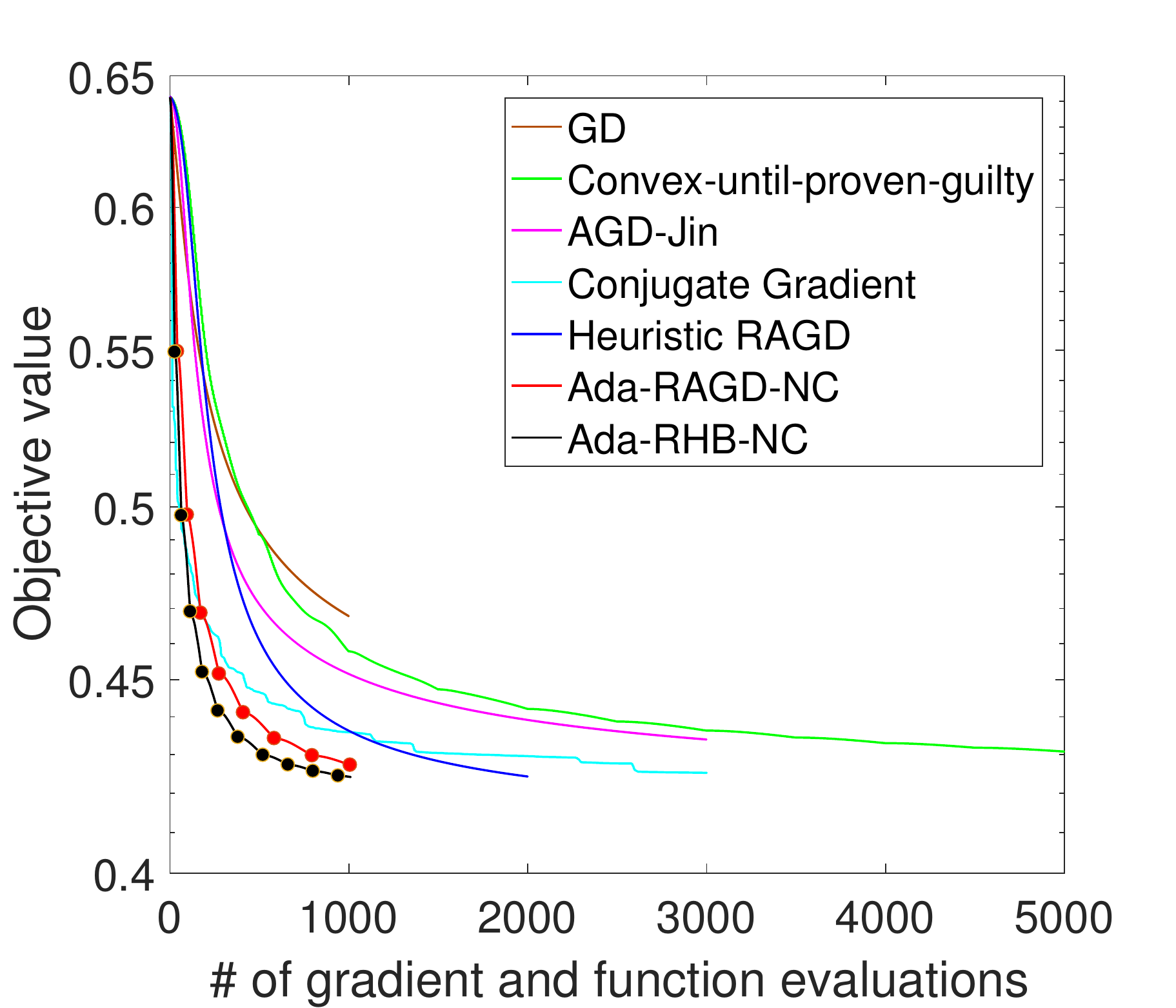}
&\includegraphics[width=0.33\textwidth,keepaspectratio]{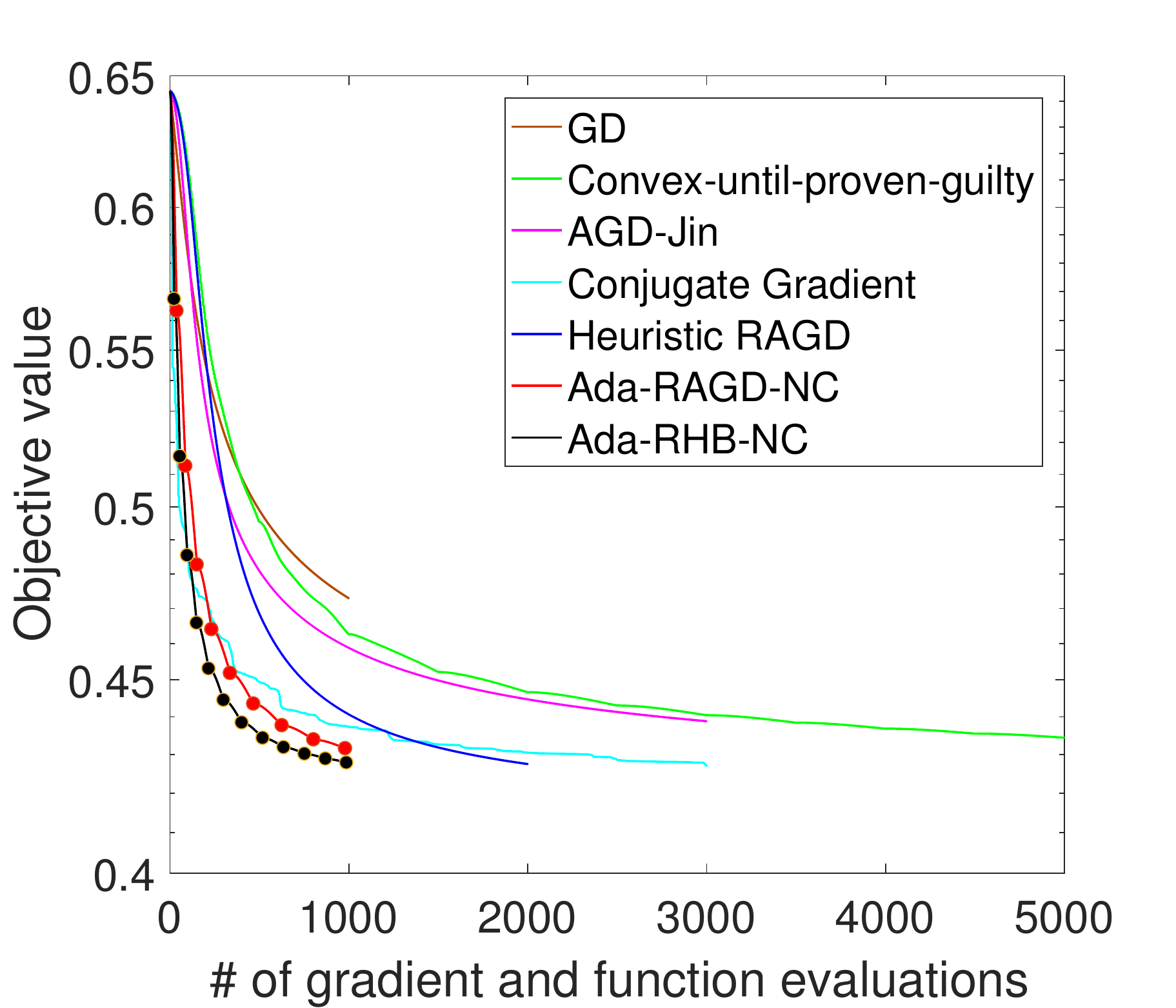}
&\includegraphics[width=0.33\textwidth,keepaspectratio]{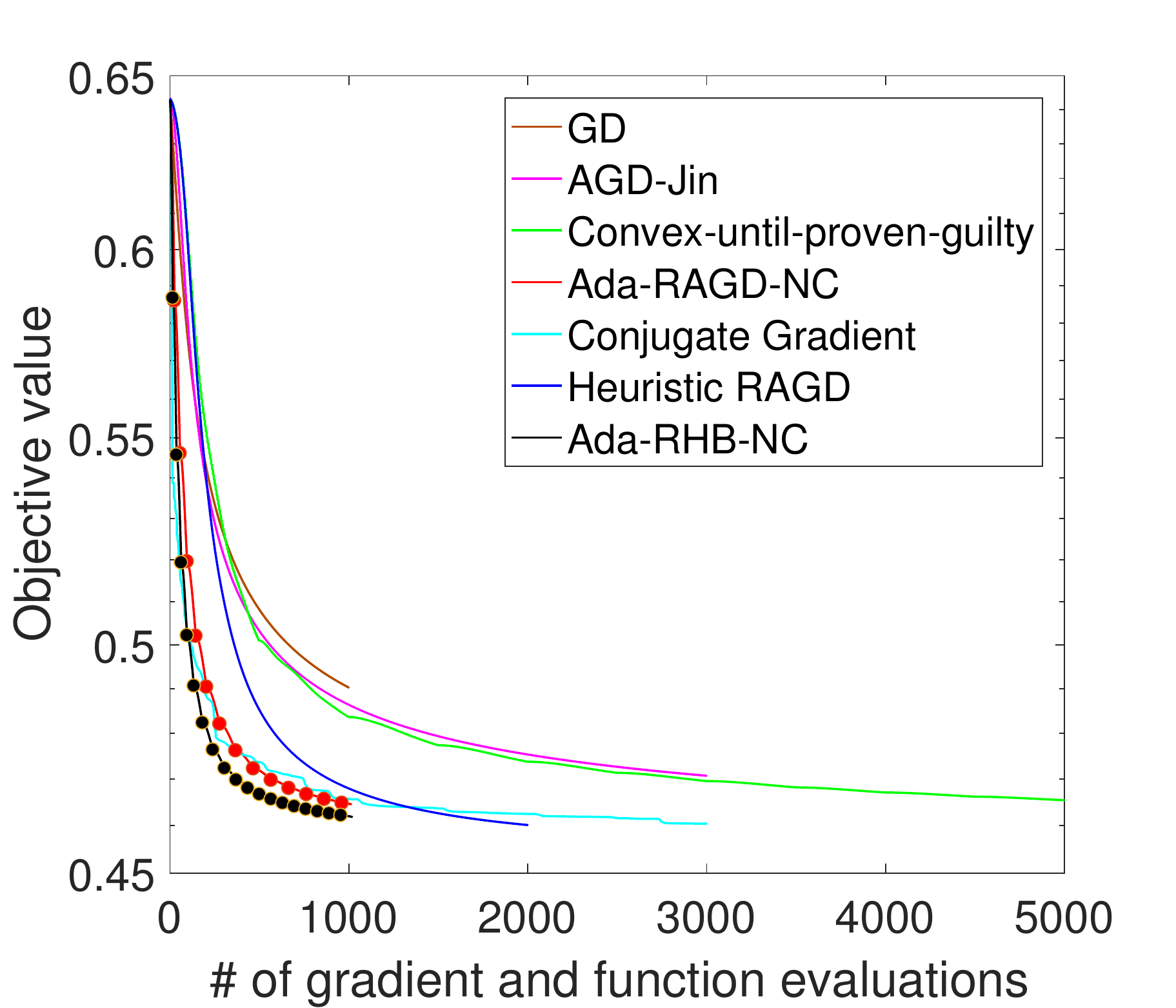}\\
\includegraphics[width=0.33\textwidth,keepaspectratio]{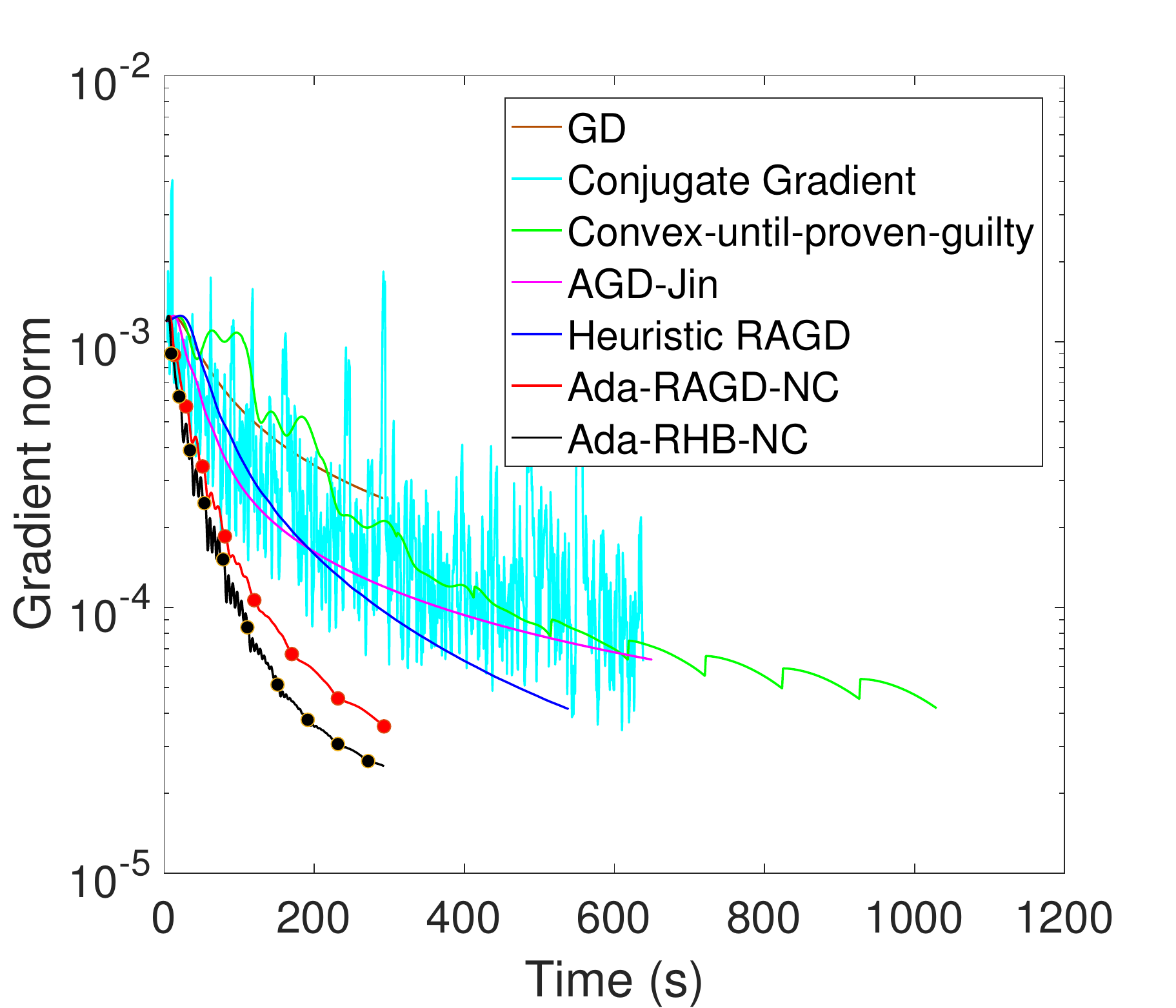}
&\includegraphics[width=0.33\textwidth,keepaspectratio]{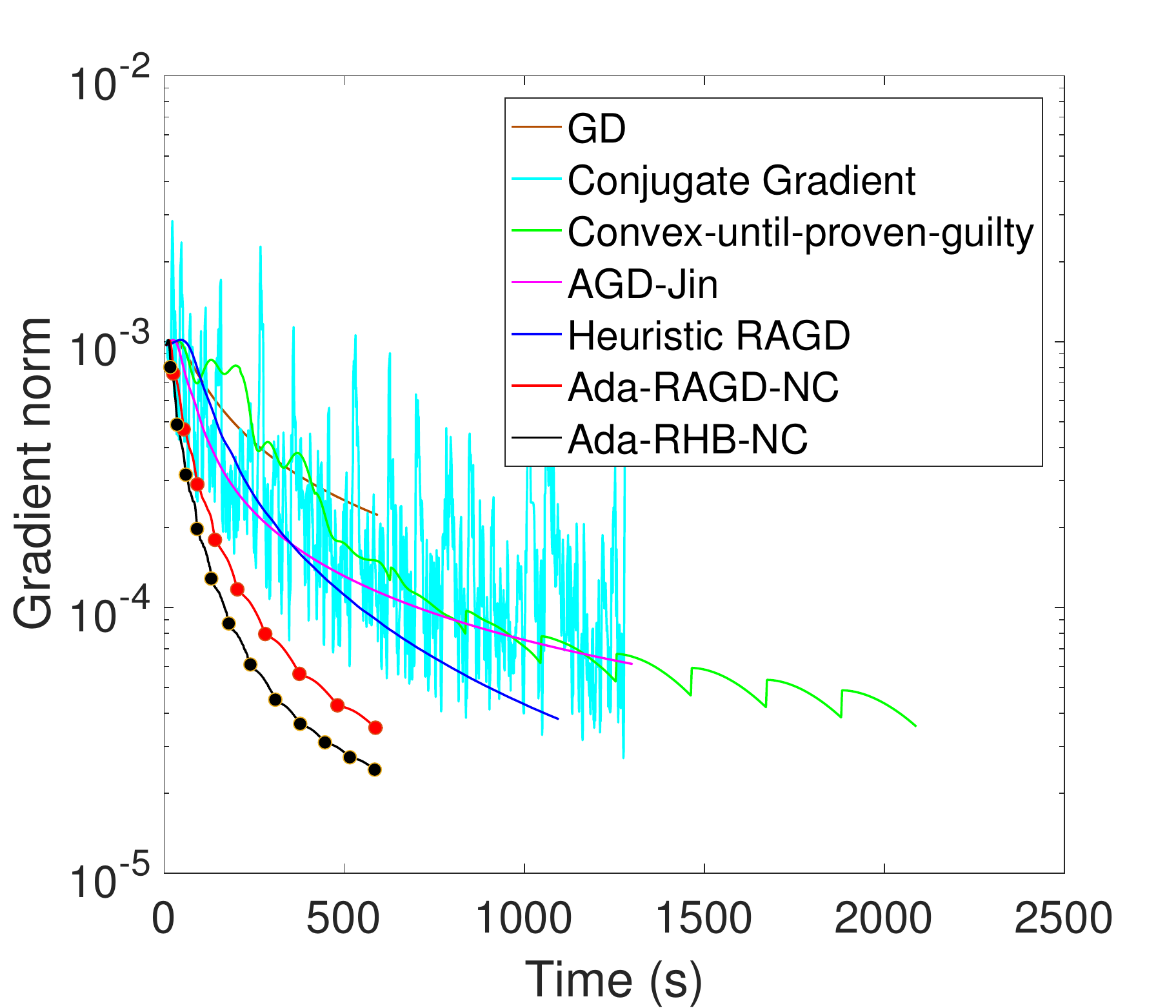}
&\includegraphics[width=0.33\textwidth,keepaspectratio]{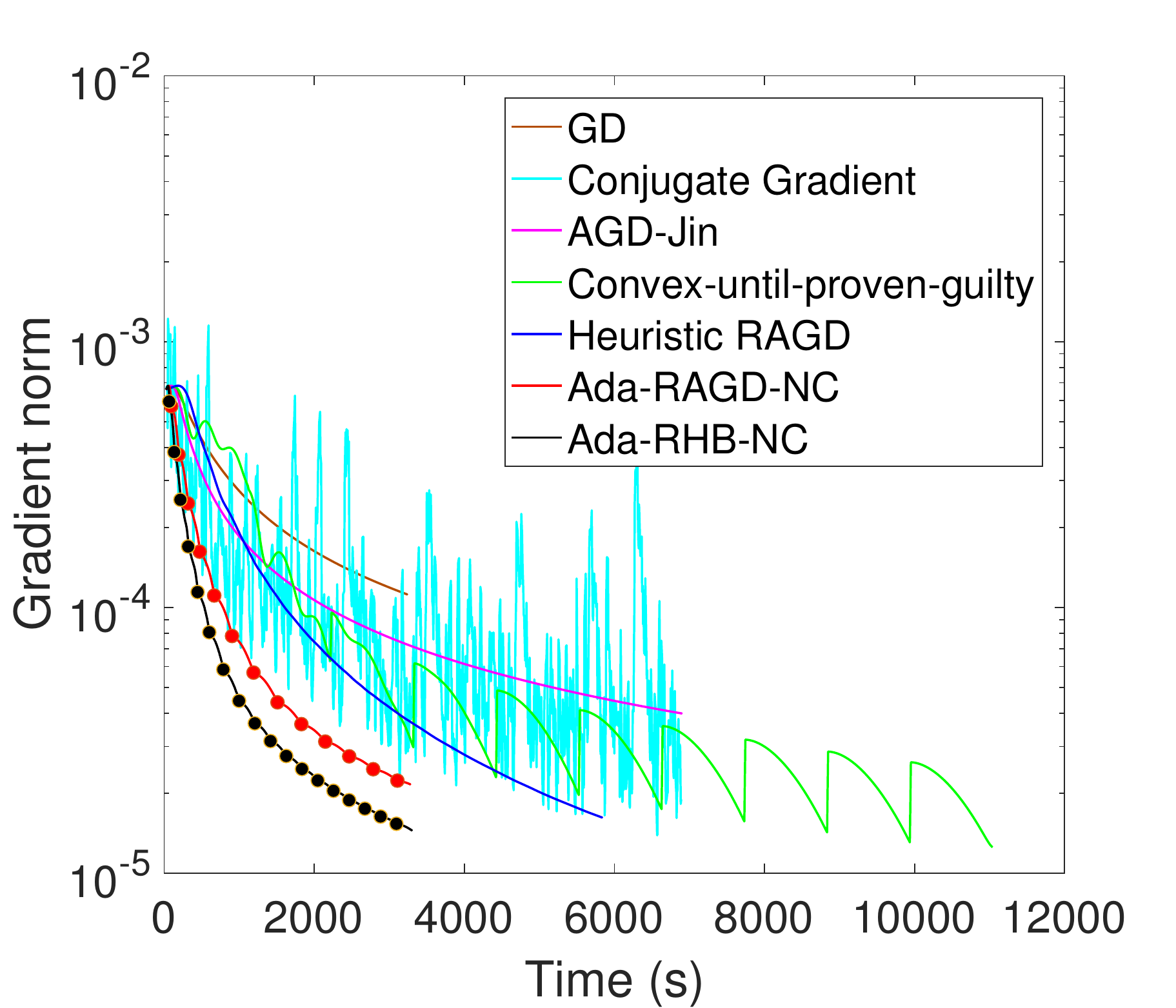}\\
\includegraphics[width=0.33\textwidth,keepaspectratio]{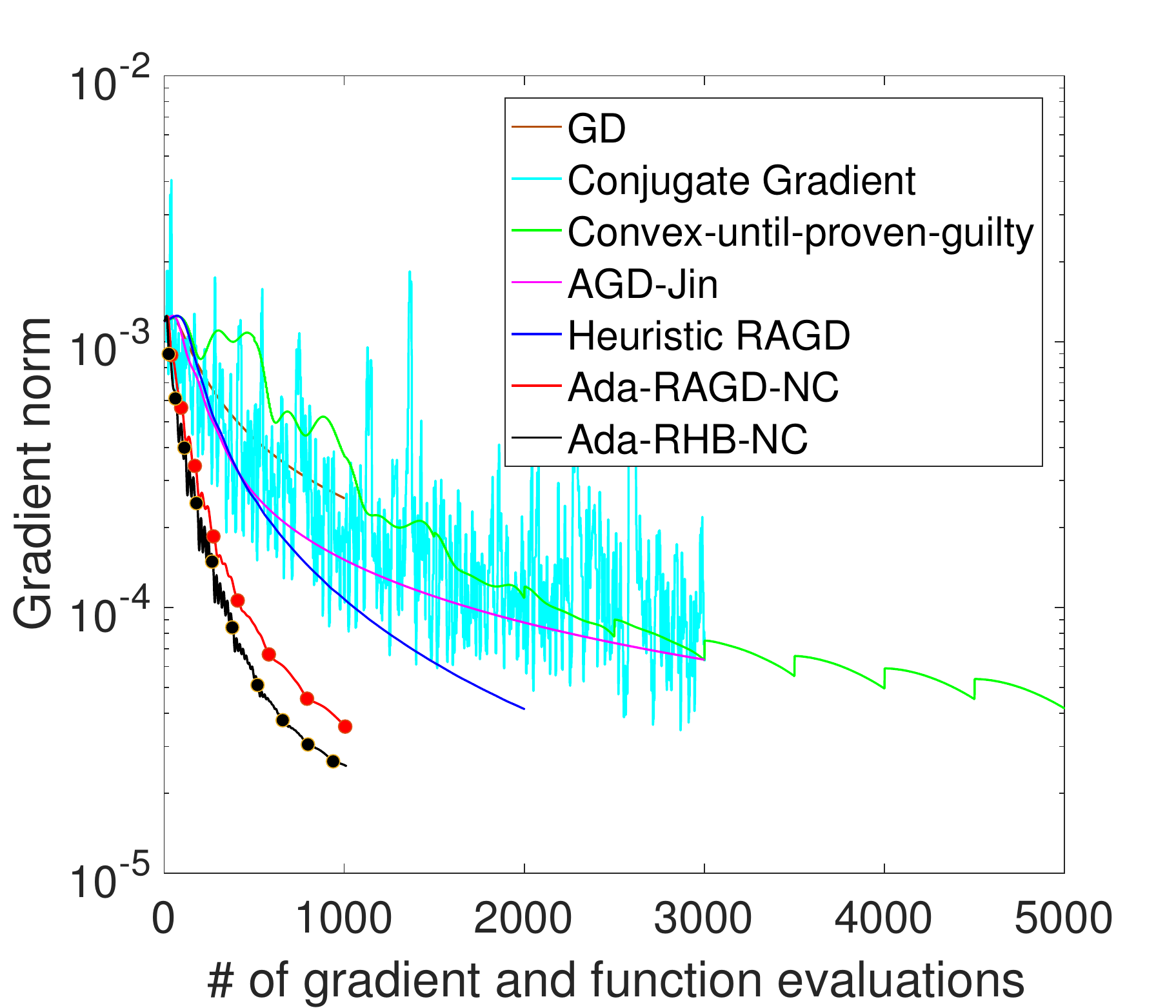}
&\includegraphics[width=0.33\textwidth,keepaspectratio]{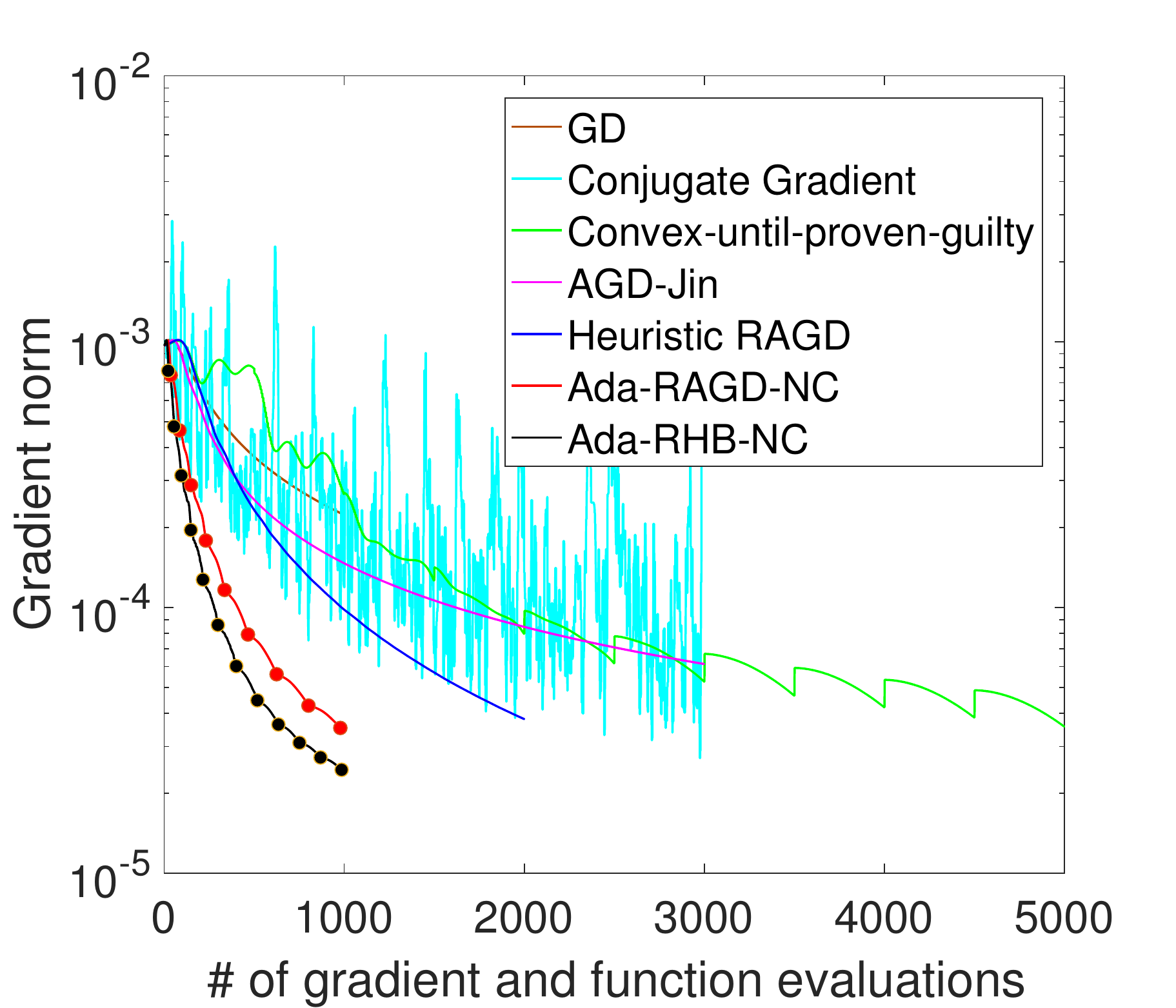}
&\includegraphics[width=0.33\textwidth,keepaspectratio]{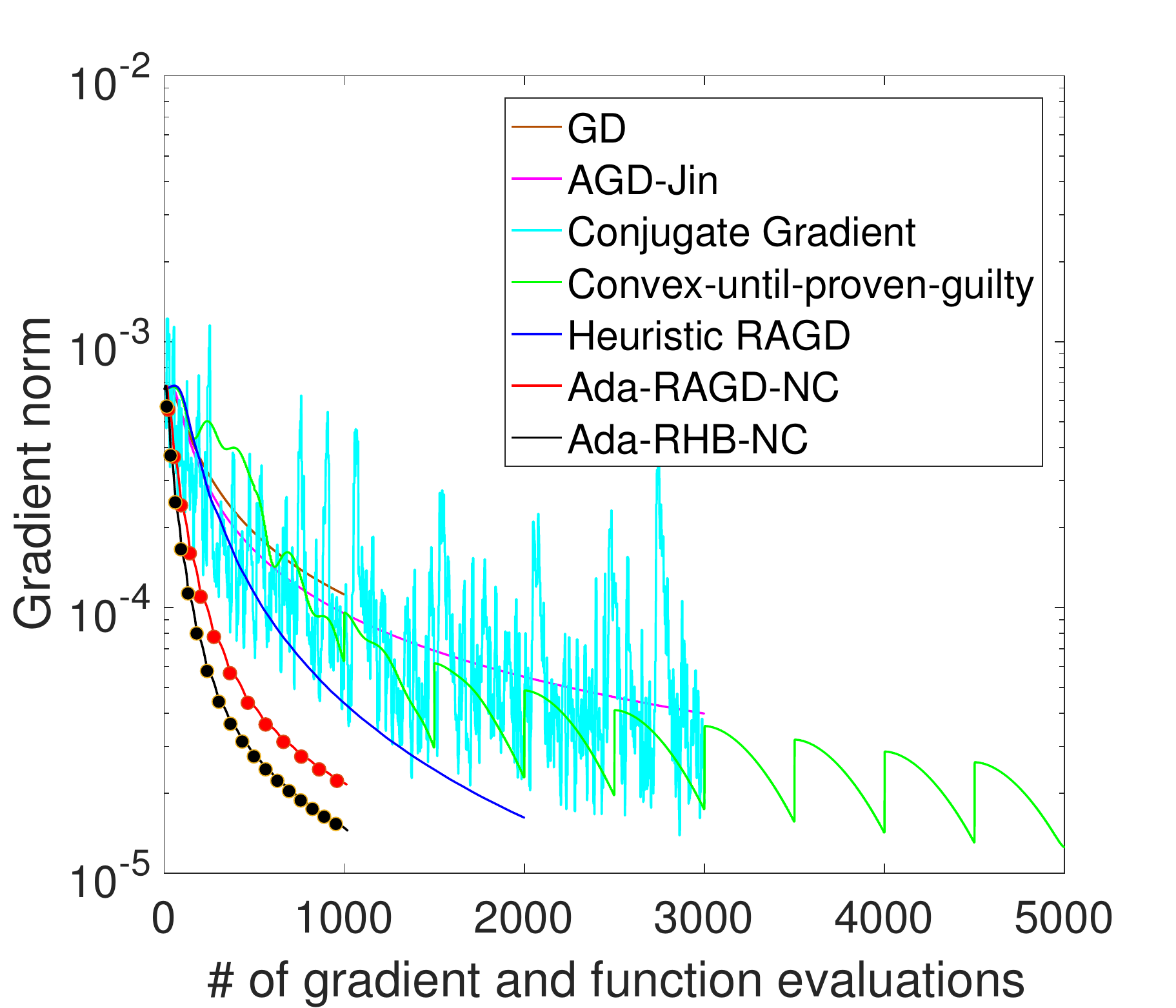}\\
(a) MovieLens-10M & (b) MovieLens-20M & (c) Netflix
\end{tabular}
\caption{Comparisons on the 1 bit matrix completion problem. The first and second row: function value. The third and forth row: gradient norm. The first and third row: use time as the horizontal axis. The second and forth row: use the number of function and gradient evaluations as the horizontal axis. Circles in Ada-RAGD-NC and Ada-RHB-NC indicate where restart occurs.}\label{fig2}
\end{figure}

\subsection{Gap Between Theory and Practice}

In the previous two sections, we only run Ada-RAGD-NC and Ada-RHB-NC for 1000 iterations such that the objective error and gradient norm are reduced to low level quickly, which is sufficient for practical machine learning applications. In this section, we verity what happens when we run the two methods for a longer time and discuss the gap between theory and practice for nonadaptive RAGD-NC and RHB-NC (Algorithms \ref{AGD1} and \ref{HB1}).

We only report the observations on the Movielens-10M data set, and the results are similar on the other two. For both the matrix completion and one bit matrix completion problems, we run Ada-RAGD-NC and Ada-RHB-NC for $10^5$ iterations and use the minimum function value to approximate the optimal one. We set the same parameters as those in Section \ref{exp_sec1}. Figure \ref{fig3} plots the results. We have the following observations and conclusions.

\begin{figure}[p]
\hspace*{-1.3cm}\begin{tabular}{@{\extracolsep{0.001em}}c@{\extracolsep{0.001em}}c@{\extracolsep{0.001em}}c}
\includegraphics[width=1.17\textwidth,keepaspectratio]{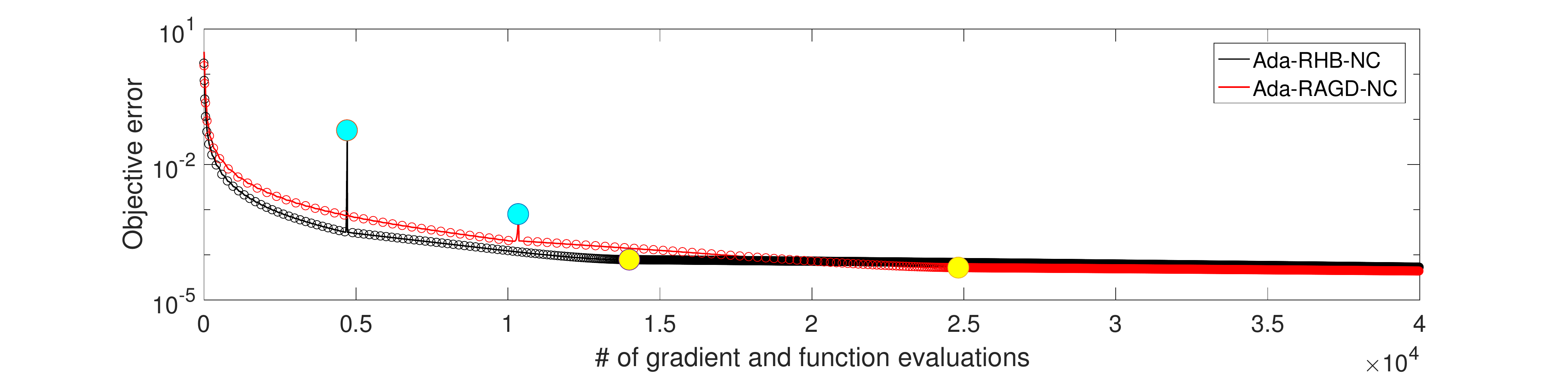}\\
\includegraphics[width=1.17\textwidth,keepaspectratio]{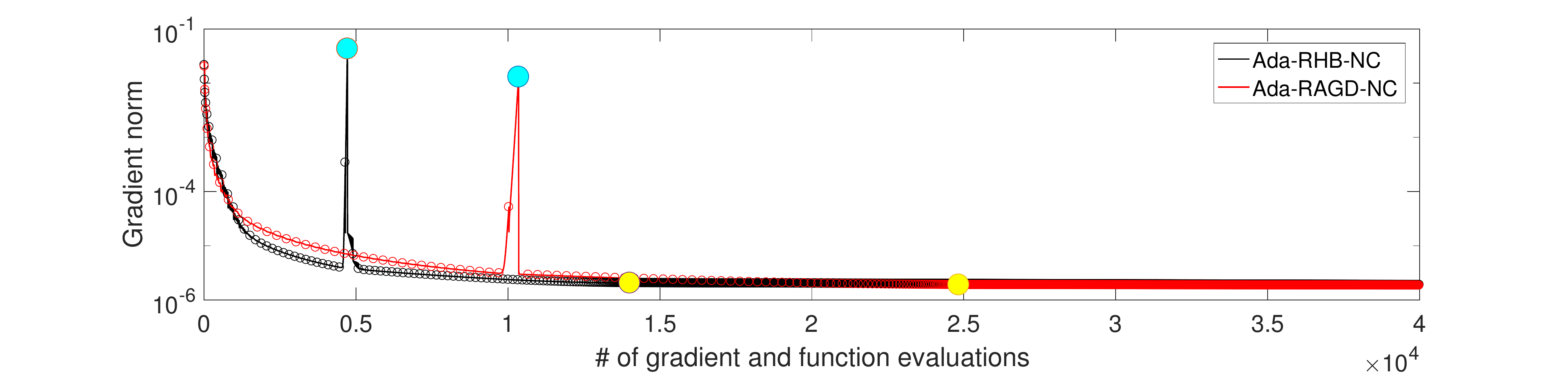}\\
Matrix completion problem\\
\includegraphics[width=1.17\textwidth,keepaspectratio]{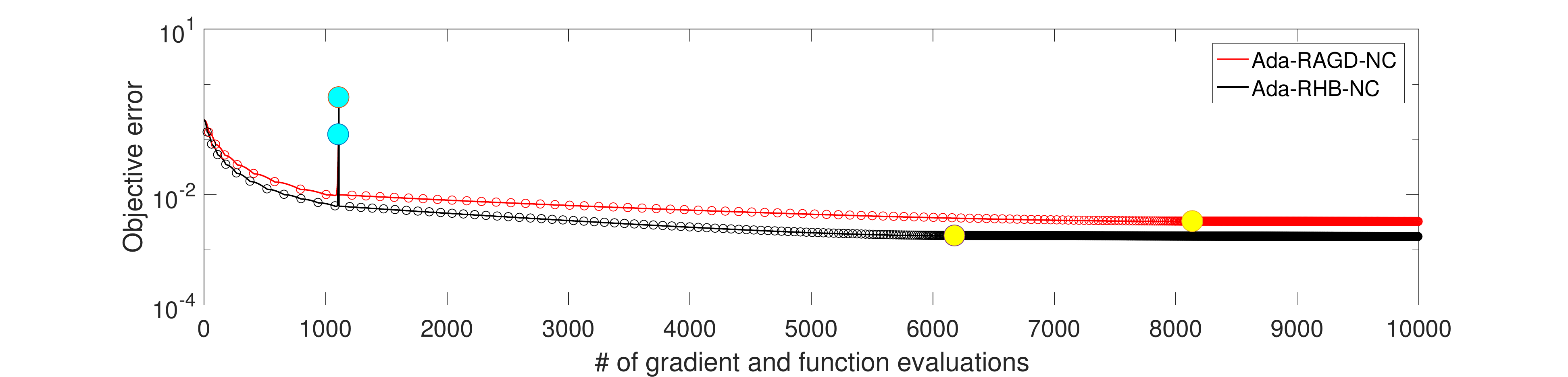}\\
\includegraphics[width=1.17\textwidth,keepaspectratio]{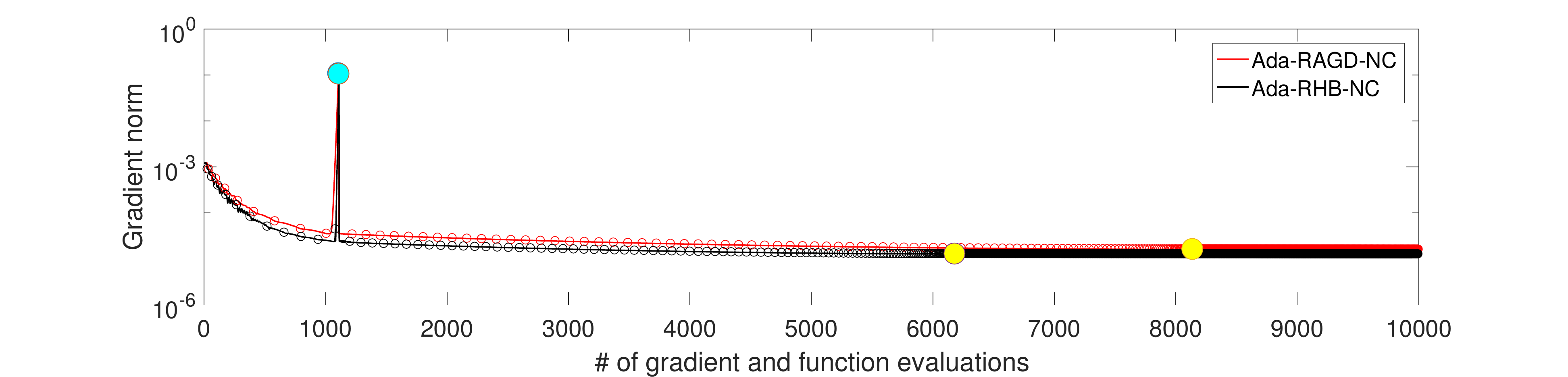}\\
1 bit matrix completion problem
\end{tabular}
\caption{Comparisons of objective error and gradient norm on the Movielens-10M dataset. Top two: matrix completion. Bottom two: 1 bit matrix completion. Small hollow circles indicate where restart occurs. Blue-green circles indicate where line 11 in Algorithms \ref{AGD1p} and \ref{HB1p} is invoked. Yellow circles indicate where $B_0$ decreases to be smaller than $B$ for the first time.}\label{fig3}
\end{figure}

\begin{enumerate}
\item We see that line 11 (in fact, step (\ref{agdp-cont1})) in Algorithms \ref{AGD1p} and \ref{HB1p} is invoked only once for both two methods (marked by the blue-green circle), at which time the objective error and gradient norm increase substantially. Then we decrease $B_0$ and $\eta$ and increase the estimated $\rho$ properly. After the adaptive update, line 11 is never invoked.
\item When $B_0$ decreases to be smaller than $B$ (marked by the yellow circle), we see that both methods restart frequently. Specifically, we observe that for the matrix completion problem, Ada-RAGD-NC restarts every 16 iterations after $B_0\leq B$ while Ada-RHB-NC restarts every 10 iterations. For the one bit matrix completion problem, Ada-RAGD-NC restarts every 4 iterations after $B_0\leq B$ while Ada-RHB-NC restarts every 3 iterations. It seems to take an extremely long time to break the while loop (that is, no restart occurs in $K$ iterations), especially for high dimensional problems ($\|\x^{t+1}-\x^t\|$ is not likely to be small even if $|\x_i^{t+1}-\x_i^t|$ is small for each $i=1,2\cdots,d$). Thus, we suggest to stop the algorithm in practice when the gradient norm is smaller than a threshold or the number of iterations exceeds the maximum one.
\item Note that Ada-RAGD-NC and Ada-RHB-NC reduce to their nonadaptive counterparts (Algorithms \ref{AGD1} and \ref{HB1}) when $B_0\leq B$, and the plots after the yellow circles may illustrate the practical performance of the nonadaptive methods. Thus, nonadaptive RAGD-NC and RHB-NC (Algorithms \ref{AGD1} and \ref{HB1}) are only for the theoretical purpose and we do not suggest to use them in practice due to their frequent restart, unless we do not follow the theory to set the parameters, especially the parameter $B$.
\end{enumerate}

\section{Conclusion}

This paper proposes two simple accelerated gradient methods, restarted AGD and restarted HB, for general nonconvex problems with Lipschitz continuous gradient and Hessian. Our simple methods find an $\epsilon$-approximate first-order stationary point within $\bO(\epsilon^{-7/4})$ gradient evaluations, which improves over the best known complexity by the $\bO(\log\frac{1}{\epsilon})$ factor. Our proofs only use elementary analysis. We hope our analysis may lead to a better understanding of the acceleration mechanism for nonconvex optimization.

\nocite{langley00}

\bibliography{22-0522}
\bibliographystyle{icml2022}

\end{document}